\documentclass{amsart}

\usepackage{enumerate,stmaryrd}
\usepackage{amsmath,amsfonts,amssymb,amsthm,tikz-cd,color}

\newcommand\mcal\mathcal
\newcommand\mbf\mathbf
\newcommand\mrm\mathrm
\newcommand\mfk\mathfrak
\newcommand\msf\mathsf

\renewcommand{\bar}{\overline}
\renewcommand{\tilde}{\widetilde}
\renewcommand{\hat}{\widehat}

\def\N{\mathbb N}
\def\Z{\mathbb Z}
\def\Q{\mathbb Q}
\def\R{\mathbb R}
\def\C{\mathbb C}

\def\A{\mathbb A}
\def\P{\mathbb P}

\def\G{\mathbb G}

\def\Spec{\mathrm{Spec}}

\renewcommand\O{\mathcal O}
\def\Div{\mathrm{Div}}
\def\Pic{\mathrm{Pic}}

\def\Ab{\mathbf{Ab}}

\def\ob{\mathrm{ob}}

\def\Hom{\mathrm{Hom}}

\def\Ext{\mathrm{Ext}}

\def\coker{\mathrm{coker}}

\DeclareMathOperator\liminv{\underset\longleftarrow\lim}
\DeclareMathOperator\limdir{\underset\longrightarrow\lim}

\def\op{\mathrm{op}}
\def\cts{\mathrm{cts}}
\def\tors{\mathrm{tors}}
\def\tf{\mathrm{tf}}

\def\ab{\mathrm{ab}}

\def\isoarrow{\overset\sim\rightarrow}

\def\leftisoarrow{\overset\sim\leftarrow}

\renewcommand{\H}{\mathrm H}

\def\gr{\mathrm{gr}}
\def\rk{\mathrm{rk}}

\numberwithin{equation}{subsection}
\theoremstyle{plain}
\newcounter{theoremcount}[subsection]

\newtheorem{theorem}[theoremcount]{Theorem}
\newtheorem*{theorem*}{Theorem}

\newtheorem{lemma}[theoremcount]{Lemma}
\newtheorem{proposition}[theoremcount]{Proposition}
\newtheorem{corollary}[theoremcount]{Corollary}

\theoremstyle{definition}
\newtheorem{definition}[theoremcount]{Definition}
\newtheorem{example}[theoremcount]{Example}
\newtheorem{remark}[theoremcount]{Remark}
\newtheorem{notation}[theoremcount]{Notation}

\newtheorem{definition-lemma}[theoremcount]{Definition-Lemma}
\newtheorem{construction}[theoremcount]{Construction}

\theoremstyle{theorem}

\newcommand\llbrack\llbracket
\newcommand\rrbrack\rrbracket

\def\topp{\mathrm{top}}

\def\EZ{\nabla}
\def\CoCh{\mathrm{CoCh}}
\def\Sym{\mathrm{Sym}}
\def\Cycle{\mathrm Z}
\def\acts{\curvearrowright}
\def\actsarrow{\overset\curvearrowright\rightarrow}
\def\cent{\mathsf{z}}
\def\centarrow{\overset{\mathsf{z}}\rightarrow}

\def\proj{\mathrm{proj}}
\def\id{\mathrm{id}}
\def\Kan{\Gamma}

\def\Tot{\mathrm{Tot}}

\def\B{\mathsf B}

\def\D{\mathsf D}

\renewcommand\O{\mathcal O}

\def\uu{\mathfrak u}

\def\Lie{\mathrm{Lie}}
\def\dR{\mathrm{dR}}
\def\et{\mathrm{\acute et}}

\def\cris{\mathrm{cris}}

\def\st{\mathrm{st}}

\def\nr{\mathrm{nr}}
\def\Mod{\mathbf{Mod}}

\def\cts{\mathrm{cts}}
\def\Map{\mathrm{Map}}

\def\MHS{\mathbf{MHS}}
\def\MF{\mathbf{MF}}
\def\wa{\mathrm{w.a.}}
\def\Ad{\mathrm{Ad}}

\def\dee{\mathrm d}

\def\1{\mathbf 1}
\def\prim{\mathrm{prim}}
\def\gplike{\mathrm{gplike}}

\def\Alb{\mathrm{Alb}}
\def\codiag{\mathrm{codiag}}
\renewcommand\Re{\mathrm{Re}}

\def\dual{*}
\def\reddual{\circ}
\def\mideal{\mathfrak m}
\renewcommand{\div}{\mathrm{div}}

\input cyracc.def       

\title[Anabelian Local Heights]{The motivic anabelian geometry of local heights on abelian varieties}
\author{L.\ Alexander Betts}

\begin{document}

\setcounter{secnumdepth}{3}
\setcounter{tocdepth}{1}

\begin{abstract}
We study the problem of describing local components of N\'eron--Tate height functions on abelian varieties over characteristic $0$ local fields as functions on spaces of torsors under various realisations of the $2$-step unipotent motivic fundamental group of the $\mathbb G_m$-torsor corresponding to the defining line bundle. To this end, we present three main theorems giving such a description in terms of the $\mathbb Q_\ell$- and $\mathbb Q_p$-pro-unipotent \'etale realisations when the base field is $p$-adic, and in terms of the $\mathbb R$-pro-unipotent Betti--de Rham realisation when the base field is archimedean.

In the course of proving the $p$-adic instance of these theorems, we develop a new technique for studying local non-abelian Bloch--Kato Selmer sets, working with certain explicit cosimplicial group models for these sets and using methods from homotopical algebra. Among other uses, these models enable us to construct a pro-unipotent generalisation of the Bloch--Kato exponential sequence under minimal assumptions.

\end{abstract}
\maketitle

\tableofcontents

\newpage
\section{Introduction}
\label{c:intro}
\subsection{Anabelian geometry and local heights}
\subsubsection{Background}

In his letter to Faltings \cite{grothendieck_letter}, Grothendieck proposed a programme of studying the arithmetic and Diophantine geometry of hyperbolic curves $Y$ over characteristic $0$ fields $F$ by means of their profinite \'etale fundamental groups $\pi_1^\et(Y_{\bar F})$ along with their outer Galois actions. When one uses an $F$-rational point $y\in Y(F)$ as the basepoint for the fundamental group, the outer Galois action canonically lifts to a genuine action, and the link with Diophantine geometry comes concretely via a certain \emph{non-abelian Kummer map}\[Y(F)\rightarrow\H^1(G_F,\pi_1^\et(Y_{\bar F};y)),\]taking values in the non-abelian Galois cohomology set $\H^1(G_F,\pi_1^\et(Y_{\bar F};y))$ classifying profinite torsors under $\pi_1^\et(Y_{\bar F};y)$ with compatible Galois action; the non-abelian Kummer map sends an $F$-rational point $z\in Y(F)$ to the class of the profinite \'etale torsor of paths $\pi_1^\et(Y_{\bar F};y,z)$.

In the case that $F$ is a number field, Grothendieck's \emph{section conjecture} posits that the link between the fundamental group and the Diophantine geometry of $Y$ provided by the non-abelian Kummer map should be very strong indeed: when $Y$ is proper, the map should be bijective, and for general hyperbolic $Y$ it should be injective (this much was already proved in \cite{grothendieck_letter}), with the only extraneous elements of $\H^1(G_F,\pi_1^\et(Y_{\bar F};y))$ coming from the cusps \cite{stix}. It was hoped that a proof of the section conjecture would enable one to find an alternative proof of the Mordell conjecture, which had been earlier proven by Faltings \cite{faltings}.

In practice, the highly non-abelian nature of $\pi_1^\et(Y_{\bar F};y)$ makes studying it particularly difficult, and it is convenient to replace it with an appropriately linearised variant which is much closer to being abelian. Such a variant is provided by the $\Q_p$-pro-unipotent \'etale fundamental group $U$ of $Y$ based at $y$, which can be defined equivalently as the Tannaka group of the category of \'etale $\Q_p$-local systems on $Y_{\bar F}$, or as the continuous $\Q_p$-Mal\u cev completion of $\pi_1^\et(Y_{\bar F};y)$. Although one cannot hope that the non-abelian Kummer map\[Y(F)\rightarrow\H^1(G_F,U(\Q_p))\]will be an isomorphism except in very specific cases, this map is nonetheless very useful in the study of the arithmetic and Diophantine geometry of $Y$.

For instance, when $F=\Q$ and $Y$ is projective with good reduction at $p$, Kim's non-abelian Chabauty method \cite{minhyong:selmer} aims to constrain the locus of $Y(\Q)\subseteq Y(\Q_p)$ by finding functions on the local non-abelian cohomology $\H^1(G_{\Q_p},U_n(\Q_p))$ -- or more accurately on a subset $\H^1_f(G_{\Q_p},U_n(\Q_p))$ cut out by local Selmer conditions -- which vanish on the image of the global cohomology. Here $U_n$ denotes the maximal $n$-step unipotent quotient of $U$, which is finite dimensional. When $n>>0$, any one of various well-known conjectures (Bloch--Kato, Fontaine--Mazur, Janssen) implies the existence of such functions, providing a conditional anabelian proof of the Siegel--Faltings theorem on the finiteness of the set of $S$-integral points of $Y$. This proof can be even made unconditional in many cases: when $Y=\P^1_\Q\setminus\{0,1,\infty\}$ \cite{minhyong:siegel}; when $Y$ is a once-punctured CM elliptic curve \cite{minhyong:cm}; when $Y$ is a complete hyperbolic curve whose Jacobian has CM \cite{minhyong-coates}; or when $Y$ is any complete hyperbolic curve which is a solvable cover of $\P^1_\Q$ \cite{ellenberg-hast}.

At the other end of the spectrum, an explicit study of $\H^1_f(G_{\Q_p},U_n(\Q_p))$ for $n$ small leads to explicit methods for constraining the locus of $Y(\Q)\subseteq Y(\Q_p)$, and in practice often completely computing it. For instance, when $n=1$ and $Y$ is proper, the Chabauty--Kim method recovers Coleman's explicit Chabauty method \cite{coleman}. The fundamental group $U_1(\Q_p)=V_pJ$ is the $\Q_p$-linear Tate module of the Jacobian $J$ of $Y$, $\H^1_f(G_{\Q_p},V_pJ)\cong\Lie(J_{\Q_p})$ is the Lie algebra of $J$, and the linear functions on $\Lie(J_{\Q_p})$ produced by Chabauty--Coleman agree with those produced by Chabauty--Kim.

When $n=2$, the constraints on $Y(\Q)\subseteq Y(\Q_p)$ produced by non-abelian Chabauty involve finer arithmetic invariants. A detailed study of this case was made in \cite{balakrishnan-dogra_1,balakrishnan-dogra_2}; among other results, it is shown in \cite{balakrishnan-dogra_1} that when $\rk(J(\Q))=\dim_\Q(J)$, the constraints produced by non-abelian Chabauty involve not only quadratic functions on $\Lie(J_{\Q_p})$, but also the local component of the $p$-adic height function in the sense of Nekov\'a\u r \cite{nekovar}. This raises the rather tantalising prospect that the constraints produced by non-abelian Chabauty for $n>2$ should involve quantities that play the role of higher $p$-adic height functions.

\subsubsection{Local heights as functions on moduli spaces of torsors}

Another example of the link between $2$-step unipotent fundamental groups and local heights appears in a paper of Balakrishnan, Dan-Cohen, Kim and Wewers.

\begin{theorem}\cite{minhyong-etal:bsd_conjecture}\label{thm:prototype_version}
Let $E/K$ be an elliptic curve over a $p$-adic\footnote{We have interchanged the roles of $p$ and $\ell$ from \cite{minhyong-etal:bsd_conjecture} -- for us $p$ will always be the residue characteristic of the base field, and $\ell$ the residue characteristic of the field of coefficients for the fundamental groups involved.} number field, and $U_2/\Q_\ell$ for $\ell\neq p$ the $2$-step unipotent fundamental group of the punctured curve $E^\circ=E\setminus\{0\}$ based at a rational tangent vector at the origin. Then the map $\Q_\ell(1)\hookrightarrow U_2$ induced from the inclusion of the cusp induces a bijection\[\H^1(G_K,\Q_\ell(1))\isoarrow\H^1(G_K,U_2(\Q_\ell)),\]and the composite map\[E^\circ(K)\rightarrow\H^1(G_K,U_2(\Q_\ell))\leftisoarrow\H^1(G_K,\Q_\ell(1))\isoarrow\Q_\ell\]takes values in $\Q\subseteq\Q_\ell$, and is a N\'eron function on $E(K)$ with divisor $2[0]$.
\end{theorem}

In this theorem we see a first instance of what will be the main theme of this paper: \emph{that local components of height functions admit particularly natural interpretations as functions on spaces of torsors under $2$-step unipotent fundamental groups}. This theme is best viewed in the context of Deligne's motivic philosophy for the pro-unipotent fundamental group \cite{deligne}, which stresses the importance of studying not only the \'etale fundamental group of a variety, but also a plethora of other kinds of highly structured fundamental groups (Betti, de Rham, log-crystalline\footnote{This is a mild anachronism, since \cite{deligne} was written when the development of logarithmic geometry \cite{kato_log-structures} was in its infancy, and in particular predates the development of both Hyodo--Kato cohomology \cite{hyodo-kato} and log-crystalline fundamental groups \cite{shiho_1}. Thus, \cite{deligne} only concerns itself with the crystalline fundamental group at places of good reduction, where a good theory already existed.}) related by various comparison isomorphisms. Thus, in this paper we will not only be interested in relating local components of heights at non-archimedean places to \'etale and log-crystalline fundamental groups, but also in relating local components at archimedean places to Betti--de Rham fundamental groups, endowed with their mixed Hodge structure constructed by Hain \cite{hain_MHS} using Chen's non-abelian de Rham theorem. Thus, our implicit emphasis here lies in studying the local components of $\R$-valued heights, as opposed to the $\Q_p$-valued heights studied in \cite{balakrishnan-dogra_1,balakrishnan-dogra_2}.

\smallskip

To make the link between $2$-step unipotent fundamental groups and local heights as clear as possible, let us shift focus slightly and consider local canonical heights on an abelian variety $A$ over a characteristic $0$ local field $K$ induced by a line bundle $L/A$. There is a natural anabelian invariant associated to this setup, namely the pro-unipotent fundamental group $U$ of the complement $L^\times$ of the zero section in $L$; this pro-unipotent fundamental group is in fact already $2$-step unipotent, since by the homotopy exact sequence of a fibration it is a central extension of the fundamental group of $A$ by the fundamental group of $\G_m$.

In this setup, we shall prove three theorems (\ref{thm:main_theorem_l-adic}, \ref{thm:main_theorem_p-adic} and \ref{thm:main_theorem_archimedean}) concretely relating the fundamental group $U$ to local heights, of the following kinds:
\begin{itemize}
	\item when $K$ is $p$-adic, the local height can be realised as a function on the local non-abelian Galois cohomology set $\H^1(G_K,U(\Q_\ell))$, where $U$ is the $\Q_\ell$-unipotent \'etale fundamental group of $L^\times$ for some $\ell\neq p$;
	\item when $K$ is $p$-adic, the local height can be realised as a function on the local Bloch--Kato Selmer set $\H^1_g(G_K,U(\Q_p))$, where $U$ is the $\Q_p$-unipotent \'etale fundamental group of $L^\times$; and
	\item when $K$ is archimedean, the local height can be realised as a function on the set $\H^1(\MHS_\R,U(\R))$ classifying torsors under $U/\R$ with mixed Hodge structure, where $U$ is the $\R$-unipotent Betti fundamental group of $L^\times$.
\end{itemize}
Let us remark also that in the first and last of these theorems, the space of torsors will be one-dimensional (in an appropriate sense), so that the local height is in fact the \emph{only} information seen by the non-abelian Kummer map. The second theorem is more subtle, as the local Bloch--Kato Selmer set $\H^1_g(G_K,U(\Q_p))$ contains plenty of analytic information, and a significant part of formulating the appropriate $p$-adic instance of these theorems lies in finding the right way of quotienting out the analytic part of the cohomology so that the quotient is one-dimensional and describes exactly the local height. Since developing the appropriate machinery for dealing with the $p$-adic case is the major part of this paper, we will for the time being just discuss the $\ell$-adic and archimedean cases, returning to discuss the $p$-adic case in more detail later in this introduction.

\begin{remark}
The problem of giving a motivic interpretation of local heights was also studied in \cite{scholl}, where the emphasis was on giving a description of height pairings on cycles of complementary codimension in terms of motivic cohomology. Of course, a motivic description of the height pairing $\langle-,-\rangle\colon A(F)\otimes A^\vee(F)\rightarrow\R$ on an abelian variety $A$ over a number field $F$ gives rise to a description of the canonical height function induced by a line bundle $L/A$ in terms of the usual map $\phi_L\colon A\rightarrow A^\vee$ via the formula \cite[Proposition 9.3.6 \& Corollary 9.3.7]{bombieri-gubler}\[\hat h_L(z)=\langle z,\phi_L(z)\rangle+\frac12\langle z,L\otimes[-1]^*L^{-1}\rangle.\]It is interesting that both Scholl's approach and ours require consideration of mildly non-linear invariants associated to the line bundle: in our case the fundamental group of $L^\times$; and in Scholl's the pairing on motivic cohomology $h^1(A)$ induced by $\phi_L$. Indeed, this pairing should be none other than the commutator pairing in the motivic fundamental group of $L^\times$, which should be a central extension of $h_1(A)=h^1(A)(1)$ by~$\Q(1)$.
\end{remark}

\begin{remark}
In the particular case that the line bundle $L$ is algebraically equivalent to the trivial bundle, the corresponding $\G_m$-torsor $L^\times$ canonically has the structure of a semiabelian variety with $\tilde0$ as its identity, in which case the fundamental group of $L^\times$ is abelian (e.g.\ the \'etale fundamental group is the Tate module), and many of the arguments we will present in this article reduce to more well-known abelian arguments. The chief advantage of our non-abelian techniques, then, is that it gives us a natural framework for studying line bundles that are not algebraically trivial, for instance ample or anti-ample bundles.
\end{remark}

\subsubsection{N\'eron log-metrics}
\label{sss:neron_log-metrics}

Let us now explain the precise notion of local heights that we will be using in this paper, leading up to a precise statement of the $\ell$-adic and archimedean main theorems. Since we will be working in the language of line bundles rather than divisors, it will be convenient for us to follow the approach of Bombieri and Gubler and view local heights as local metrics on line bundles \cite[Section 2.7]{bombieri-gubler} rather than as Weil functions on complements of divisors \cite[Section 10.2]{lang}. To briefly recall what is meant by this, let us fix a characteristic $0$ local field $K$ with absolute value $|\cdot|$, and write $v=-\log|\cdot|$; we assume that $v$ is normalised to be $\Q$-valued if $K$ is non-archimedean. If $L$ is a line bundle on a smooth connected variety $Y/K$, a \emph{continuous log-metric} on $L$ is a continuous function $\lambda_L\colon L^\times(\bar K)\rightarrow\R$ which satisfies the relation\[\lambda_L(\alpha z)=\lambda_L(z)+v(\alpha)\]for all $z\in L^\times(\bar K)$ and $\alpha\in {\bar K}^\times$; here $L^\times=L\setminus0$ denotes the complement of the zero section in $L$. Such a $\lambda_L$ is then the logarithm of a metric in the sense of \cite[Definition 2.7.1]{bombieri-gubler}. A continuous log-metric always exists by \cite[Proposition 2.7.5]{bombieri-gubler}, and if $s$ is a non-zero section of $L$ then $\lambda_L\circ s$ is a Weil function on $Y$ associated to the Cartier divisor $\div(s)$ \cite[Section 10.2]{lang}.

In the particular case that $Y=A$ is an abelian variety, there is a canonical choice of continuous log-metric $\lambda_L$ on $L$ up to additive constants, namely the log-metric such that for any (equivalently just one) non-zero section $s$, $\lambda_L\circ s$ is a N\'eron function on $A$ associated to the Cartier divisor $\div(s)$ \cite[Section 11.1]{lang}. We will call such a log-metric a \emph{N\'eron log-metric} on $L$, and if we are provided with a point $\tilde0\in L^\times(\bar K)$, we will unambiguously refer to \emph{the} N\'eron log-metric as being the unique one such that $\lambda_L(\tilde0)=0$.

\begin{remark}
N\'eron log-metrics are local components of heights in the following sense: if $A$ is an abelian variety over a number field $F$ and $L/A$ is a line bundle with a distinguished point $\tilde0\in L^\times(\bar F)$ in the fibre over $0\in A(\bar F)$, then the function $L^\times(\bar F)\rightarrow\R$ given by an appropriately normalised sum of the N\'eron log-metrics at all places of $\bar F$ (is well-defined and) on the fibre over some $z\in A(\bar F)$ is constant, equal to the canonical height $\hat h_L(z)$ (by the corresponding result for N\'eron functions \cite[Theorem 11.1.6]{lang}).
\end{remark}

\begin{remark}
In the spirit of Mazur and Tate \cite{mazur-tate:circle_pairing}, this description of the global height admits a natural refinement of a more Arakelov-theoretic nature when restricted to $F$-rational points\footnote{We assume here that $\tilde0\in L^\times(F)$ is chosen $F$-rational.}. Specifically, the N\'eron log-metrics at finite places are all $\Q$-valued, so the collection of all local metrics defines naturally a global function\[\lambda_L\colon L^\times(F)\rightarrow\hat\Div_\Q(\O_F):=\bigoplus_{v\nmid\infty}\Q\oplus\bigoplus_{v\mid\infty}\R\]valued in a $\Q$-linearised analogue of the Arakelov divisor group of $\O_F$. By again evaluating this function on lifts of elements of $A(F)$, we thereby obtain a function\[h_L\colon A(F)\rightarrow\hat\Pic_\Q(\O_F):=\hat\Div_\Q(\O_F)/\div(F^\times)\]valued in the corresponding analogue of the Arakelov Picard group, through which the canonical height factors via the natural degree map $\deg\colon\hat\Pic_\Q(\O_F)\twoheadrightarrow\R$.

It turns out (though we shall not prove this in this paper) that this lift of the global height is a simultaneous quadratic refinement of Grothendieck's monodromy pairing on component groups \cite[Th\'eor\`eme 7.1(b)]{SGA7.1}, and of Mazur and Tate's circle pairing \cite[Section 3.5.2]{mazur-tate:circle_pairing}. More precisely, $h_L$ is always a quadratic function\footnote{That is, satisfies the relation $h_L(x+y+z)-h_L(x+y)-h_L(y+z)-h_L(z+x)+h_L(x)+h_L(y)+h_L(z)-h_L(0)=0$}, and in the particular case that $L^\times$ underlies a biextension $E$ of two abelian varieties $A,B$ by $\G_m$, the map\[h_L\colon A(F)\times B(F)\rightarrow\hat\Pic_\Q(\O_F)\]is a bilinear pairing which, on the one hand, when restricted to $A(F)\times B(F)^0$ takes values in the usual Arakelov Picard group $\hat\Pic(\O_F)$ and agrees with the Mazur--Tate circle pairing associated to $E$ \cite[Section 3.5.2]{mazur-tate:circle_pairing}\footnote{The presentation in \cite{mazur-tate:circle_pairing} assumes that $A$ and $B$ have Tamagawa number $1$ at all finite places---the setup in \cite[Section 2]{mazur-tate:bsd} covers how to define this pairing in the general case. The assumption in the latter paper that $A$ and $B$ are dual and that $E$ is the Poincar\'e biextension is unnecessary.}, and on the other hand, when composed with the natural map $\hat\Pic_\Q(\O_F)\rightarrow\bigoplus_{v\nmid\infty}\Q/\Z$ agrees with the composite\[A(F)\times B(F)\rightarrow\bigoplus_{v\nmid\infty}\Phi_{A,v}(\kappa_v)\times\Phi_{B,v}(\kappa_v)\rightarrow\bigoplus_{v\nmid\infty}\Q/\Z\]where the right-hand arrow is the sum of the local Grothendieck monodromy pairings. Here, $\Phi_{A,v}/\kappa_v$ denotes the group-scheme of components of the fibre of the N\'eron model of $A$ over the residue field $\kappa_v$ at $v$, the implicit maps $A(F)\rightarrow\Phi_{A,v}(\kappa_v)$ send an $F$-point of $A$ to the component containing its reduction, $B(F)^0$ denotes the intersection of the kernels of all the maps $B(F)\rightarrow\Phi_{B,v}(\kappa_v)$ for $v\nmid\infty$, and the valuations at finite places $v$ are normalised to have image $\Z$. In fact, the quadratic refinement of Grothendieck's monodromy pairing which we produce has already been considered by Breen \cite[Section 3]{breen}, using the theory of cubical structures on torsors.
\end{remark}


\subsubsection{The $\ell$-adic and archimedean main theorems}

Now that we have made precise the form of local heights best adapted to our setup, we are now able to state our first main theorem, which is a more or less direct generalisation of theorem \ref{thm:prototype_version} to the setup of line bundles on general abelian varieties.

\begin{theorem}\label{thm:main_theorem_l-adic}
Let $A/K$ be an abelian variety over a $p$-adic field $K$, $L/A$ a line bundle, and $\tilde 0\in L^\times(K)$ a basepoint in the complement $L^\times=L\setminus0$ of the zero section lying over $0\in A(K)$. Let $U/\Q_\ell$ be the $\Q_\ell$-pro-unipotent (in fact, unipotent) \'etale fundamental group of $L^\times$ based at $\tilde 0$ where $\ell\neq p$. Then the natural map $\Q_\ell(1)\hookrightarrow U$ from the inclusion of the fibre over $0\in A(K)$ induces a bijection\[\H^1(G_K,\Q_\ell(1))\isoarrow\H^1(G_K,U(\Q_\ell)),\] and the composite map\[L^\times(K)\rightarrow\H^1(G_K,U(\Q_\ell))\leftisoarrow\H^1(G_K,\Q_\ell(1))\isoarrow\Q_\ell\]is the N\'eron log-metric on $(L,\tilde 0)$ (and in particular is $\Q$-valued).

Here the isomorphism $\H^1(G_K,\Q_\ell(1))\isoarrow\Q_\ell$ is the one arising from Kummer theory, normalised so that the composite\[K^\times\rightarrow\H^1(G_K,\Q_\ell(1))\isoarrow\Q_\ell\]is the ($\Q$-valued) valuation on $K$.
\end{theorem}

\begin{remark}
	Phrased in the language of N\'eron functions, theorem~\ref{thm:main_theorem_l-adic} says that if~$s$ is a non-zero rational section of the line bundle~$L$, with divisor~$D$, then the composite map\[(A\setminus D)(K)\xrightarrow{s}L^\times(K)\rightarrow\H^1(G_K,U(\Q_\ell))\leftisoarrow\H^1(G_K,\Q_\ell(1))\isoarrow\Q_\ell\] is a N\'eron function on~$A$ associated with~$D$. This implies in particular that the N\'eron function associated with~$D$ factors through the non-abelian Kummer map\[(A\setminus D)(K)\to\H^1(G_K,U'(\Q_\ell))\]where~$U'$ is the $\Q_\ell$-pro-unipotent \'etale fundamental group of~$A\setminus D$.
\end{remark}

The archimedean analogue of theorem \ref{thm:main_theorem_l-adic} in terms of the $\R$-pro-unipotent Betti fundamental group (with its mixed Hodge structure) then has a very similar statement. For the sake of brevity, we will focus only on the case $K=\C$.

\begin{theorem}\label{thm:main_theorem_archimedean}
Let $A/\C$ be an abelian variety, $L/A$ a line bundle, and $\tilde 0\in L^\times(\C)$ a basepoint in the complement $L^\times=L\setminus0$ of the zero section lying over $0\in A(\C)$. Let $U/\R$ be the $\R$-pro-unipotent (in fact, unipotent) Betti fundamental group of $L^\times$ based at $\tilde 0$. Then the natural map $\R(1)\hookrightarrow U$ from the inclusion of the fibre over $0\in A(\C)$ induces a bijection\[\H^1(\MHS_\R,\R(1))\isoarrow\H^1(\MHS_\R,U(\R)),\] and the composite map\[L^\times(\C)\rightarrow\H^1(\MHS_\R,U(\R))\leftisoarrow\H^1(\MHS_\R,\R(1))\isoarrow\R\]is the N\'eron log-metric on $(L,\tilde 0)$.

Here $\H^1(\MHS_\R,W(\R))$ denotes the set of isomorphism classes of torsors under $W$ with compatible mixed Hodge structure. The isomorphism $\H^1(\MHS_\R,\R(1))\isoarrow\R$ here is normalised so that the composite\[\C^\times\rightarrow\H^1(\MHS_\R,\R(1))\isoarrow\R\]is the valuation $v=-\log|\cdot|$ on $\C$.
\end{theorem}

\subsubsection{Method of proof}\label{ss:method}

The proofs of theorems~\ref{thm:main_theorem_l-adic} and~\ref{thm:main_theorem_archimedean} proceed by showing that the N\'eron log-metrics are uniquely characterised by a certain list of properties (see lemmas~\ref{lem:neron_log-metrics} and~\ref{lem:archimedean_neron_log-metrics}), and then verifying these properties for the maps described in the statements of the respective theorems. This is mostly just a case of chasing some diagrams; the only non-formal input required is showing that the maps described in theorems~\ref{thm:main_theorem_l-adic} and~\ref{thm:main_theorem_l-adic} are locally constant and locally bounded, respectively. The first of these was shown to hold in much greater generality in \cite[Theorem~1.1]{me:local_constancy}: for any smooth geometrically connected variety~$Y/K$ with $\Q_\ell$-pro-unipotent \'etale fundamental group~$U$ based at some $y\in Y(K)$, the non-abelian Kummer map
\[
Y(K) \to \H^1(G_K,U(\Q_\ell))
\]
is locally constant (provided $\ell\neq p$).

In the archimedean case, we will prove the corresponding analogue of \cite[Theorem~1.1]{me:local_constancy}, using the theory of higher Albanese manifolds of Hain--Zucker.

\begin{theorem}\label{thm:continuity_of_Kummer_maps_archimedean}
	Let $Y/\C$ be a smooth connected variety, $y\in Y(\C)$ a basepoint, and $U/\R$ the pro-unipotent Betti fundamental group of $Y$ based at $y$. Then:
	\begin{enumerate}
		\item $U/\R$ carries a mixed Hodge structure with only negative weights;
		\item $\H^1(\MHS_\R,U(\R))$ is canonically a pro-$C^\infty$ real manifold; and
		\item the non-abelian Kummer map\[Y(\C)\rightarrow\H^1(\MHS_\R,U(\R))\]is pro-$C^\infty$ (and in particular continuous).
	\end{enumerate}
\end{theorem}

\subsection{Local Bloch--Kato Selmer sets and homotopical algebra}

\subsubsection{Definitions and the $p$-adic main theorem}

Our $p$-adic analogue of theorem \ref{thm:main_theorem_l-adic} will describe $p$-adic local heights as functions on spaces of torsors under $\Q_p$-pro-unipotent fundamental groups. However, we cannot hope that anything as naive as theorem \ref{thm:main_theorem_l-adic} will hold verbatim when $\ell=p$, not least because the Galois cohomology $\H^1(G_K,\Q_p(1))$ is much larger than $1$-dimensional. Indeed, in the $\ell=p$ setting, the correct way to recover the valuation on $K^\times$ from Galois cohomology is to consider the quotient\[\H^1_{g/e}(G_K,\Q_p(1)):=\H^1_g(G_K,\Q_p(1))/\H^1_e(G_K,\Q_p(1))\]of local Bloch--Kato Selmer groups \cite[Equation 3.7.2]{bloch-kato}. This is one-dimensional, and the composite\[K^\times\rightarrow\H^1_{g/e}(G_K,\Q_p(1))\isoarrow\Q_p\]with the non-abelian Kummer map is, for a suitable normalisation of the final isomorphism, the valuation on $K$. This motivates what will be our adaptation of theorem \ref{thm:main_theorem_l-adic} to the $p$-adic setting: we will define a certain ``quotient'' $\H^1_{g/e}(G_K,U(\Q_p))$ of local Bloch--Kato Selmer sets for a general (usually de Rham) representation of $G_K$ on a finitely generated pro-unipotent group $U/\Q_p$, and when $U/\Q_p$ is the $\Q_p$-pro-unipotent \'etale fundamental group of $L^\times$ as earlier, we will realise our local heights as functions on $\H^1_{g/e}(G_K,U(\Q_p))$ via the mod $\H^1_e$ non-abelian Kummer map\[L^\times(K)\rightarrow\H^1_{g/e}(G_K,U(\Q_p)).\]Indeed, $\H^1_{g/e}(G_K,U(\Q_p))$ will prove in this case to be $1$-dimensional, analogously to theorems \ref{thm:main_theorem_l-adic} and \ref{thm:main_theorem_archimedean}.

Let us now recall the definition of the local Bloch--Kato Selmer sets associated to a Galois representation of a $p$-adic local field $K$ on a finitely generated pro-unipotent group, as well as giving the definition of their ``relative quotients''. As usual with non-abelian Galois cohomology, the resulting objects do not carry any natural group structure, instead simply being pointed sets.



\begin{definition}[Local Bloch--Kato Selmer sets]\label{def:bloch-kato_sets}
Let $U/\Q_p$ be a representation of $G_K$ on a finitely generated pro-unipotent group (see definition-lemma \ref{def-lem:unipotent_representation}). We define the \emph{local Bloch--Kato Selmer (pointed) sets} $\H^1_e(G_K,U(\Q_p))\subseteq\H^1_f(G_K,U(\Q_p))\subseteq\H^1_g(G_K,U(\Q_p))$ to be the respective kernels of the natural maps
\[\H^1(G_K,U(\Q_p))\rightarrow\H^1(G_K,U(\B_\cris^{\varphi=1}))\]
\[\H^1(G_K,U(\Q_p))\rightarrow\H^1(G_K,U(\B_\cris))\]
\[\H^1(G_K,U(\Q_p))\rightarrow\H^1(G_K,U(\B_\st))\]
on continuous cohomology. Here $\B_\cris$ and $\B_\st$ denote the usual period rings, and the continuous Galois action on, for example, $U(\B_\st)=\Hom(\Spec(\B_\st),U)$ is the ``semilinear'' one induced from the actions of $G_K$ on $\B_\st$ and on $U$.

More generally, we let $\sim_{\H^1_e}$, $\sim_{\H^1_f}$ and $\sim_{\H^1_g}$ be the equivalence relations on the cohomology set $\H^1(G_K,U(\Q_p))$ whose equivalence classes are the fibres of the respective maps above (so that $\H^1_*(G_K,U(\Q_p))$ is the $\sim_{\H^1_*}$-equivalence class of the distinguished point). We then define the \emph{local Bloch--Kato Selmer quotients} to be the quotient (pointed) sets:
\[\H^1_{f/e}(G_K,U(\Q_p)):=\H^1_f(G_K,U(\Q_p))/\sim_{\H^1_e};\]
\[\H^1_{g/e}(G_K,U(\Q_p)):=\H^1_g(G_K,U(\Q_p))/\sim_{\H^1_e};\]
\[\H^1_{g/f}(G_K,U(\Q_p)):=\H^1_g(G_K,U(\Q_p))/\sim_{\H^1_f}.\]
\end{definition}

\begin{remark}
The original definition of $\H^1_g$ in \cite[Equation 3.7.2]{bloch-kato} differs from ours in that the semistable period ring $\B_\st$ is replaced by the de Rham one $\B_\dR$. However, we will see in lemma \ref{lem:our_H^1_g_is_right} that, at least when $U$ is de Rham (see definition-lemma \ref{def-lem:admissible_rep}) -- for instance the pro-unipotent fundamental group of any smooth $K$-variety -- there is no difference between the two definitions, so we are free to adopt the one which is best-suited to our techniques. When $U$ is replaced by a $p$-adic representation $V$, this is proved in unpublished work by Hyodo \cite{hyodo}, and later in other places, e.g.\ \cite{berger}.
\end{remark}

A major part of this article will be concerned with developing tools for the study of non-abelian Bloch--Kato Selmer sets and their quotients, but before we outline how we intend to develop these tools, let us use the local Bloch--Kato Selmer quotients to give a precise statement of our $p$-adic analogue of theorem \ref{thm:main_theorem_l-adic}.

\begin{theorem}\label{thm:main_theorem_p-adic}
Let $A/K$ be an abelian variety over a $p$-adic field $K$, $L/A$ a line bundle, and $\tilde 0\in L^\times(K)$ a basepoint in the complement $L^\times=L\setminus0$ of the zero section lying over $0\in A(K)$. Let $U/\Q_p$ be the $\Q_p$-pro-unipotent (in fact, unipotent) \'etale fundamental group of $L^\times$ based at $\tilde 0$. Then the natural map $\Q_p(1)\hookrightarrow U$ from the inclusion of the fibre over $0\in A(K)$ induces a bijection\[\H^1_{g/e}(G_K,\Q_p(1))\isoarrow\H^1_{g/e}(G_K,U(\Q_p)),\] and the (well-defined) composite map\[L^\times(K)\rightarrow\H^1_{g/e}(G_K,U(\Q_p))\leftisoarrow\H^1_{g/e}(G_K,\Q_p(1))\isoarrow\Q_p\]is the N\'eron log-metric on $(L,\tilde 0)$ (and in particular is $\Q$-valued).

Here the isomorphism $\H^1_{g/e}(G_K,\Q_p(1))\isoarrow\Q_p$ is normalised so that the composite\[K^\times\rightarrow\H^1_{g/e}(G_K,\Q_p(1))\isoarrow\Q_p\]is the ($\Q$-valued) valuation on $K$.
\end{theorem}

\begin{remark}\label{rmk:what_about_f/e}
Informally, theorem \ref{thm:main_theorem_p-adic} shows that when $U/\Q_p$ is the unipotent fundamental group of $L^\times$ as above, then $\H^1_e(G_K,U(\Q_p))$ is codimension $1$ in $\H^1_g(G_K,U(\Q_p))$ and the gap between $\H^1_e$ and $\H^1_g$ describes the local component of the height function. It is natural to wonder where the intermediate set $\H^1_f$ sits between these two sets.

In fact, in our setup it will already be equal to $\H^1_e$, as we shall see in corollary \ref{cor:H^1_e=H^1_f} (see remark \ref{rmk:D_cris^phi=1_is_trivial}). Correspondingly, according to Bloch and Kato's $\ell$-adic analogies for local Bloch--Kato Selmer groups \cite[Equation 3.7.1]{bloch-kato}, one would expect the $\Q_\ell$-unipotent fundamental group of $L^\times$ to satisfy $\H^1_\nr(G_K,U(\Q_\ell))=1$ where\[\H^1_\nr(G_K,U(\Q_\ell)):=\ker\left(\H^1(G_K,U(\Q_\ell))\rightarrow\H^1(I_K,U(\Q_\ell))\right)\]with $I_K\unlhd G_K$ the inertia subgroup -- this is also true, and proven in proposition \ref{prop:H^1_nr=1} (see remark \ref{rmk:H^1_nr=1}).
\end{remark}

\subsubsection{Cosimplicial models for local Bloch--Kato Selmer sets}

Ultimately, our aim is to use the local Bloch--Kato Selmer sets and quotients not just to state the $p$-adic analogue of theorem \ref{thm:main_theorem_l-adic}, but also to mimic closely the proof strategy outlined in \S\ref{ss:method}. The non-formal input required -- that the non-abelian Kummer map
\[
Y(K) \to \H^1_{g/e}(G_K,U(\Q_p))
\]
modulo~$\sim_{\H^1_e}$ is well-defined and locally constant -- was proven in \cite[Theorem~1.2]{me:local_constancy}. In this paper, we will develop a formal calculus for working with the local Bloch--Kato Selmer sets~$\H^1_*(G_K,U(\Q_p))$ and the quotients described above, for a general de Rham representation of~$G_K$ on a finitely generated pro-unipotent group~$U/\Q_p$. In particular, our theory will give us a way to produce long exact sequences relating local Bloch--Kato Selmer quotients in central extensions, analogously to the long exact sequence in cohomology.

These non-abelian Bloch--Kato Selmer sets $\H^1_*(G_K,U(\Q_p))$ have already been studied in the work of Kim \cite{minhyong:selmer}, using their moduli interpretations as classifying spaces for torsors possessing trivialisations after base change to various period rings (see section \ref{s:torsors}). However, our approach will differ significantly from Kim's, in that we will in general forgo these explicit moduli descriptions, instead producing certain explicit cosimplicial models for the local Bloch--Kato Selmer sets, enabling us to bring methods from homotopical algebra to bear. The advantages of this approach are twofold: firstly, our explicit models will be cosimplicial groups, and thus will retain the group structure which is lost in passage to non-abelian Galois cohomology; and secondly, the models will form a kind of derived refinement of the local Bloch--Kato Selmer sets, and in particular will contain information in other degrees which is not seen directly by the local Bloch--Kato Selmer set but is nonetheless of relevance, for example in the construction of long exact sequences.

Explicit cochain models for local Bloch--Kato Selmer groups have already been produced by Nekov\'a\u r in the abelian case \cite{nekovar,nekovar_abel-jacobi}, and our treatment closely parallels his in the non-abelian context. To briefly review Nekov\'a\u r's methods, let us recall that there are certain cochain complexes $\mathsf C^\bullet_*(V)$ for $*\in\{e,f,g\}$ associated to any de Rham representation $V$ of $G_K$, whose first cohomology $\H^1\left(\mathsf C^\bullet_*(V)\right)$ is canonically identified with the corresponding local Bloch--Kato Selmer group $\H^1_*(G_K,V)$ (see \cite[Section 1.19]{nekovar}). The complexes $\mathsf C^\bullet_*(V)$ are built out of various Dieudonn\'e functors evaluated at $V$; for instance, the complex $\mathsf C^\bullet_e(V)$ is the two-term complex\[\D_\cris^{\varphi=1}(V)\overset{-\iota}\longrightarrow\D_\dR(V)/F^0,\]where $\iota$ is the natural map. The reader will note that the assertion that $\H^1(\mathsf C^\bullet_e(V))=\H^1_e(G_K,V)$ is equivalent to part of the Bloch--Kato exponential sequence \cite[Corollary 3.8.4]{bloch-kato}.

In the pro-unipotent setting we will generalise this picture by associating to each de Rham representation of $G_K$ on a finitely generated pro-unipotent group $U/\Q_p$ a certain \emph{cosimplicial group} $\mathsf C^\bullet_*(U(\Q_p))$, whose \emph{first cohomotopy set} $\pi^1\left(\mathsf C^\bullet_*(U(\Q_p))\right)$ is canonically identified with the corresponding local Bloch--Kato Selmer set $\H^1_*(G_K,U(\Q_p))$ (see section \ref{c:homotopical_algebra} for definitions). The cosimplicial groups $\mathsf C^\bullet_*(U(\Q_p))$ are again built out of certain non-abelian Dieudonn\'e functors evaluated on $U$ (see lemma \ref{lem:dieudonne_functor}). More precisely, we will define various cosimplicial $\Q_p$-algebras $\B^\bullet_*$ with $G_K$-action, built out of the standard $p$-adic period rings, so that the desired cosimplicial group $\mathsf C^\bullet_*(U(\Q_p))$ is then produced by taking points of $U$ in $\B^\bullet_*$ and taking $G_K$-invariants. In fact, we will use exactly the same setup to analyse the local Bloch--Kato quotients $\H^1_{g/e}$ and $\H^1_{f/e}$, defining cosimplicial $\Q_p$-algebras $\B^\bullet_{g/e}$ and $\B^\bullet_{f/e}$, and obtaining the description\[\H^1_*(G_K,U(\Q_p))\cong\pi^1\left(U(\B^\bullet_*)^{G_K}\right)\]also in this case.

Just as in the abelian case, these descriptions can be viewed as various instances of exponential-type exact sequences -- indeed, this is how we will prove these descriptions. For instance, in the particular case $*=e$, we will prove the description via the following non-abelian analogue of the classical Bloch--Kato exponential.

\begin{lemma}[Non-abelian Bloch--Kato exponential, see corollary~\ref{cor:pi^1_is_bloch-kato}]\label{lem:intro_exponential}
Let $U/\Q_p$ be a de Rham representation of $G_K$ on a finitely generated pro-unipotent group. Then there is an exact sequence\[1\!\rightarrow\! U(\Q_p)^{G_K}\!\rightarrow\D_\cris^{\varphi=1}(U)(\Q_p)\times\D_\dR^+(U)(K)\actsarrow\D_\dR(U)(K)\rightarrow\H^1_e(G_K,U(\Q_p))\!\rightarrow\!1,\]where $\D_\cris^{\varphi=1}$, $\D_\dR$ and $\D_\dR^+$ denote the Dieudonn\'e functors associated to the period rings $\B_\cris^{\varphi=1}$, $\B_\dR$ and $\B_\dR^+$ respectively (see lemma \ref{lem:dieudonne_functor} for how to apply Dieudonn\'e functors to pro-unipotent groups).
\end{lemma}

The exact sequence appearing in our non-abelian Bloch--Kato exponential is a mixed exact sequence consisting of both groups and pointed sets. Since such sequences will be ubiquitous throughout this paper, let us fix our chosen notation for such sequences.

\begin{definition}\label{def:exact_sequence}
An \emph{exact sequence (of groups and pointed sets)}\[\cdots\rightarrow X^{r-1}\rightarrow X^r\rightarrow X^{r+1}\rightarrow\cdots\rightarrow X^j\centarrow X^{j+1}\rightarrow\cdots\rightarrow X^k\actsarrow X^{k+1}\rightarrow\cdots\]consists of:
\begin{itemize}
	\item two distinguished indices $j<k$;
	\item an exact sequence (finite or infinite) of groups $\cdots\rightarrow X^{k-1}\rightarrow X^k$;
	\item an exact sequence (finite or infinite) of pointed sets $X^{k+1}\rightarrow X^{k+2}\rightarrow\cdots$; and
	\item a right action of $X^k$ on $X^{k+1}$ (not necessarily preserving the basepoint);
\end{itemize}
such that:
\begin{itemize}
	\item $X^r$ is abelian for $r\leq j$;
	\item the image of $X^j\rightarrow X^{j+1}$ is a central subgroup;
	\item the image of $X^{k-1}\rightarrow X^k$ is exactly the stabiliser of the basepoint of $X^{k+1}$ under the action of $X^k$; and
	\item the orbits of the action of $X^k$ on $X^{k+1}$ are exactly the fibres of the map $X^{k+1}\rightarrow X^{k+2}$.
\end{itemize}
Note that in this setup, the action of $X^k$ on the basepoint of $X^{k+1}$ provides a map $X^k\rightarrow X^{k+1}$ making the whole sequence exact.
\end{definition}

\begin{remark}
The prototypical example of such a sequence is the homotopy exact sequence associated to a fibration \cite[Section IX.4.1]{bousfield-kan}, where there are exactly three non-abelian group terms ($k=j+3$) and three pointed set terms. In fact, essentially all the exact sequences we will study in this paper will have this same feature.
\end{remark}

\begin{remark}\label{rmk:explicit_H^1_e}
In our non-abelian Bloch--Kato exact sequence, the action of the group $\D_\cris^{\varphi=1}(U)(\Q_p)\times\D_\dR^+(U)(K)$ on the set $\D_\dR(U)(K)$ is given by $(x,y)\colon z\mapsto y^{-1}zx$, so that, concretely, the exponential exact sequence asserts that\[U(\Q_p)^{G_K}=\D_\dR^+(U)(K)\cap\D_\cris^{\varphi=1}(U)(\Q_p)\](inside $\D_\dR(U)(K)$), and that we have a canonical identification\[\H^1_e(G_K,U(\Q_p))\cong\D_\cris^{\varphi=1}(U)(\Q_p)\backslash\D_\dR(U)(K)/\D_\dR^+(U)(K)\]of $\H^1_e(G_K,U(\Q_p))$ as a double-coset space.
\end{remark}

\begin{remark}
An alternative description of a non-abelian Bloch--Kato exponential has been given by Sakugawa \cite[Theorem 1.1]{sakugawa} following on from work of Kim \cite{minhyong:selmer} under some auxiliary conditions including triviality of $\D_\cris^{\varphi=1}(U)$ (but without necessarily assuming $U$ unipotent), obtaining in this case a Bloch--Kato exponential \emph{isomorphism}\[\H^1_f(G_K,U(\Q_p))\isoarrow\D_\dR(U)(K)/\D_\dR^+(U)(K).\]Indeed, we will show in corollary \ref{cor:H^1_e=H^1_f} that, under this additional hypothesis, one has $\H^1_e(G_K,U(\Q_p))=\H^1_f(G_K,U(\Q_p))$, so one recovers Sakugawa's isomorphism from ours.
\end{remark}

\begin{remark}
When $U=V$ is abelian, the central map\[\D_\dR^+(V)\oplus\D_\cris^{\varphi=1}(V)\rightarrow\D_\dR(V)\]from our exponential exact sequence lemma \ref{lem:intro_exponential} is the negative of the usual one. This is due to an unfortunate clash between the signs in the exact sequences in \cite[Section 1.19]{nekovar} with those of \cite[Proposition 1.17]{bloch-kato}, and between the order conventions for path-composition in the papers of Kim and of Hain. In this paper we will follow the conventions of Kim and Deligne in composing paths in the functional order -- $fg$ is path $g$ followed by path $f$ -- and following the sign conventions of Nekov\'a\u r, pointing out discrepancies with the conventions of our sources as they arise.
\end{remark}

\subsubsection{A log-crystalline version}

In accordance with Deligne's philosophy of the motivic fundamental group, one might wonder whether, in addition to our $\ell$-adic, $p$-adic and Hodge-theoretic descriptions, one can give a description of the N\'eron log-metric on a pointed line bundle $(L,\tilde0)$ on an abelian variety $A$ over a $p$-adic local field $K$ in terms of the pro-unipotent log-crystalline fundamental group of the special fibre of a suitable integral log-scheme model of $L^\times$. Though this is not the main thrust of this paper, we will briefly outline how to  reinterpret theorem \ref{thm:main_theorem_p-adic} in these terms.

To briefly explain, let us fix $(L,\tilde0)$, $A$ and $K$ as in the preceding paragraph, and assume for simplicity that $A$ has semistable reduction. If we write $U/\Q_p$ for the unipotent \'etale fundamental group of $L^\times$ based at $\tilde0\in L^\times(K)$, then our assumptions ensure that $U$ is semistable in the sense of definition-lemma \ref{def-lem:admissible_rep}, since it is de Rham by \cite[Theorem~1.4(1)]{me:local_constancy} and is a central extension of semistable representations. The finitely generated pro-unipotent group $\D_\st(U)/K_0$ (lemma \ref{lem:dieudonne_functor}) then carries an ind-weakly admissible filtered $(\varphi,N)$-module structure on its affine ring $\O(\D_\st(U))=\D_\st(\O(U))$ (section \ref{ss:extra_structure}), and the local Bloch--Kato Selmer set $\H^1_g(G_K,U(\Q_p))$ classifies torsors under $\D_\st(U)$ with compatible ind-weakly admissible filtered $(\varphi,N)$-module structures on their affine rings (see lemma \ref{lem:isocrystal_interpretation_H^1_g}). We will see in lemma \ref{lem:isocrystal_interpretation_H^1_g/e} that passing from $\H^1_g$ to $\H^1_{g/e}$ corresponds to forgetting the Hodge filtration on these torsors, so that $\H^1_{g/e}(G_K,U(\Q_p))$ has a natural log-crystalline interpretation as the moduli set $\H^1(\Mod_K^{\varphi,N},\D_\st(U))$ of torsors under $\D_\st(U)$ with compatible ind-$(\varphi,N)$-module structures on their affine rings. This allows us to reinterpret theorem \ref{thm:main_theorem_p-adic} in terms of the non-abelian Kummer map\[L^\times(K)\rightarrow\H^1(\Mod_K^{\varphi,N},\D_\st(U))\]assigning to some $\tilde z\in L^\times(K)$ the class of the torsor $\D_\st(P_{\tilde0,\tilde z}):=\Spec(\D_\st(\O(P_{\tilde0,\tilde z})))$, where $P_{\tilde0,\tilde z}$ is the $\Q_p$-unipotent \'etale torsor of paths from $\tilde0$ to $\tilde z$ in $L^\times$. (The fact that~$\D_\st(P_{\tilde 0,\tilde z})$ is a $\D_\st(U)$-torsor is a consequence of the fact that it $\O(P_{\tilde 0,\tilde z})$ is ind-de Rham by \cite[Theorem~1.4(1)]{me:local_constancy}, hence ind-semistable.)

In order to see this rephrasing as a truly log-crystalline version, we would like to identify $\D_\st(U)$ (resp.\ $\D_\st(P_{\tilde0,\tilde z})$) as the pro-unipotent log-crystalline fundamental group (resp.\ torsor of paths) of the special  fibre of a suitable log-scheme model\footnote{Officially, a proper log-smooth log-scheme $\bar{\mcal L}$ over $\O_K$ whose generic fibre is the total space of the projective closure $\P(L\oplus\O_A)$ of $L$ with the log-structure coming from the union of the $0$- and $\infty$-sections, such that the log-structure on $\Spec(\O_K)$ pulled back from $\bar{\mcal L}$ along the closure of $\tilde 0\in\bar{\mcal L}(K)$ is the usual one arising from the special point.} of $L^\times$. Several definitions of such an object are available in the literature:  there is a Tannakian definition due to Shiho \cite[Definition 4.1.5]{shiho_1}, a direct construction via the de Rham--Witt complex due to Kim and Hain \cite[Definition 1]{kim-hain}, or one could take the Hyodo--Kato realisation of Cushman's pro-unipotent fundamental group in the category of Nori motives \cite[Theorem 3.1]{cushman}, as in \cite[Section 4.3]{deglise-niziol}. These definitions should describe the same object, though we have been unable to find a proof of this in the literature. In any case, the log-crystalline fundamental group is known \cite[Corollary 4.19(2)]{deglise-niziol} or believed \cite[Remark 3.4.1]{shiho_p-adic_Hodge} to possess a comparison isomorphism with the \'etale fundamental group, identifying it with $\D_\st(U)$ as claimed.

\subsection{Overview of sections and standing assumptions}

After recalling some basic facts about pro-unipotent groups and filtrations in section \ref{c:preliminaries}, we will begin in section \ref{c:main_theorem_l-adic} by proving the $\ell$-adic version of our main theorem, theorem \ref{thm:main_theorem_l-adic}. The proof is relatively simple, but the details will serve to motivate many of the lemmas necessary in the proof of its much harder cousin, theorem \ref{thm:main_theorem_p-adic}. The following sections serve to build up to the proof of this latter theorem, developing first the basic language of local Galois representations on finitely generated pro-unipotent groups in section \ref{c:reps}, then studying the local Bloch--Kato Selmer sets and quotients in section \ref{c:bloch-kato} using the homotopical-algebraic language reviewed in section \ref{c:homotopical_algebra}.

This theoretical basis being established, the proof in section \ref{c:main_theorem_p-adic} of the $p$-adic main theorem then follows exactly the blueprint of the $\ell$-adic proof, and indeed is the shortest section of the paper. Then, in section \ref{c:main_theorem_archimedean} we prove the archimedean version of our main theorem, developing a mixed-Hodge-theoretic analogue of the representation-theoretic language of the rest of the paper, and relating the archimedean non-abelian Kummer map to the higher Albanese maps of Hain and Zucker \cite{hain-zucker}. The proof of the archimedean main theorem then follows exactly the blueprint of the previous two proofs.


\smallskip

Since the majority of this paper will be devoted to the analysis over a non-archimedean base field, let us fix some notation for the rest of the paper.

\begin{notation}[Standing assumptions]
Fix a prime $p$, a finite extension $K/\Q_p$, and an algebraic closure $\bar K/K$ determining an absolute Galois group $G_K=G_{\bar K|K}$. Let $\O_K$ denote the ring of integers, $\varpi$ a uniformiser generating the maximal ideal $\mideal_K=(\varpi)$, $k=\O_K/\varpi$ the residue field, and $v\colon K^\times\rightarrow\Q$ the valuation, normalised so that $v(p)=1$. We also denote by $\C_K$ the completion of $\bar K$, by $K_0$ the maximal unramified subfield of $K$, and by $\B_\cris$, $\B_\st$ and $\B_\dR$ the period rings constructed by Fontaine \cite{fontaine2}.
\end{notation}

\subsection{Acknowledgements}

I am grateful to Jonathan Pridham for helpful discussions about comparison theorems for fundamental groups, and to Fabrizzio Andreatta, Christiaan Peters and Joseph Steenbrink for taking the time to answer my questions on various technical aspects of their papers and books. I am also grateful to Adam Morgan for introducing me to Breen's theory of cubical structures. Lastly, I am very grateful to my two thesis examiners, Damian R\"ossler and Gerd Faltings, for their careful reading of the initial version of the dissertation on which this paper is based and their many helpful suggestions.

Particular thanks are due to Netan Dogra for many helpful discussions, and also to Minhyong Kim, without whose guidance the present paper would not have been possible.

\smallskip

This article forms an abridged version of the author's dissertation at the University of Oxford \cite{thesis}.
\section{Preliminaries}
\label{c:preliminaries}

Since pro-unipotent groups form an integral part of this paper, let us briefly collect a few basic facts on them for the benefit of the reader. Since the majority of the statements here qualify as well-known, we omit all proofs here, referring the reader to \cite{milne} and \cite{javier}, and the corresponding section of \cite{thesis} for details.

Throughout this section, we fix a characteristic $0$ field $F$; all schemes, vector spaces etc.\ will be over $F$.

\subsection{Pro-unipotent groups and their Lie algebras and Hopf algebras}

A \emph{pro-unipotent group} $U/F$ is an affine group-scheme which is a (possibly infinite) iterated extension of vector groups. A pro-unipotent group $U/F$ is very closely related to both its Hopf algebra \cite[Chapter $3$]{milne}, i.e.\ the affine ring $\O(U)$, and its Lie algebra \cite[Chapter $12$]{milne}\[\Lie(U) := \ker\left(U(F[\epsilon]/(\epsilon^2))\rightarrow U(F)\right).\]The Lie algebra of~$U$ is pro-nilpotent, i.e.\ an inverse limit of finite-dimensional nilpotent Lie algebras, and dually its Hopf algebra satisfies a certain conilpotency property.

\begin{definition}[Conilpotency filtration]\label{def:J-filtration}
Let $H$ be a commutative Hopf algebra over $F$. The \emph{conilpotency filtration} \cite[Section 3.3.4]{javier} on $H$ is the increasing filtration $J_\bullet H$ defined by\[J_nH:=\ker\left(\Delta^{n+1}\colon H\rightarrow(H/F)^{\otimes(n+1)}\right)\]where $\Delta^{n+1}$ denotes the $n$-fold comultiplication on $H$ and $H/F$ denotes the cokernel of the unit map $F\rightarrow H$. In other words, if $J$ denotes the augmentation ideal of the dual Hopf algebra $H^\dual$, then $J_nH$ is the annihilator of $J^{n+1}$. It follows from this description that the conilpotency filtration on $H$ is preserved by the Hopf algebra structure maps.

The Hopf algebra $H$ is called \emph{conilpotent} (or \emph{$J$-adically cocomplete} in \cite{thesis}, or \emph{coconnected} in \cite[Definition 14.4]{milne}) just when its conilpotency filtration is exhaustive. Equivalently, $H$ is conilpotent just when its dual $H^\dual$ is $J$-adically complete.
\end{definition}

\begin{remark}
When dealing with pro-nilpotent Lie algebras $\Lie(U)$ or dual Hopf algebras $\O(U)^\dual$, these objects should be viewed as \emph{pro-finite dimensional vector spaces}, i.e.\ as pro-objects in the category of finite dimensional vector spaces. The category of pro-finite dimensional vector spaces is canonically opposite to the category of all vector spaces via the functors sending a vector space $V$ to ids dual $V^\dual=\Hom(V,F)$ and a pro-finite dimensional vector space $V=\liminv(V_i)$ to its \emph{decompleted dual} $V^\reddual:=\limdir(V_i^\dual)$. The pro-object structure on the dual $V^\dual$ of a vector space is the natural one arising from writing $V$ as a filtered colimit of finite dimensional subspaces.

The category of pro-finite dimensional vector spaces supports a symmetric monoidal \emph{completed tensor product} $\hat\otimes$, given by
\[
V\hat\otimes W := (V^\reddual\otimes W^\reddual)^\dual = \liminv(V_i\otimes W_j)
\]
for pro-finite dimensional vector spaces $V=\liminv(V_i)$ and $W=\liminv(W_j)$. The natural Hopf algebra structure on the dual $\O(U)^\dual$ is not a true Hopf algebra structure, but the structure of a completed Hopf algebra \cite[Definition 3.60]{javier}, i.e.\ a Hopf algebra object in the category of pro-finite dimensional vector spaces with respect to the completed tensor product. In particular, the comultiplication on $\O(U)^\dual$ is given by a map
\[
\Delta\colon \O(U)^\dual\rightarrow \O(U)^\dual\hat\otimes\O(U)^\dual
\]
which does not in general factor through the usual tensor product $\O(U)^\dual\otimes\O(U)^\dual$ \cite[Example 3.61]{javier}.
\end{remark}

It turns out that the property of being conilpotent \emph{exactly} characterises affine rings of pro-unipotent groups \cite[Theorem 14.5]{milne}, and the functors $U\mapsto\Lie(U)$ and $U\mapsto\O(U)$ provide equivalences of categories
\begin{align*}\label{eq:equivalences}\tag{$\ast$}
\begin{split}
\{\text{pro-unipotent groups}\} &\isoarrow \{\text{pro-nilpotent Lie algebras}\} \\
 &\isoarrow \{\text{conilpotent commutative Hopf algebras}\}^\op.
\end{split}
\end{align*}
Since we will use this three-way equivalence throughout this paper, let us briefly describe each of the four remaining functors.

\smallskip

Firstly, if $H$ is a conilpotent commutative Hopf algebra, then one recovers its associated pro-unipotent group $U$ as its spectrum $\Spec(H)$. Equivalently, $U$ is the functor of \emph{grouplike elements} of the dual Hopf algebra $H^\dual$, i.e.\[U(\Lambda)=H_\Lambda^{\dual,\gplike}:=\{\gamma\in H_\Lambda^\dual|\Delta(\gamma)=\gamma\hat\otimes\gamma\text{ and }\varepsilon(\gamma)=1\}\]
for any $F$-algebra $\Lambda$, functorial in $\Lambda$, where $H_\Lambda^\dual=\Hom_{\text{$F$-linear}}(H,\Lambda)$ denotes the base-change of $H$ from $F$ to $\Lambda$. The group structure on $U$ is the natural one induced from the multiplication on $H^\dual$.

The Lie algebra $\uu$ corresponding to $H$ can similarly be recovered as the space of \emph{primite elements} of $H^\dual$, i.e.\ the pro-finite dimensional subspace\[\uu=H^{\dual,\prim}:=\{x\in H^\dual|\Delta(x)=x\hat\otimes1+1\hat\otimes x\}.\]
The Lie bracket on $\uu$ is the natural one induced from the commutator bracket on $H^\dual$. The usual exponential power series converges on $\uu\subseteq H^\dual$ and induces a canonical bijection\[\exp\colon\uu=H^{\dual,\prim}\isoarrow H^{\dual,\gplike}=U(F).\]
This can even be upgraded to an isomorphism $\exp\colon\uu\isoarrow U$ of $F$-schemes, where by abuse of notation we conflate $\uu$ with its corresponding affine space $\Spec(\Sym^\bullet\uu^\reddual)$. The inverse isomorphism is denoted $\log\colon U\rightarrow\uu$.

\smallskip

Lastly, if $\uu$ is a pro-nilpotent Lie algebra, then its corresponding pro-unipotent group $U$ is the affine group-scheme on (the affine space corresponding to) $\uu$ whose group law is given by the Baker--Campbell--Hausdorff formula \cite[Theorem 15.36]{milne}\[x\cdot y=x+y+\frac12[x,y]+\frac1{12}[x,[x,y]]-\frac1{12}[y,[x,y]]+\dots.\]This description of $U$ allows many statements about pro-unipotent groups to be reduced to their Lie algebras. For example, exactness of a sequence of pro-unipotent groups can be checked on $F$-points, and surjections (resp.\ injections) of pro-unipotent groups have sections (resp.\ retractions) as $F$-schemes.

The commutative Hopf algebra $H$ associated to $\uu$ is the decompleted dual of the \emph{completed universal enveloping algebra} $\hat{\mcal U}(\uu)$. When $\uu$ is finite dimensional, $\hat{\mcal U}(\uu)$ is defined to be the completion of the universal enveloping algebra $\mcal U(\uu)$ with respect to powers of its augmentation ideal; for general $\uu$ it is defined by taking the inverse limit of the completed universal enveloping algebras of its finite dimensional quotients.

\begin{remark}
For us, especially in section \ref{c:reps}, one important consequence of the equivalences \eqref{eq:equivalences} is that they enable one to make sense of various kinds of extra structures on pro-unipotent groups. For instance, if $\mcal C$ is a Tannakian category with fibre functor $\omega$, then specifying an action of the Tannaka group on a pro-unipotent group $U$ is equivalent to giving $\O(U)$ (resp.\ $\Lie(U)$) the structure of an ind-object (resp.\ pro-object) in the category $\mcal C$ compatible with the Hopf algebra structure (resp.\ with the Lie bracket).
\end{remark}

\subsubsection{The conilpotency filtration and the descending central series}

Let us briefly record a few properties fo the conilpotency filtration $J_\bullet\O(U)$ on the affine ring of a pro-unipotent group $U$. The first asserts a very precise connection between the conilpotency filtration and the descending central series of $U$.

\begin{lemma}\label{lem:J-filtration_is_descending_central_series}
Let $U/F$ be a pro-unipotent group, and for $n\geq0$ let $U_n$ denote the maximal $n$-step pro-unipotent quotient of $U$. Then $\O(U_n)\leq\O(U)$ is the subalgebra generated by $J_n\O(U)$, and the kernel of the surjection $\Sym^\bullet(J_n\O(U))\twoheadrightarrow\O(U_n)$ is the ideal generated by $1-[1]$ and all elements of the form $[x][y]-[xy]$ for $x\in J_i\O(U)$ and $y\in J_j\O(U)$ for $i+j\leq n$. (Here $[-]$ denotes the canonical inclusion $J_n\O(U)\hookrightarrow\Sym^\bullet(J_n\O(U))$.)
\end{lemma}

Although the pro-unipotent groups arising in this paper will in general be infinite dimensional, they will at least be finitely generated (i.e.\ have a finite generated Zariski-dense subgroup of their $F$-points). This translates to the following useful finiteness property of the affine ring $\O(U)$.

\begin{proposition}\cite[Proposition 1.5]{lubotzky-magid}\label{prop:lubotzky-magid}
Let $U/F$ be a pro-unipotent group. Then $U$ is finitely generated if and only if each $J_n\O(U)$ is finite dimensional (if and only if $J_1\O(U)$ is finite dimensional).
\end{proposition}

\subsection{Torsors under pro-unipotent groups}

As well as pro-unipotent groups $U$, we will also study torsors under them, i.e.\ non-empty $F$-schemes $P$ with a (right\nobreakdash-)action $P\times U\rightarrow P$ such that the natural map
\begin{align*}
P\times U &\rightarrow P\times P \\
(\gamma,u) &\mapsto (\gamma\cdot u,\gamma)
\end{align*}
is an isomorphism of schemes. Such torsors automatically have an $F$-point $\gamma\in P(F)$, and the action on any such $\gamma$ provides an isomorphism $U\isoarrow P$ of $U$-torsors (for the right-multiplication action of $U$ on itself).

The affine ring $\O(P)$ of a $U$-torsor carries a canonical increasing \emph{conilpotency filtration} $J_\bullet\O(P)$ analogous to the conilpotency filtration on $\O(U)$ (definition \ref{def:J-filtration}), defined by
\[
J_n\O(P):=\ker\left(\Delta^{n+1}\colon\O(P)\rightarrow(\O(P)/F)\otimes(\O(U)/F)^{\otimes n}\right),
\]
where $\Delta^{n+1}$ denotes the $(n+1)$-fold coaction. This filtration is compatible with the algebra structure on $\O(P)$, as well as its induced $\O(U)$-comodule structure. We will frequently use the following result.

\begin{proposition}\label{prop:graded_isomorphism}
Let $P$ be a torsor under a pro-unipotent group $U$, and choose $\gamma\in P(F)$. Then the isomorphism
\[
\O(P) \isoarrow \O(U)
\]
determined by $\gamma$ is a $J$-filtered isomorphism, and the induced isomorphism
\[
\gr^J_\bullet\O(P) \isoarrow \gr^J_\bullet\O(U)
\]
is independent of the choice of $\gamma$.
\smallskip
In particular, if $U$ is finitely generated then each $J_n\O(P)$ is finite dimensional.
\end{proposition}

As for pro-unipotent groups, we can make sense of the notion of various extra structures on torsors. If $(\mcal C,\omega)$ is a neutralised $F$-linear Tannakian category and $U$ is a pro-unipotent group in $\mcal C$, a \emph{torsor under $U$ (in $\mcal C$)} consists of a $U$-torsor $P$ endowed with the structure of an ind-object of $\mcal C$ on its affine ring $\O(P)$ which is compatible with both its algebra structure and its right $\O(U)$-comodule structure. If $P$ is a $U$-torsor in a neutralised Tannakian category $(\mcal C,\omega)$, the isomorphism
\[
\O(P) \isoarrow \O(U)
\]
determined by a point $\gamma\in P(F)$ need not respect the $\mcal C$-object structures. However, the subspaces $J_n\O(P)$ are necessarily ind-$\mcal C$-subobjects and the induced canonical isomorphism
\[
\gr^J_\bullet\O(P) \isoarrow \gr^J_\bullet\O(U)
\]
is necessarily an isomorphism of graded ind-objects of $\mcal C$. We will see this argument spelled out explicitly in proposition \ref{prop:J_n_is_sub-MHS}.

\subsection{Mal\u cev completion}

The pro-unipotent Betti and \'etale fundamental groups we will study in this paper can be thought of as certain pro-unipotent envelopes of the more usual discrete Betti or profinite \'etale fundamental groups. The exact definition of these pro-unipotent envelopes is best given via a universal property, following \cite[Section 3.3]{javier}.

\begin{definition}[Mal\u cev completion, {\cite[Definition 3.95]{javier}}]
\label{defn:malcev_completion}
Let $\Pi$ be a discrete group. Then the functor from pro-unipotent groups over $F$ to sets defined on objects $U$ by\[U\mapsto\Hom(\Pi,U(F))\]is representable by a pro-unipotent group $\Pi_F$ and homomorphism $\Pi\rightarrow \Pi_F(F)$. This group is called the \emph{Mal\u cev completion} of $\Pi$, and is functorial in $\Pi$. For an $F$-algebra $\Lambda$, we may write $\Pi(\Lambda)$ for $\Pi_F(\Lambda)$ where the meaning is clear.

More generally, if $\Pi$ is a topological group and $F$ a topological field of characteristic zero, the functor from pro-unipotent groups over $F$ to sets defined on objects $U$ by\[U\mapsto\Hom_\cts(\Pi,U(F))\]is representable by a pro-unipotent group $\Pi_F$ and continuous homomorphism $\Pi\rightarrow \Pi_F(F)$. This group is called the \emph{continuous Mal\u cev completion} of $\Pi$, and is functorial in $\Pi$.

There is an analogous notion of the (continuous) Mal\u cev completion of a $\Pi$-torsor \cite[Section 3.3.7]{javier}.
\end{definition}

\begin{example}\label{ex:fundamental_groups_as_malcev_completions}
\leavevmode
\begin{itemize}
	\item If $X$ is a path-connected, locally path-connected and locally simply connected topological space with a basepoint $x\in X$, then the $\Q$-pro-unipotent fundamental group of $(X,x)$ (the Tannaka group of the category of unipotent $\Q$-local systems on $X$) is the $\Q$-Mal\u cev completion of the usual topological fundamental group $\pi_1^\topp(X;x)$.
	\item If $X$ is a connected variety over an algebraically closed field and $x$ a geometric basepoint, the $\Q_\ell$-pro-unipotent \'etale fundamental group of $(X,x)$ (the Tannaka group of the category of unipotent \'etale $\Q_\ell$-local systems on $X$) is the continuous $\Q_\ell$-Mal\u cev completion of the profinite \'etale fundamental group $\pi_1^\et(X;x)$.
\end{itemize}
If $X$ is homotopy-equivalent to a finite CW complex (resp.\ is defined over an algebraically closed field of characteristic $0$), then this pro-unipotent group is finitely generated.
\end{example}

\begin{remark}
If $X$ is a connected variety over an algebraically closed subfield of $\C$ and $x$ is a $\C$-point of $X$, then there is a canonical isomorphism
\[
\pi_1^\et(X;x) \isoarrow \pi_1^\topp(X(\C);x)^\wedge
\]
between the profinite \'etale fundamental group of $X$ and the profinite completion of the topological fundamental group of $X(\C)$ (given its usual topology induced by that on $\C$). It follows that the $\Q_\ell$-pro-unipotent \'etale fundamental group of $X$ is canonically identified with the $\Q$-pro-unipotent (topological) fundamental group of $X(\C)$, base-changed to $\Q_\ell$.
\end{remark}

We will need the following properties of Mal\u cev completion.

\begin{proposition}\label{prop:properties_of_malcev}
\leavevmode
\begin{itemize}
	\item Let $\Pi$ be a finitely generated discrete group and $\Pi_n$ its maximal torsion-free $n$-step nilpotent quotient. Then the natural map
	\[
	\Pi_F \isoarrow \liminv(\Pi_n)_F
	\]
	is an isomorphism.
	\item Let
	\[
	\Pi'' \rightarrow \Pi \rightarrow \Pi'
	\]
	be an exact sequence of finitely generated nilpotent groups. Then the induced sequence
	\[
	\Pi''_F \rightarrow \Pi_F \rightarrow \Pi'_F
	\]
	on $F$-Mal\u cev completions is also exact.
	\item Let $\Pi$ be a finitely generated discrete group. Then its $F$-Mal\u cev completion $\Pi_F$ is canonically isomorphic to $\Spec(H)$, where
	\[
	H := \limdir\left(\Hom\left(F\Pi/J^{n+1},F\right)\right)
	\]
	with $J\unlhd F\Pi$ the augmentation ideal of the group algebra $F\Pi$.
	\item Let $\Pi$ be a finitely generated discrete group and $F'/F$ a field extension. Then there is a canonical isomorphism $\Pi_{F'}\isoarrow(\Pi_F)_{F'}$.
\end{itemize}
\end{proposition}

\section{The main theorem ($\ell$-adic version)}
\label{c:main_theorem_l-adic}

Let us begin by giving the proof of theorem \ref{thm:main_theorem_l-adic}. Although the details of the proof will present few difficulties for us, variations of the arguments outlined here will reappear in \S\ref{c:main_theorem_p-adic} to prove the $p$-adic analogue of theorem \ref{thm:main_theorem_l-adic}.

For the remainder of this section let us fix an abelian variety $A/K$ over our fixed $p$-adic field $K$ and a prime number $\ell\neq p$. We will denote by $(L,\tilde 0)$ a pair of a varying line bundle $L/A$ and a basepoint $\tilde 0\in L^\times(K)$ in the complement $L^\times=L\setminus0$ of the zero section lying over $0\in A(K)$, and by $U/\Q_\ell$ the $\Q_\ell$-pro-unipotent \'etale fundamental group of $L^\times$ based at $\tilde 0$.

If we pick an embedding $K\hookrightarrow\C$, then $L^\times(\C)$ is an oriented $\C^\times$-bundle over the torus $A(\C)$. The homotopy exact sequence associated to a fibration ensures that\[1\rightarrow\pi_1^\topp(\C^\times;1)\centarrow\pi_1^\topp(L^\times(\C);\tilde 0)\rightarrow\pi_1^\topp(A(\C);0)\rightarrow1\]is a central extension of an abelian group by an abelian group. By the comparison in example \ref{ex:fundamental_groups_as_malcev_completions} and proposition \ref{prop:properties_of_malcev}, the corresponding sequence
\begin{equation}\label{eq:pi_1_of_L}1\rightarrow\Q_\ell(1)\centarrow U\rightarrow V_\ell A\rightarrow1
\end{equation}
on \'etale fundamental groups arising from the fibration sequence $(\G_m,1)\hookrightarrow(L^\times,\tilde 0)\twoheadrightarrow(A,0)$ can be identified (after a suitable choice of a embedding $\bar K\hookrightarrow\C$) with the $\Q_\ell$-Mal\u cev completion of the first sequence. Hence, by proposition \ref{prop:properties_of_malcev}, sequence \ref{eq:pi_1_of_L} is also a central extension and, in particular, $U$ is a $2$-step unipotent group.

Now, since all the unipotent groups appearing in sequence \ref{eq:pi_1_of_L} are isomorphic to affine space as varieties, it follows that $\Q_\ell(1)\subseteq U(\Q_\ell)$ has the subspace topology and $U(\Q_\ell)\twoheadrightarrow V_\ell A$ has a continuous splitting, so that we obtain (part of) a long exact sequence\begin{equation}\label{eq:l-adic_seq}\H^0(G_K,V_\ell A)\rightarrow\H^1(G_K,\Q_\ell(1))\actsarrow\H^1(G_K,U(\Q_\ell))\rightarrow\H^1(G_K,V_\ell A)\end{equation}in continuous Galois cohomology (see \S\ref{c:reps}). The leftmost term vanishes since $A(K)_\tors$ is finite, and the rightmost term vanishes by the Euler--Poincar\'e characteristic formula. Hence the map $\H^1(G_K,\Q_\ell(1))\rightarrow\H^1(G_K,U(\Q_\ell))$ is bijective, so that we have a well-defined composite map\[\lambda_L\colon L^\times(K)\rightarrow\H^1(G_K,U(\Q_\ell))\leftisoarrow\H^1(G_K,\Q_\ell(1))\isoarrow\Q_\ell\]where the final isomorphism arises from Kummer theory.

It remains to show that the $\lambda_L$ are the N\'eron log-metrics on our varying line bundle $L$ (and in particular that they are $\Q$-valued). In order to do this, we just need to verify that these functions satisfy a certain list of properties uniquely characterising the N\'eron log-metric, analogous to those characterising N\'eron functions \cite[Theorem 11.1.1]{lang}. Since the functions we are interested in are not taking values in $\R$, the exact conditions are a little delicate.

\begin{lemma}\label{lem:neron_log-metrics}
Let $A/K$ be an abelian variety over our fixed $p$-adic field $K$. Then for any pair $(L,\tilde0)$ consisting of a line bundle $L/A$ and a basepoint $\tilde0\in L^\times(K)$ lying over $0\in A(K)$, the N\'eron log-metric $\lambda_L\colon L^\times(\bar K)\rightarrow\R$ (see \S\ref{sss:neron_log-metrics}) is $\Q$-valued, and the restricted functions\[\lambda_L\colon L^\times(K)\rightarrow\Q\]are uniquely characterised by the following properties:
\begin{enumerate}\setcounter{enumi}{-1}
	\item $\lambda_L$ only depends on $(L,\tilde0)$ up to isomorphism;
	\item\label{condn:local_constancy} $\lambda_L$ is locally constant (for the $p$-adic topology on $L^\times(K)$);
	\item\label{condn:additivity} $\lambda_{L_1\otimes L_2}(P_1\otimes P_2)=\lambda_{L_1}(P_1)+\lambda_{L_2}(P_2)$;
	\item\label{condn:pullbacks} $\lambda_{[2]^* L}(P)=\lambda_L([2]P)$, where by abuse of notation we also denote by $[2]$ the topmost map in the pullback square
	\begin{center}
	\begin{tikzcd}
	{[2]}^*L^\times \arrow[r]\arrow[d] & L^\times \arrow[d] \\
	A \arrow[r,"{[2]}"] & A;
	\end{tikzcd}
	\end{center}
	\item\label{condn:trivial_bundle} when $L=A\times_K\A^1$ is the trivial line bundle, $\lambda_L$ is the composite\[L^\times(K)=A(K)\times K^\times\rightarrow K^\times\overset v\rightarrow\Q;\text{ and}\]
	\item\label{condn:normalisation} $\lambda_L(\tilde 0)=0$.
\end{enumerate}
More strongly, a family $(\lambda_L)_{(L,\tilde0)}$ of functions valued in a $\Q$-algebra $F$ satisfying the above conditions with $\Q$ replaced by $F$ is automatically $\Q$-valued, and agrees with the restriction of the N\'eron log-metric to $L^\times(K)$.
\begin{proof}
The fact that the N\'eron log-metric is $\Q$-valued and locally constant follows from the corresponding result for N\'eron functions \cite[Theorems 11.5.1 \& 11.5.2]{lang}. The properties of the N\'eron log-metric then follow from the corresponding properties of N\'eron functions \cite[Theorem 11.1.1]{lang}.

Conversely, if $(\lambda_L')_{(L,\tilde0)}$ is another such family valued in a $\Q$-algebra $F$, then by conditions \ref{condn:additivity} and \ref{condn:trivial_bundle} we have $\lambda_L'(\alpha z)=\lambda_L'(z)+v(\alpha)$ for all $\alpha\in K^\times$ and $z\in L^\times(K)$, so that the function $\lambda_L-\lambda_L'$ is constant on the fibres of $L^\times(K)\twoheadrightarrow A(K)$ and hence factors through a unique locally constant function $\delta_L\colon A(K)\rightarrow F$. The assignment $(L,\tilde0)\mapsto\delta_L$ depends only on $L$, and is a group homomorphism $\Pic(A)(K)\rightarrow\mathrm{Map}(A(K),F)$ by condition \ref{condn:additivity}.

Now if $L$ is symmetric (so $[2]^*L\simeq L^{\otimes4}$), then $\delta_L([2]P)=4\delta_L(P)$ for all $P\in A(K)$ by condition \ref{condn:pullbacks}. Thus the image of $\delta_L$ is closed under multiplication by $4$. Since it is also finite, this is only possible if $\delta_L=0$. A similar argument shows that $\delta_L=0$ also when $L$ is antisymmetric, and hence in general by writing $L^{\otimes2}$ for a general $L$ as a product of a symmetric and an antisymmetric part. This shows that $\lambda_L'=\lambda_L$ as desired.
\end{proof}
\end{lemma}

Let us now verify that the conditions of lemma \ref{lem:neron_log-metrics} are satisfied for the functions $\lambda_L\colon L^\times(K)\rightarrow\Q_\ell$ constructed above -- for instance, condition \ref{condn:normalisation} is immediate, since $\lambda_L$ is defined as a composite of basepoint-preserving maps. In fact, with one exception (condition \ref{condn:local_constancy}), the verifications will be entirely formal.

For condition \ref{condn:trivial_bundle}, we note that the inclusion $\G_m\hookrightarrow L^\times=\G_m\times_KA$ has a section given by projection, and hence the inverse to the isomorphism $\H^1(G_K,\Q_\ell(1))\isoarrow\H^1(G_K,U(\Q_\ell))$ is the map induced by projection. Naturality of the non-abelian Kummer map then gives a commuting diagram
\begin{center}
\begin{tikzcd}
K^\times\times A(K) \arrow[r,"\proj"]\arrow[d] & K^\times \arrow[r,"v"]\arrow[d] & \Q_\ell \arrow[d,equals] \\
\H^1(G_K,U(\Q_\ell)) \arrow[r,"\sim"] & \H^1(G_K,\Q_\ell(1)) \arrow[r,"\sim"] & \Q_\ell
\end{tikzcd}
\end{center}
which shows that $\lambda_L(\alpha,P)=v(\alpha)$ in this case, as desired.

Condition \ref{condn:pullbacks} also follows from naturality of the non-abelian Kummer map, the proof being entirely contained in the commuting diagram
\begin{center}
\begin{tikzcd}
\lbrack2\rbrack^*L^\times(K) \arrow[r]\arrow[d,"\lbrack2\rbrack"] & \H^1(G_K,U'(\Q_\ell)) \arrow[d,"\lbrack2\rbrack_*"] & \H^1(G_K,\Q_\ell(1)) \arrow[l,swap,"\sim"]\arrow[r,"\sim"]\arrow[d,equals] & \Q_\ell \arrow[d,equals] \\
L^\times(K) \arrow[r] & \H^1(G_K,U(\Q_\ell)) & \H^1(G_K,\Q_\ell(1)) \arrow[l,swap,"\sim"]\arrow[r,"\sim"] & \Q_\ell,
\end{tikzcd}
\end{center}
where $U'/\Q_\ell$ is the unipotent fundamental group of $[2]^*L^\times$.

Of the formal conditions, the only mildly difficult verification is that of condition \ref{condn:additivity}, for which we consider the pointed $\G_m\times_K\G_m$-torsor $L_1^\times\times_AL_2^\times$ on $A$. Writing $U_{1,2}$ for the fundamental group of $L_1^\times\times_A L_2^\times$ based at $(\tilde 0_1,\tilde 0_2)$, the same argument as earlier establishes that the inclusion of the fibre over $0$ induces an isomorphism\[\H^1(G_K,\Q_\ell(1)\times\Q_\ell(1)))\isoarrow\H^1(G_K,U_{1,2}(\Q_\ell)),\]and hence we have a well-defined map\[\lambda_{L_1,L_2}\colon(L_1^\times\times_AL_2^\times)(K)\rightarrow\H^1(G_K,U_{1,2}(\Q_\ell))\leftisoarrow\H^1(G_K,\Q_\ell(1)\times\Q_\ell(1))\isoarrow\Q_\ell\oplus\Q_\ell.\]

Now the tensor-multiplication map $\otimes\colon L_1^\times\times_AL_2^\times\rightarrow(L_1\otimes L_2)^\times$ restricts to the codiagonal (multiplication) map $\G_m\times_K\G_m\rightarrow\G_m$ on the fibres over $0\in A(K)$, and hence we have a commuting diagram
\begin{center}
\begin{tikzcd}[column sep=small]
(L_1^\times\times_AL_2^\times)(K) \arrow[r]\arrow[d,"\otimes"] & \H^1(G_K,U_{1,2}(\Q_\ell)) \arrow[d,"\otimes_*"] & \H^1(G_K,\Q_\ell(1)\times\Q_\ell(1)) \arrow[l,swap,"\sim"]\arrow[r,"\sim"]\arrow[d,"\codiag_*"] & \Q_\ell^{\oplus2} \arrow[d,"\codiag"] \\
(L_1\otimes L_2)^\times(K) \arrow[r] & \H^1(G_K,U_{1\otimes2}(\Q_\ell)) & \H^1(G_K,\Q_\ell(1)) \arrow[l,swap,"\sim"]\arrow[r,"\sim"] & \Q_\ell
\end{tikzcd}
\end{center}
where $U_{1\otimes2}/\Q_\ell$ is the unipotent fundamental group of $(L_1\otimes L_2)^\times$ based at $\tilde 0_1\otimes\tilde 0_2$. Reading the composites around the outside, we then see that $\codiag\circ\lambda_{L_1,L_2}=\lambda_{L_1\otimes L_2}\circ\otimes$.

Repeating the same argument with the projections $\proj_i\colon L_1^\times\times_AL_2^\times\rightarrow L_i^\times$ for $i=1,2$ , one sees similarly that $\proj_i\circ\lambda_{L_1,L_2}=\lambda_{L_i}\circ\proj_i$. Since the codiagonal map is the pointwise sum of the projections, we thus obtain that\[\lambda_{L_1\otimes L_2}\circ\otimes=\lambda_{L_1}\circ\proj_1+\lambda_{L_2}\circ\proj_2\]as functions on $(L_1^\times\times_AL_2^\times)(K)$, which is exactly the desired equality.

To conclude the proof of theorem \ref{thm:main_theorem_l-adic}, it just remains to check condition \ref{condn:local_constancy}, i.e.\ that the $\lambda_L$ are locally constant. This is in fact a general phenomenon, see \cite[Theorem~1.1]{me:local_constancy}.
\section{Representations on pro-unipotent groups}
\label{c:reps}

In this section, we will explain and explore the basic theory of \emph{(continuous) representations of $G_K$ on finitely generated pro-unipotent groups $U/\Q_\ell$}, which will form a mildly non-abelian generalisation of the notion of (continuous, finite-dimensional) representations of $G_K$ on $\Q_\ell$-vector spaces. For the time being, $\ell$ here is any prime, possibly equal to $p$. Just as $\ell$-adic linear Galois representations arise from geometry as the \'etale cohomology of $K$-varieties with coefficients in $\Q_\ell$, the principal source of Galois representations on pro-unipotent groups will be the pro-unipotent fundamental groups of $K$-varieties at $K$-rational basepoints.

However, just as the theory of linear Galois representations can be developed independently of their desired geometric source, we will now investigate these representations in the abstract. In this section in particular, we will be interested in extending the usual definitions of various classes of linear representations to the unipotent context. In extending these definitions, we will typically give three equivalent characterisations of the property of interest: one in terms of the group scheme $U$; one in terms of the Lie algebra $\Lie(U)$; and one in terms of the Hopf algebra $\O(U)$. The archetypal example of this is the following.

\begin{definition-lemma}\label{def-lem:unipotent_representation}
Let $U/\Q_\ell$ be a finitely generated pro-unipotent group endowed with an action of a topological group $G$ by algebraic group automorphisms. Then the conilpotency filtration (see definition \ref{def:J-filtration}) on $\O(U)$ is $G$-stable.

We shall say that $U$ is a \emph{(continuous)\footnote{We will omit the word ``continuous'' throughout this article, following from the standard practice in discussing linear $\ell$-adic representations.} representation of $G$ on a finitely generated pro-unipotent group} just when the following equivalent conditions hold:
\begin{enumerate}
	\item the action of $G$ on $U(\Q_\ell)$ is continuous;
	\item the action of $G$ on $\Lie(U)$ is continuous;
	\item the action of $G$ on each $J_n\O(U)$ is continuous.
\end{enumerate}
In this, the topology on $U(\Q_\ell)$ is the natural one\footnote{The natural topology on the $\Q_\ell$-points of an affine $\Q_\ell$-scheme $U$ is the subspace topology on $U(\Q_\ell)\subseteq\Q_\ell^n$ arising from embedding $U$ as a closed subscheme of an affine space $\A^n$, perhaps of infinite dimension. This is independent of the choice of embedding. When $U$ is a pro-unipotent group, this embedding of course can be chosen to be the isomorphism between $U$ and its Lie algebra.} induced by the topology on $\Q_\ell$.
\begin{proof}[Proof of equivalence]
$G$-invariance of the conilpotency filtration follows from its definition in terms of kernels of $G$-equivariant maps.

The logarithm isomorphism $U\isoarrow\Lie(U)$ (of $\Q_\ell$-schemes) ensures that $U(\Q_\ell)$ is homeomorphic to $\Lie(U)$ compatibly with the $G$-action, which gives the equivalence (1)$\Leftrightarrow$(2).

For the equivalence (2)$\Leftrightarrow$(3), we note that each condition holds for $U$ iff it holds for all of its finite-dimensional $G$-equivariant quotients, so that it suffices to prove the result when $U$ is unipotent. In this case, the third condition is equivalent to the action on each $\left(J_n\O(U)\right)^\dual\cong\mcal U/J^{n+1}$ being continuous, where $\mcal U$ is the universal enveloping algebra of $\Lie(U)$ and $J$ is its augmentation ideal. Since $\mcal U$ is generated as an algebra by $\Lie(U)$, continuity of the action on $\Lie(U)$ clearly implies continuity of the action on each $J_n\O(U)$. Conversely, $\Lie(U)$ injects into $\left(J_n\O(U)\right)^\dual$ for $n>>0$, and hence continuity of the action on $J_n\O(U)$ implies continuity on $\Lie(U)$.
\end{proof}
\end{definition-lemma}

\begin{remark}
It follows immediately from the definition that any $G$-stable finitely generated subgroup or quotient of a representation of a topological group $G$ on a finitely generated pro-unipotent group is also a representation, i.e.\ is continuous in the above sense. It also follows that any representation on a finitely generated pro-unipotent group is an inverse limit of representations on unipotent groups, so we can often prove results about general representations by reducing to the finite-dimensional case.
\end{remark}

Of course, pro-unipotent fundamental groups are an example of such representations. This is generally well-known, but let us nonetheless briefly recall how the proof goes.

\begin{proposition}\label{prop:pi_1_finitely_generated}
Let $Y$ be a smooth connected variety over a characteristic $0$ field $F$, and $y\in Y(F)$. Then the $\Q_\ell$-pro-unipotent fundamental group of $Y$ based at $y$ is a representation of the absolute Galois group $G_F$ on a finitely generated pro-unipotent group in the sense of definition-lemma \ref{def-lem:unipotent_representation}.
\begin{proof}
There are two claims here: that the fundamental group is finitely generated; and that the Galois action is continuous. This latter claim follows from \cite[Section 13.4]{deligne} by taking an inverse limit.

For the former claim, note that by the Lefschetz principle, it suffices to prove the result for $F=\C$, where by the comparison isomorphism it suffices to prove that the usual topological fundamental group is finitely generated. Moreover, by the Seifert--van Kampen theorem, it suffices to prove this for quasi-affine varieties. However, by \cite[Semi-algebraic triangulation theorem]{hironaka} (see also \cite[Remark 1.10]{hironaka}), any algebraic subset of $\C^n=\R^{2n}$ is homeomorphic to the complement in a finite simplicial complex of a subcomplex, and hence has finitely generated fundamental group.
\end{proof}
\end{proposition}

\subsubsection{Serre twisting}
\label{ss:serre_twisting}

If $U/\Q_\ell$ is a representation of a topological group $G$ on a finitely generated pro-unipotent group, then we will very often be interested in the non-abelian continuous Galois cohomology set\[\H^1(G,U(\Q_\ell)),\]defined to be the quotient of the set of continuous cocycles $\alpha\colon G\rightarrow U(\Q_\ell)$ (maps satisfying $\alpha(gh)=\alpha(g)g(\alpha(h))$) modulo the right action of $U(\Q_\ell)$ given by\[(\alpha\cdot u)(g)=u^{-1}\alpha(g)g(u).\]When $U$ is non-abelian, this in general does not have a group structure, merely a distinguished point given by the class of the trivial cocycle. In particular, we will several times encounter the problem that the fibres of a map of pointed sets are not in general determined simply by its kernel -- for instance a map with trivial kernel can nonetheless fail to be injective.

To circumvent these difficulties, we will use \emph{Serre twists} \cite[Section I.5.3]{serre}, whose definition we now recall. Given a continuous cocycle $\alpha\colon G\rightarrow U(\Q_\ell)$, the \emph{Serre twist} ${}_\alpha U/\Q_\ell$ of $U$ by $\alpha$ is defined to be the representation of $G$ on the same underlying pro-unipotent group, whose action on $\Lambda$-points for some $\Q_\ell$-algebra is given by\[g\colon u\mapsto\alpha(g)g(u)\alpha(g)^{-1}.\]This action is clearly algebraic (i.e.\ natural in $\Lambda$) and continuous on $\Q_\ell$-points, since $\alpha$ was chosen continuous.

The main basic result regarding Serre twists is the following, which allows one in practice to prove results about all fibres of maps out of Galois cohomology by just considering kernels.

\begin{proposition}\label{prop:serre_twisting_bijection}\cite[Proposition I.35 bis]{serre}\footnote{Strictly, Serre only considers the case when the action of $G$ on $U(\Q_\ell)$ is continuous for the \emph{discrete} topology and the cohomology considered is non-continuous cohomology. Nonetheless, the same argument works.}
Let $U/\Q_\ell$ be a representation of a topological group $G$ on a finitely generated pro-unipotent group, and $\alpha\colon G\rightarrow U(\Q_\ell)$ a continuous cocycle. Then right-multiplication by $\alpha$ induces a bijection\[\H^1(G,{}_\alpha U(\Q_\ell))\isoarrow\H^1(G,U(\Q_\ell))\]taking the distinguished point of $\H^1(G,{}_\alpha U(\Q_\ell))$ to the class $[\alpha]\in\H^1(G,U(\Q_\ell))$.
\end{proposition}

As an example, let us now use the notion of Serre twists to prove a lemma allowing one to reduce study of the Galois cohomology of representations on finitely generated pro-unipotent groups to the finite-dimensional case.

\begin{lemma}\label{lem:cohomology_of_inverse_limits}
Let $U/\Q_\ell$ be a representation of a topological group $G$ on a finitely generated pro-unipotent group, and write $U=\liminv U_n$ as an inverse limit of finite-dimensional $G$-equivariant quotients. Then the natural map\[\H^1(G,U(\Q_\ell))\rightarrow\liminv\H^1(G,U_n(\Q_\ell))\]is bijective.
\begin{proof}
We may assume the indexing set for the inverse system is $\N$.

Let us begin by showing surjectivity. To do this, choose an element $([\beta_n])_{n\in\N}\in\liminv\H^1(G,U_n(\Q_\ell))$, and write $S_n$ for the set of continuous cocycles $\alpha_n\colon G\rightarrow U_n(\Q_\ell)$ whose class is $[\beta_n]$. If $m\geq n$ then $S_m\twoheadrightarrow S_n$ is surjective, for, if $\alpha_n\in S_n$ is any element, we may pick some $\alpha_m'\in S_m$, so that the image of $\alpha_m'$ in $S_n$ differs from $\alpha_n$ by the right action of some $u_n\in U_n(\Q_\ell)$. Lifting $u_n$ to some $u_m\in U_m(\Q_\ell)$ and replacing $\alpha_m'$ with $\alpha_m=\alpha_m'\cdot u_m$, we obtain a lift of $\alpha_n$ to $S_m$. In particular, $\liminv S_n$ is non-empty. But any element $\alpha\in\liminv S_n$ can be thought of as a continuous cocycle $\alpha\colon G\rightarrow U(\Q_\ell)$, and the class $[\alpha]\in\H^1(G,U(\Q_\ell))$ maps by construction to $[\beta_n]$ in each $\H^1(G,U_n(\Q_\ell))$, proving surjectivity.

Let us show next that $\H^1(G,U(\Q_\ell))\rightarrow\liminv\H^1(G,U_n(\Q_\ell))$ has trivial kernel. Let $\alpha\colon G\rightarrow U(\Q_\ell)$ be a cocycle representing an element of the kernel, and for each $n$ let $T_n\subseteq U_n(\Q_\ell)$ denote the set of elements whose coboundary is the image of $\alpha$ in $U_n(\Q_\ell)$. Each $T_n$ is non-empty by assumption, and is a left torsor under $U_n(\Q_\ell)^G$ for the multiplication action. Since $U_n(\Q_\ell)^G$ is the $\Q_\ell$-points of a unipotent subgroup of $U_n$, it follows that the images of $U_m(\Q_\ell)^G\rightarrow U_n(\Q_\ell)^G$ stabilise for $m>>n$, i.e.\ the system $(U_n(\Q_\ell)^G)_{n\in\N}$ satisfies the Mittag--Leffler condition. Hence the system $(T_n)_{n\in\N}$ of non-empty sets also satisfies the Mittag--Leffler condition, and in particular $\liminv T_n$ is non-empty. But by construction the coboundary of any element of $\liminv T_n\subseteq U(\Q_\ell)$ is $\alpha$, so that $\alpha$ is a trivial cocycle as desired.

It remains to show that $\H^1(G,U(\Q_\ell))\rightarrow\liminv\H^1(G,U_n(\Q_\ell))$ is actually injective. If we pick any cocycle $\alpha\colon G\rightarrow U(\Q_\ell)$, and write $\alpha_n$ for its images in each $U_n(\Q_\ell)$, then Serre twisting provides us a commuting square
\begin{center}
\begin{tikzcd}
\H^1(G,{}_\alpha U(\Q_\ell)) \arrow[r]\arrow[d,"\wr"] & \liminv\H^1(G,{}_{\alpha_n}U_n(\Q_\ell)) \arrow[d,"\wr"] \\
\H^1(G,U(\Q_\ell)) \arrow[r] & \liminv\H^1(G,U_n(\Q_\ell)).
\end{tikzcd}
\end{center}
We've shown that the top map has trivial kernel, so that $[\alpha]$ is the unique point in its fibre under the map $\H^1(G,U(\Q_\ell))\rightarrow\liminv\H^1(G,U_n(\Q_\ell))$. Since $\alpha$ was arbitrary, this shows that the map is injective, as desired.
\end{proof}
\end{lemma}

\subsection{$\ell$-adic Galois representations}

Before we embark on our analysis of representations of $G_K$ on finitely generated pro-unipotent groups over $\Q_p$, let us first examine the much simpler case of groups over $\Q_\ell$, where now $\ell$ is some prime different from $p$. All the results proven here will have their $p$-adic analogues in the following sections, and should serve to illustrate the methods we will employ in developing the basic theory.

The following definition describes how to generalise the basic notions from the theory of $\Q_\ell$-linear representations to the finitely generated pro-unipotent setting, as usual phrased in three equivalent ways.

\begin{definition-lemma}\label{def-lem:unramified_rep}
Let $U/\Q_\ell$ be a representation of $G_K$ on a finitely generated pro-unipotent group, where $\ell\neq p$. We shall say that $U$ is \emph{semistable} (resp.\ \emph{unramified}) just when the following equivalent conditions hold:
\begin{enumerate}
	\item $U$ has a (separated) $G_K$-stable filtration $U=U_1\unrhd U_2\unrhd\dots$ by subgroup-schemes with each $U_n/U_{n+1}$ abelian of finite type and the action of inertia on $U_n/U_{n+1}$ being unipotent (resp.\ the action of inertia on $U$ is trivial);
	\item $\Lie(U)$ is pro-semistable (resp.\ unramified), i.e.\ all of its finite-dimensional $G_K$-equivariant quotients are semistable (resp.\ unramified) representations;
	\item each $J_n\O(U)$ is semistable (resp.\ unramified).
\end{enumerate}
\begin{proof}[Proof of equivalence]
The equivalence of the different definitions of unramifiedness is clear, so we only deal with the case of semistability.

For the equivalence (1)$\Leftrightarrow$(2), consider any (separated) $G_K$-stable filtration of $U$ by finite-index subgroups with $U_n/U_{n+1}$ abelian, such as for instance the descending central series filtration. The action on $\Lie(U)$ is pro-semistable iff the action on each $\Lie(U)/\Lie(U_n)$ is semistable, which occurs iff the action on each $\Lie(U_n)/\Lie(U_{n+1})\cong U_n/U_{n+1}$ is semistable, as desired.

For the equivalence (2)$\Leftrightarrow$(3), we note that the third point is equivalent to ind-semistability of $\O(U)$, and hence each condition holds of $U$ iff it holds for all of its finite-dimensional $G_K$-equivariant quotients, so that it suffices to prove the result when $U$ is unipotent. Again, semistability of $J_n\O(U)$ is equivalent to semistability of $\mcal U/J^{n+1}$, where $\mcal U$ is the universal enveloping algebra of $\Lie(U)$. Since $\Lie(U)$ injects into $\mcal U/J^{n+1}$ for $n>>0$ and $\mcal U$ is a quotient of the tensor algebra $T^\otimes\Lie(U)$, the equivalence (2)$\Leftrightarrow$(3) follows.
\end{proof}
\end{definition-lemma}

\begin{remark}
It follows directly from the definitions that $G_K$-stable finitely generated pro-unipotent subgroups and quotients of semistable (resp.\ unramified) representations are again semistable (resp.\ unramified), and that a representation of $G_K$ on a finitely generated pro-unipotent group is semistable (resp.\ unramified) iff all of its finite-dimensional quotients are. Moreover, the class of semistable representations is closed under central extensions.
\end{remark}

Using this definition, we can prove the somewhat surprising result that semistability depends only on the abelianisation of a finitely generated pro-unipotent group, deducing as a corollary that Grothendieck's $\ell$-adic monodromy theorem holds in this setting. Note that one can deduce this immediately for unipotent groups by considering the Lie algebra, but the infinite-dimensional version genuinely does require reduction to the abelianisation.

\begin{lemma}\label{lem:semistability_via_abelianisation_l-adic}
Let $U/\Q_\ell$ be a representation of $G_K$ on a finitely generated pro-unipotent group, where $\ell\neq p$. Then $U$ is semistable iff $U^\ab$ is.
\begin{proof}
The ``only if'' direction is trivial, so let us concentrate on the converse implication. Let $U=U_1\unrhd U_2\unrhd\dots$ denote the descending central series, so that $U_n$ are finite-index subgroups with $U_n/U_{n+1}$ abelian. $U_1/U_2$ is then the abelianisation of $U$, and we have surjective iterated commutator maps\[\left(U^\ab\right)^{\otimes n}\twoheadrightarrow U_n/U_{n+1}.\]Since by assumption $U^\ab$ is an abelian semistable representation, so too is each $U_n/U_{n+1}$. Thus by definition-lemma \ref{def-lem:unramified_rep} $U$ is semistable.
\end{proof}
\end{lemma}

\begin{corollary}[$\ell$-adic monodromy theorem for pro-unipotent representations]\label{cor:potential_semistability_l-adic}
Let $U/\Q_\ell$ be a representation of $G_K$ on a finitely generated pro-unipotent group, where $\ell\neq p$. Then $U$ is potentially semistable (i.e.\ becomes semistable after restriction to an open subgroup).
\begin{proof}
This follows immediately from Grothendieck's $\ell$-adic monodromy theorem, applied to the finite-dimensional $U^\ab$.
\end{proof}
\end{corollary}

Let us conclude this section by making good on our promise in remark \ref{rmk:what_about_f/e} and showing that the set of relatively unramified torsors under the $\Q_\ell$-unipotent fundamental group of $L^\times=L\setminus0$ is trivial, for $L$ a line bundle on an abelian variety $A/K$. In the proof, we will need the following easily-proved non-abelian analogue of the inflation-restriction sequence.

\begin{proposition}[Non-abelian inflation-restriction]
Let $G$ be a profinite group acting continuously on a topological group $U$, and $I\unlhd G$ a closed normal subgroup. Then we have an inflation-restriction exact sequence\[1\rightarrow\H^1(G/I,U^I)\rightarrow\H^1(G,U)\rightarrow\H^1(I,U).\]
\end{proposition}

\begin{proposition}\label{prop:H^1_nr=1}
Let $U/\Q_\ell$ be a representation of $G_K$ on a unipotent group, where $\ell\neq p$. Then\[\H^1_\nr(G_K,U(\Q_\ell)):=\H^1(G_K/I_K,U(\Q_\ell)^{I_K})=\ker\left(\H^1(G_K,U(\Q_\ell))\rightarrow\H^1(I_K,U(\Q_\ell))\right)\]is trivial iff $U(\Q_\ell)^{G_K}=1$.
\begin{proof}
Since $G_K/I_K$ is pro-cyclic, generated by Frobenius $\varphi$, and $U(\Q_\ell)^{I_K}$ is a direct limit of profinite groups, we know that $\H^1(G_K/I_K,U(\Q_\ell)^{I_K})$ is the orbit-space of the twisted conjugation action $u\colon w\mapsto u^{-1}w\varphi(u)$. Hence $\H^1_\nr(G_K,U(\Q_\ell))=1$ iff this self-action of $U(\Q_\ell)^{I_K}$ is transitive, and $U(\Q_\ell)^{G_K}=1$ iff the stabiliser of the basepoint is trivial.

In proving the equivalence of these conditions, we are free to replace $U$ by $U^{I_K}$, so we may and will assume that inertia acts trivially on $U$.

We now proceed by induction. When $U$ is abelian, the twisted conjugation is just the translation action along the endomorphism $\varphi-1$ of $U(\Q_\ell)$. The orbit space is then the cokernel of $\varphi-1$ while the point-stabiliser is the kernel of $\varphi-1$; either of these vanishes iff the other does by dimension considerations.

In general, we write $U$ as a $G_K/I_K$-equivariant central extension\[1\rightarrow Z\centarrow U\rightarrow Q\rightarrow1\]where we already know the result for $Z$ and $Q$. If on the one hand the self-action of $U(\Q_\ell)$ is transitive, then so too is the self-action of $Q(\Q_\ell)$, so that by induction the latter action has trivial point-stabiliser. This forces the self-action of $Z(\Q_\ell)$ to also be transitive, and hence also have trivial point-stabiliser. The triviality of the point-stabilisers for the self-actions of $Q(\Q_\ell)$ and $Z(\Q_\ell)$ then implies the same for $U(\Q_\ell)$, as desired.

If, on the other hand, the point-stabiliser for the self-action of $U(\Q_\ell)$ is trivial, then the same is true for the self-action of $Z(\Q_\ell)$, which is thus transitive. This then forces the self-action of $Q(\Q_\ell)$ to have trivial point-stabiliser, so again by induction it is transitive, and combining transitivity for $Q(\Q_\ell)$ and $Z(\Q_\ell)$ we obtain transitivity for $U(\Q_\ell)$, as desired.
\end{proof}
\end{proposition}

\begin{remark}\label{rmk:H^1_nr=1}
The condition that $U(\Q_\ell)^{G_K}=1$ holds whenever $U$ is the $\Q_\ell$-pro-unipotent fundamental group of a smooth geometrically connected $K$-variety based at a $K$-rational point; this follows from the fact that $U$ possesses a weight filtration whose graded pieces are abelian representations satisfying a certain weight--monodromy property with negative weight \cite[Theorem~1.3(1)]{me-daniel:weight-monodromy}\footnote{It seems that this weight--monodromy result for the fundamental group was unknown in general at the time this paper was first written, where it was formulated as a conjecture. So the chronology is backwards here: it was this paper which in part inspired \cite{me-daniel:weight-monodromy}.}. For the fundamental groups studied in this paper, this is more classical: when $U$ is the fundamental group of a $\G_m$-torsor $L^\times$ over an abelian variety $A$, then $U$ is a central extension of $V_\ell A$ by $\Q_\ell(1)$, so the vanishing of $U(\Q_\ell)^{G_K}$ follows from the vanishing of $V_\ell A^{G_K}$ \cite[Corollaire IX.4.4]{SGA7.1}.
\end{remark}

\subsection{$p$-adic Galois representations, admissibility, and Dieudonn\'e functors}

Having sketched some of the basic theory of Galois representations on pro-unipotent groups over $\Q_\ell$ for $\ell\neq p$, let us now discuss the theory of such representations when $\ell=p$. As in the abelian case, the theory we obtain is much richer than the $\ell\neq p$ case, and will constitute the technical basis of our coming analysis of local Bloch--Kato Selmer sets.

The basic constructions and notions in abelian $p$-adic Hodge theory as developed by Fontaine in \cite{fontaine3} revolve around certain \emph{period rings} $\B$, which are topological $\Q_p$-algebras with a continuous $G_K$-action, and consideration of the $\B$-semilinear $G_K$-action on $\B\otimes_{\Q_p}V$ for a $p$-adic Galois representation of interest. Replacing $V$ by a finitely generated pro-unipotent group $U$, the correct analogue of $\B\otimes_{\Q_p}V$ is $U(\B)=\Hom_{\Q_p}(\Spec(\B),U)$, given its natural topology induced from that on $\B$ and its natural ``semilinear'' (continuous\footnote{One can see continuity, for instance, by reducing to the case of $U$ finite-dimensional and examining the action on $\Lie(U)(\B)\cong U(\B)$.}) $G_K$-action arising from those on $\B$ and $U$.

With this analogy in mind, it is now straightforward for us to define the \emph{Dieudonn\'e functor} $\D$ associated to a period ring $\B$, which will take representations on finitely generated pro-unipotent groups over $\Q_p$ to pro-unipotent groups over the fixed field $\B^{G_K}$. As usual, for the various period rings arising in practice, we will decorate the corresponding Dieudonn\'e functor with the same symbols as the period ring without comment -- for instance $\D_\dR^+$ will be the Dieudonn\'e functor associated to $\B_\dR^+$.

\begin{lemma}\label{lem:dieudonne_functor}
Let $\B/\Q_p$ be a $(\Q_p,G_K)$-regular $\Q_p$-algebra \cite[Definition 1.4.1]{fontaine3} with $G_K$-fixed field $F$ (or $\B$ a $G_K$-stable $F$-subalgebra of such an algebra), and let $U/\Q_p$ be a representation of $G_K$ on a finitely generated pro-unipotent group. Then the functor $\D(U)$ on $F$-algebras defined by\[\D(U)(\Lambda)=U(\B\otimes_F\Lambda)^{G_K}\]is a pro-unipotent group over $F$ with Lie algebra $\D(\Lie(U))$, and there is a natural injection\[\D(U)_\B\hookrightarrow U_\B\]between the base-changes to $\B$. In particular, if $U$ is finite-dimensional, then so is $\D(U)$ and we have the inequality $\dim_F\D(U)\leq\dim_{\Q_p}U$.
\begin{proof}
Since the formation of $U(\B)^{G_K}$ is compatible with inverse limits in $U$, we may reduce to the case that $U$ is finite-dimensional.

Using the isomorphism $\log\colon U\rightarrow\Lie(U)$, we obtain a natural isomorphism\[\D(U)(\Lambda)\cong\left(\Lie(U)\otimes_{\Q_p}\B\otimes_F\Lambda\right)^{G_K}\cong\left(\Lie(U)\otimes_{\Q_p}\B\right)^{G_K}\otimes_F\Lambda.\]But the right-hand side is the functor-of-points of the vector group associated to the finite-dimensional $F$-vector space $\D(\Lie(U))=\left(\Lie(U)\otimes_{\Q_p}\B\right)^{G_K}$. Hence $\D(U)/F$ is a variety, isomorphic to affine space.

Now the group operation on $U$ is induced from the Lie bracket on $\Lie(U)$ via the Baker--Campbell--Hausdorff formula. From this, it follows that the natural group operation on $\D(U)(\Lambda)$ is induced from the natural Lie bracket on $\D(\Lie(U))\otimes_F\Lambda$ also by the Baker--Campbell--Hausdorff formula, so that $\D(U)$ is unipotent with Lie algebra $\D(\Lie(U))$.

Finally, the natural map $U(\B\otimes_F\Lambda)^{G_K}\rightarrow U(\Lambda)$ induced by the multiplication map $\B\otimes_F\Lambda\rightarrow\Lambda$ for any $\B$-algebra $A$ induces a morphism\[\D(U)_\B\rightarrow U_\B\]which, identifying both sides with their Lie algebras, is the usual one from the theory of admissible representations, and hence injective. In particular, $\dim_F\D(U)\leq\dim_{\Q_p}U$.
\end{proof}
\end{lemma}

Just as in abelian $p$-adic Hodge theory, the Dieudonn\'e functor allows us to define certain \emph{admissible classes} of representations on pro-unipotent groups. In fact, the notion we get is essentially already encompassed by the notion of admissibility for linear $p$-adic representations, as in the following lemma.

\begin{definition-lemma}\label{def-lem:admissible_rep}
Let $\B/\Q_p$ be a $(\Q_p,G_K)$-regular $\Q_p$-algebra \cite[Section 1.4]{fontaine3} with $G_K$-fixed field $F$, and let $U/\Q_p$ be a representation of $G_K$ on a finitely generated pro-unipotent group. The following are equivalent:
\begin{enumerate}
	\item $\D(U)_\B\hookrightarrow U_\B$ is an isomorphism;
	\item[(1')] (if $U$ finite-dimensional) $\dim_F\D(U)=\dim_{\Q_p}U$;
	\item $\Lie(U)$ is pro-$\B$-admissible;
	\item each $J_n\O(U)$ is $\B$-admissible.
\end{enumerate}
When these equivalent conditions hold, we will say that $U$ is \emph{$\B$-admissible}. In the specific case when $\B=\B_\dR$ (resp.\ $\B=\B_\st$, resp.\ $\B=\B_\cris$), we will refer to $U$ as being \emph{de Rham} (resp.\ \emph{semistable}, resp.\ \emph{crystalline}).
\begin{proof}[Proof of equivalence]
For the equivalences (1)$\Leftrightarrow$(2)$\Leftrightarrow$(3), we may reduce to the case that $U$ is finite-dimensional, since the construction of $\D(U)$ and the map $\D(U)_\B\hookrightarrow U_\B$ commutes with inverse limits in $U$.

In this case, $\D(\Lie(U))$ is the Lie algebra of $\D(U)$ by lemma \ref{lem:dieudonne_functor}, so that we have $\dim_F(\D(U))=\dim_F\D(\Lie(U))$, which proves the equivalence (1')$\Leftrightarrow$(2).

Next, to prove the equivalence (1)$\Leftrightarrow$(2), we simply note from the proof of lemma \ref{lem:dieudonne_functor} that, via the usual logarithm maps, the injection $\D(U)(\Lambda)\hookrightarrow U(\Lambda)$ for a $\B$-algebra $\Lambda$ is identified with the base change to $\Lambda$ of the usual map $\B\otimes_F\D(\Lie(U))\hookrightarrow\B\otimes_{\Q_p}\Lie(U)$. Thus the one map is an isomorphism iff the other is.

Finally, for the equivalence (2)$\Leftrightarrow$(3), we note that condition (3) is equivalent to $\B$-admissibility of each $\mcal U/J^{n+1}$, where $\mcal U$ is the universal enveloping algebra of $\Lie(U)$ and $J$ is the augmentation ideal. Again, the fact that $\mcal U$ is a surjective image of the tensor algebra $T^\otimes\Lie(U)$ and that $\Lie(U)$ is a subrepresentation of $\mcal U/J^{n+1}$ for $n>>0$ justifies the equivalence (2)$\Leftrightarrow$(3).
\end{proof}
\end{definition-lemma}

\begin{remark}\label{rmk:affine_rings_of_admissible_representations}
When $U$ is $\B$-admissible, we've seen that $\O(U)$ is ind-$\B$-admissible, so that $\D(\O(U))$ is a Hopf algebra over $F=\B^{G_K}$. Indeed, the $\Lambda$-points of the corresponding affine group scheme are\begin{align*}\Hom_F(\D(\O(U)),\Lambda)^\gplike&=\Hom_\B(\B\otimes_F\D(\O(U)),\B\otimes_F\Lambda)^{G_K,\gplike}\\&=\Hom_\B(\B\otimes_{\Q_p}\O(U),\B\otimes_F\Lambda)^{\gplike,G_K}=U(\B\otimes_F\Lambda)^{G_K},\end{align*}so that the corresponding affine group scheme is canonically identified with $\D(U)$.

Moreover, it follows from definition \ref{def:J-filtration} and the fact that everything in sight is $\B$-admissible that in this case $J_n\D(\O(U))=\D(J_n\O(U))$, so that $\D(U)$ is finitely generated (see proposition \ref{prop:lubotzky-magid}).
\end{remark}

\begin{remark}
It follows directly from the definitions and \cite[Proposition 1.5.2]{fontaine3} that $G_K$-stable finitely generated subgroups and quotients of $\B$-admissible representations are again $\B$-admissible, and that a representation of $G_K$ on a finitely generated pro-unipotent group is $\B$-admissible iff all of its finite-dimensional quotients are.
\end{remark}

The equivalence between $\B$-admissibility of $U$ and $\Lie(U)$ means that we immediately deduce non-abelian versions of certain basic properties of admissible representations.

\begin{lemma}\label{lem:seqs_of_de_Rham_reps}
Let\[1\rightarrow Z\rightarrow U\rightarrow Q\rightarrow1\]be an exact sequence of de Rham representations of $G_K$ on finitely generated pro-unipotent groups over $\Q_p$. Then the sequence remains exact on taking $\D_\st$ and on taking $\D_\dR^+$.
\begin{proof}
Since the functors $\D_\st$ and $\D_\dR^+$ commute with inverse limits, it suffices to treat the case that $U$, $Q$ and $Z$ are finite dimensional. Since $\Lie(\D_\st(U))=\D_\st(\Lie(U))$ and similarly for $\D_\dR^+$, it suffices also to prove the corresponding statement for the Lie algebras of $U$, $Q$ and $Z$. But the Lie bracket plays no role in exactness of the sequences, so it suffices to prove the result when $U$, $Q$ and $Z$ are replaced by $\Q_p$-linear representations.

Now since the representations $U$, $Q$ and $Z$ are all potentially semistable \cite[Th\'eor\`eme 0.7]{berger}, exactness after applying $\D_\st$ is dealt with in the argument (but not the statement) of \cite[Lemme 6.4]{berger}.

To prove exactness after applying $\D_\dR^+$, we may choose by potential semistability an open normal subgroup $G_L\unlhd G_K$ such that $U$, $Q$ and $Z$ are semistable when the action is restricted to $G_L$. Since $\D_{\st,L}(U)$ etc.\ are admissible filtered $(\varphi,N)$-modules and morphisms of such are strict \cite[Th\'eor\`eme 5.3.5]{fontaine3}, we thus have that\[0\rightarrow\D_{\dR,L}^+(Z)\rightarrow\D_{\dR,L}^+(U)\rightarrow\D_{\dR,L}^+(Q)\rightarrow0\]is exact. Since all the terms are vector spaces over a characteristic zero field, the sequence remains exact after taking invariants under the finite group $G_{L|K}$, and hence the sequence on $\D_\dR^+$ is exact as desired.
\end{proof}
\end{lemma}

\begin{corollary}
If $U/\Q_p$ is a de Rham representation of $G_K$ on a finitely generated pro-unipotent group which is an extension of a two semistable representations (on finitely generated pro-unipotent groups), then $U$ itself is semistable.
\begin{proof}
Writing $U$ as an extension of $Q$ by $Z$, both semistable, we have from lemma \ref{lem:seqs_of_de_Rham_reps} a diagram
\begin{center}
\begin{tikzcd}
1 \arrow[r] & \D_\st(Z)_{\B_\st} \arrow[r]\arrow[d,"\wr"] & \D_\st(U)_{\B_\st} \arrow[r]\arrow[d] & \D_\st(Q)_{\B_\st} \arrow[r]\arrow[d,"\wr"] & 1 \\
1 \arrow[r] & Z_{\B_\st} \arrow[r] & U_{\B_\st} \arrow[r] & Q_{\B_\st} \arrow[r] & 1
\end{tikzcd}
\end{center}
of group schemes over $\B_\st$, with both rows exact (on $\Lambda$-points for all $\B_\st$-algebras $\Lambda$) and the outer vertical maps both isomorphisms. The five-lemma then shows that $\D_\st(U)_{\B_\st}\hookrightarrow U_{\B_\st}$ is an isomorphism, so that $U$ is semistable as desired.
\end{proof}
\end{corollary}

\subsubsection{Potential semistability}

Just as in the $\ell$-adic case, we can prove an analogue of Berger's theorem on potential semistability of de Rham representations in the setting of representations on finitely generated pro-unipotent groups. Again, the key result which allows us to reduce to the linear case is that semistability of a de Rham representation depends only on its abelianisation.

\begin{lemma}\label{lem:semistability_via_abelianisation_p-adic}
Let $U/\Q_p$ be a de Rham representation of $G_K$ on a finitely generated pro-unipotent group. Then $U$ is semistable iff $U^\ab$ is.
\begin{proof}
The same proof as lemma \ref{lem:semistability_via_abelianisation_l-adic} works, the key point being that everything in sight is de Rham, so that central extensions of semistable representations remain semistable.
\end{proof}
\end{lemma}

\begin{corollary}[$p$-adic monodromy theorem for pro-unipotent representations]\label{cor:potential_semistability_p-adic}
Let $U/\Q_p$ be a representation of $G_K$ on a finitely generated pro-unipotent group. Then $U$ is potentially semistable (i.e.\ becomes semistable after restriction to an open subgroup) iff $U$ is de Rham.
\begin{proof}
Potentially semistable representations are clearly de Rham. Conversely, if $U$ is de Rham, then by lemma \ref{lem:semistability_via_abelianisation_p-adic} it becomes semistable as soon as $U^\ab$ does. But $U^\ab$ is a finite-dimensional linear representation, so is potentially semistable by \cite[Th\'eor\`eme 0.7]{berger}.
\end{proof}
\end{corollary}

\subsubsection{Extra structures on Dieudonn\'e functors}
\label{ss:extra_structure}

Just as in the abelian case, the pro-unipotent groups $\D_\cris(U)$, $\D_\st(U)$ and $\D_\dR(U)$ carry various extra structures (Frobenius, monodromy, Hodge filtration) induced from the corresponding structures on the period rings $\B_\cris$, $\B_\st$ and $\B_\dR$. Since we will not require these extra structures in the main body of this paper, let us merely describe in outline how these structures manifest on $\D_\st(U)$, assuming for simplicity that $U$ is semistable.

Most obviously, since $\O(U)$ is ind-semistable, $\O(\D_\st(U))=\D_\st(\O(U))$ then carries the structure of an ind-weakly admissible filtered $(\varphi,N)$-module, compatible with the Hopf algebra structure morphisms. As usual, these extra structures (Frobenius, monodromy, Hodge filtration) on $\O(\D_\st(U))$ can be viewed equivalently as structures on $\Lie(\D_\st(U))$ or on $\D_\st(U)$ itself. For $\Lie(\D_\st(U))=\D_\st(\Lie(U))$ the same argument shows that it carries a pro-weakly admissible filtered $(\varphi,N)$-module structure compatible with the Lie bracket.

However, viewing these structures on $\D_\st(U)$ is more complicated, and indeed the Hodge filtration has no obvious direct interpretation (except as a structure on $\Lie(\D_\st(U))$ or $\O(\D_\st(U))$). The intepretation of Frobenius and monodromy on $\D_\st(U)$ is as follows.

On the level of $\D_\st(U)$, the semilinear Frobenius on $\O(\D_\st(U))$ induces a semilinear Frobenius $\varphi\colon\D_\st(U)\isoarrow\varphi^*\D_\st(U)$, where by abuse of notation we also denote by $\varphi$ the Frobenius on the maximal absolutely unramified subfield $K_0$ of $K$. The monodromy operator on $\O(\D_\st(U))$, which is a derivation compatible with the comultiplication, induces then a \emph{homomorphic} vector field on $\D_\st(U)$, i.e.\ a vector field which, when viewed as a section $\D_\st(U)\rightarrow T\D_\st(U)$ of the tangent bundle on $\D_\st(U)$, is a group homomorphism for the natural group structure on the total space $T\D_\st(U)=\Lie(\D_\st(U))\rtimes\D_\st(U)$ of the tangent bundle. The relation $N\varphi=p\varphi N$ between the Frobenius and monodromy on $\O(\D_\st(U))$ implies that this $\varphi$ and $N$ on $\D_\st(U)$ satisfy the relation $\varphi^*N\circ\varphi=p\cdot\dee\varphi\circ N$ as morphisms $\D_\st(U)\rightarrow\varphi^*T\D_\st(U)$, where the $p$ denotes multiplication by $p$ in the fibres of $\varphi^*T\D_\st(U)\twoheadrightarrow\varphi^*\D_\st(U)$.

\subsection{Cohomology of $p$-adic Galois representations}

Recall from the introduction that if $U/\Q_p$ is a representation of $G_K$ on a finitely generated pro-unipotent group, the \emph{local Bloch--Kato Selmer sets} are three naturally defined pointed subsets $\H^1_e(G_K,U(\Q_p))\subseteq\H^1_f(G_K,U(\Q_p))\subseteq\H^1_g(G_K,U(\Q_p))$ of the non-abelian Galois cohomology $\H^1(G_K,U(\Q_p))$ cut out by increasingly coarse Selmer conditions (see definition \ref{def:bloch-kato_sets}), which in the case that $U$ is the fundamental group of a smooth variety $Y/K$ should be closely tied to the geometry of $Y$. In order to prove theorem \ref{thm:main_theorem_p-adic}, we will need to conduct a detailed explicit study of these sets and their respective quotients, which will be carried out in \S\ref{c:bloch-kato} using tools from homotopical algebra.

Before we embark on this analysis, let us briefly record a few properties of the local Bloch--Kato Selmer sets and their quotients which are essentially immediate from the definitions.

\begin{lemma}[Stability of local Bloch--Kato sets under base change]\label{lem:bloch-kato_stable_under_base_change}
Let $U/\Q_p$ be a representation of $G_K$ on a finitely generated pro-unipotent group, and let $L/K$ be a finite extension (embedded in the fixed algebraic closure $\bar K/K$). Then the restriction map\[\H^1(G_K,U(\Q_p))\rightarrow\H^1(G_L,U(\Q_p))\]is injective, and the preimage of $\H^1_g(G_L,U(\Q_p))$ under this map is $\H^1_g(G_K,U(\Q_p))$. The same applies with $\H^1_f$ or $\H^1_e$ in place of $\H^1_g$.
\begin{proof}
To prove this, it suffices to consider only the case that $L/K$ is finite Galois. Moreover, to prove injectivity for $U$ it suffices by naturality of the Serre twisting bijection to prove that the restriction map $\H^1(G_K,{}_\alpha U(\Q_p))\rightarrow\H^1(G_L,{}_\alpha U(\Q_p))$ has trivial kernel for all twists ${}_\alpha U$ of $U$ by continuous cocycles $\alpha\colon G_K\rightarrow U(\Q_p)$.

By non-abelian inflation-restriction, the kernel of the restriction map is given by $\H^1\left(G_{L|K},{}_\alpha U(\Q_p)^{G_L}\right)$. But ${}_\alpha U(\Q_p)$ is a uniquely divisible group, and hence so is ${}_\alpha U(\Q_p)^{G_L}$. As $G_{L|K}$ is a finite group, this already implies that $\H^1\left(G_{L|K},{}_\alpha U(\Q_p)^{G_L}\right)$ is trivial as desired.

To prove the assertion about $\H^1_g$, it suffices to show that $\H^1(G_K,U(\B_\st))\rightarrow\H^1(G_L,U(\B_\st))$ has trivial kernel. In fact, it is even injective, as one proves by replacing $U(\Q_p)$ by $U(\B_\st)$ throughout the above proof. The cases of $\H^1_f$ and $\H^1_e$ are similar.
\end{proof}
\end{lemma}

\begin{lemma}\label{lem:twists_give_cosets}
Let $U/\Q_p$ be a representation of $G_K$ on a finitely generated pro-unipotent group and $*\in\{e,f,g\}$. Let~$\sim_{\H^1_*}$ denote the equivalence relation on $\H^1(G_K,U(\Q_p))$ described in Definition~\ref{def:bloch-kato_sets}. Then the $\sim_{\H^1_*}$-equivalence class of $[\alpha]\in\H^1(G_K,U(\Q_p))$ for a continuous cocycle $\alpha\in\Cycle^1(G_K,U(\Q_p)))$ is the image of $\H^1_*(G_K,{}_\alpha U(\Q_p))$ under the bijection $\H^1(G_K,{}_\alpha U(\Q_p))\isoarrow\H^1(G_K,U(\Q_p))$ from Proposition~\ref{prop:serre_twisting_bijection}.
\begin{proof}
This is obvious from, for example, commutativity of the square
\begin{center}
\begin{tikzcd}
\H^1(G_K,{}_\alpha U(\Q_p)) \arrow[d,"\wr"]\arrow[r] & \H^1(G_K,{}_\alpha U(\B_\cris^{\varphi=1})) \arrow[d,"\wr"] \\
\H^1(G_K,U(\Q_p)) \arrow[r] & \H^1(G_K,U(\B_\cris^{\varphi=1})).
\end{tikzcd}
\end{center}
\end{proof}
\end{lemma}

In applying lemma \ref{lem:twists_give_cosets}, it will be important for us to understand when the Serre twist ${}_\alpha U$ is de Rham. For our purposes, the following criterion will suffice.

\begin{proposition}\label{prop:admissibility_of_twists}
Let $U/\Q_p$ be a representation of $G_K$ on a finitely generated pro-unipotent group, and $\alpha\colon G_K\rightarrow U(\Q_p)$ a continuous cocycle whose class is in $\H^1_g(G_K,U(\Q_p))$. Then ${}_\alpha U$ is de Rham (resp.\ semistable) iff $U$ is, and the Serre twisting bijection $\H^1(G_K,{}_\alpha U(\Q_p))\isoarrow\H^1(G_K,U(\Q_p))$ induces a bijection on $\H^1_g$.
\begin{proof}
The fact that $[\alpha]\in\H^1_g(G_K,U(\Q_p))$ says that there is some $u\in U(\B_\st)$ (hence also in $U(\B_\dR)$) such that $\alpha(\sigma)=u^{-1}\sigma(u)$ for all $\sigma\in G_K$. Then left-conjugation by $u$ induces a natural $G_K$-equivariant isomorphism ${}_\alpha U_{\B_\st}\isoarrow U_{\B_\st}$ and hence also an isomorphism $\D_\st({}_\alpha U)\isoarrow\D_\st(U)$. Since these isomorphisms are compatible under the natural embedding $\D_\st(U)_{\B_\st}\hookrightarrow U_{\B_\st}$ (and similarly for ${}_\alpha U$), we have that $U$ is semistable iff ${}_\alpha U$ is (and a similar argument works for de Rhamness). The final assertion follows lemma \ref{lem:twists_give_cosets} since $[\alpha]\sim_{\H^1_g}\ast$.
\end{proof}
\end{proposition}

\begin{remark}\label{rmk:inverse_limits_of_H^1_g}
One can show relatively straightforwardly that if $U/\Q_p$ is a representation of $G_K$ on a finitely generated pro-unipotent group and we write $U=\liminv U_n$ as an inverse limit of finite-dimensional $G_K$-equivariant quotients, then the natural map\[\H^1_g(G_K,U(\Q_p))\rightarrow\liminv\H^1_g(G_K,U_n(\Q_p))\]is bijective (and similarly for $\H^1_f$ and $\H^1_e$), so that one can reduce the study of local Bloch--Kato Selmer sets to the case when $U$ is finite-dimensional. However, it will be our preference to always work at the infinite level, not least because the corresponding statements for the Bloch--Kato quotients $\H^1_{g/e}$ and $\H^1_{f/e}$ are much less straightforward to prove directly (see lemma \ref{lem:inverse_limits_of_H^1_g}).
\end{remark}

\subsection{Torsors and $\H^1$}
\label{s:torsors}

A different approach to the study of local Bloch--Kato Selmer sets has been undertaken in Kim's papers \cite{minhyong:siegel,minhyong:selmer} by viewing these sets as classifying sets for certain classes of torsors under $U$. This viewpoint has the distinct advantage of being the natural setting in which to describe the non-abelian Kummer map\[Y(K)\rightarrow\H^1(G_K,U(\Q_p)),\]which is the classifying map for path-torsors when $U$ is the $\Q_p$-pro-unipotent \'etale fundamental group of a variety $Y/K$ at a $K$-rational basepoint. Although our analysis will proceed rather differently, let us now recall this description and deduce thereby a non-abelian version of Hyodo's $\H^1_g=\H^1_\st$ theorem \cite{hyodo}.

\begin{definition-lemma}
Let $U/\Q_p$ be a representation of $G_K$ on a finitely generated pro-unipotent group, and let $P/\Q_p$ be a (right\nobreakdash-)torsor under (the underlying group scheme of) $U$, endowed with an action of $G_K$ by algebraic automorphisms. Then the conilpotency filtration on $\O(P)$ is $G_K$-stable. Moreover, the following are equivalent:
\begin{enumerate}
	\item the action of $G_K$ on $P(\Q_p)$ is continuous;
	\item the action of $G_K$ on each $J_n\O(P)$ is continuous.
\end{enumerate}
When these equivalent conditions hold, we will say that $P$ is a \emph{torsor under $U$}.
\begin{proof}[Proof of equivalence]
$G_K$-stability of the conilpotency filtration is immediate since it is defined by the kernels of $G_K$-equivariant maps.

If we write $U=\liminv U_n$ as an inverse limit of finite-dimensional $G_K$-equivariant quotients, then $P(\Q_p)$ is homeomorphic to the inverse limit of the pushouts $(P\times^UU_n)(\Q_p)$, so that it suffices to prove this for $U$ finite-dimensional. As in definition-lemma \ref{def-lem:unipotent_representation}, the second condition is equivalent to the action on each $\left(J_n\O(P)\right)^\dual$ being continuous. In one direction, $P(\Q_p)$ is the set of group-like elements of the coalgebra $\O(P)^\dual=\liminv\left(J_n\O(P)\right)^\dual$ and its topology is the subspace topology, so that the implication (2)$\Rightarrow$(1) is clear.

In the other direction, there is a canonical $G_K$-equivariant homeomorphism\[P(\Q_p)\times\left(J_n\O(U)\right)^\dual\isoarrow P(\Q_p)\times\left(J_n\O(P)\right)^\dual\]given by $(q,u)\mapsto(q,qu)$. When condition (1) holds, the $G_K$-action on the left-hand side is continuous, and hence the action on $J_n\O(P)$ is continuous also.
\end{proof}
\end{definition-lemma}

The importance of this notion of torsor is twofold. Firstly, such torsors arise naturally from geometry as torsors of \'etale paths between two $K$-rational points $x,y$ on a connected variety $X/K$. Secondly, the collection of $U$-torsors is parametrised by the continuous cohomology of $U(\Q_p)$, which allows one to understand them in a very concrete fashion.

\begin{lemma}\label{lem:torsors=H^1}
Let $U/\Q_p$ be a representation of $G_K$ on a finitely generated pro-unipotent group. Then there is a canonical bijection\[\H^1(G_K,U(\Q_p))\isoarrow\{\text{torsors under $U$}\}/\simeq.\]
\begin{proof}
This is a special case of \cite[Proposition 1]{minhyong:siegel}, where the filtration $W_n$ on $\O(P)$ is the conilpotency filtration. It is worthwhile quickly recalling how this construction goes.

To a cocycle $\alpha\in\Cycle^1(G_K,U(\Q_p))$ we associate a $U$-torsor ${}_\alpha P$ whose underlying scheme is $U$, endowed the right action of $U$ and the $\alpha$-twisted $G_K$ action, given on $\Lambda$-points for some $\Q_p$-algebra $\Lambda$ by\[g\colon u\mapsto\alpha(g)g(u).\]Just as in \S\ref{ss:serre_twisting}, this action is clearly algebraic and continuous on $\Q_p$-points.  Moreover, this torsor has a distinguished $\Q_p$-point, namely $1\in U(\Q_p)$, and it is easy to check that this construction induces a bijection between $\Cycle^1(G_K,U(\Q_p))$ and the set\footnote{Technically, setoid.} of pointed $U$-torsors (where the inverse map takes $(P,x)$ to the cocycle $g\mapsto x^{-1}g(x)$).

Now the action of $U(\Q_p)$ on $\Cycle^1(G_K,U(\Q_p))$ exactly corresponds to the change-of-basepoint action on pointed $U$-torsors, so the correspondence descends to a bijection between $\H^1(G_K,U(\Q_p))$ and the set of isomorphism classes of $U$-torsors $P$ such that $P(\Q_p)\neq\emptyset$. But since $U$ is a pro-unipotent group, all torsors satisfy this and we have the desired bijection.
\end{proof}
\end{lemma}

In view of lemma \ref{lem:torsors=H^1}, the non-abelian local Bloch--Kato Selmer sets should correspond to a subclass of torsors satisfying certain admissibility conditions with respect to $\B_\st$, $\B_\cris$ and $\B_\cris^{\varphi=1}$ respectively. In fact, this criterion is very simple to state (and is essentially already found in \cite{minhyong:selmer}).

\begin{definition-lemma}\label{def-lem:admissible_torsor}
Let $U/\Q_p$ be a representation of $G_K$ on a finitely generated pro-unipotent group, $P/\Q_p$ be a torsor under $U$, and let $\B/\Q_p$ be a $(\Q_p,G_K)$-regular $\Q_p$-algebra. The following are equivalent:
\begin{enumerate}
	\item $[P]$ lies in the kernel of $\H^1(G_K,U(\Q_p))\rightarrow\H^1(G_K,U(\B))$;
	\item $P(\B)^{G_K}\neq\emptyset$;
	\item (if $U$ is $\B$-admissible) each $J_n\O(P)$ is $\B$-admissible.
\end{enumerate}
When these equivalent conditions hold, we will say that $P$ is a \emph{(relatively) $\B$-admissible torsor} under $U$.
\begin{proof}[Proof of equivalence]
For the equivalence (1)$\Leftrightarrow$(2), we may assume that the underlying $\Q_p$-variety of $P$ is $U$, and that the $G_K$ action is that determined by a cocycle $\alpha$. It follows from chasing definitions in lemma \ref{lem:torsors=H^1} that the image of $[\alpha]$ in $\Cycle^1(G_K,U(\B))$ is the coboundary of some $u\in U(\B)$ iff $u$ is fixed for the $\alpha$-twisted action.

For the implication (2)$\Rightarrow$(3) when $U$ is $\B$-admissible, a choice of a $G_K$-fixed point on $P(\B)$ gives a $G_K$-equivariant isomorphism $P_\B\isoarrow U_\B$ of $\B$-schemes, and hence a $G_K$-equivariant isomorphism $\B\otimes_{\Q_p}\O(P)\isoarrow\B\otimes_{\Q_p}\O(U)$. Since $\O(U)$ is ind-$\B$-admissible, so too is $\O(P)$.

For the converse implication, we recall from \cite[Proposition 1.5.2]{fontaine3} that the Dieudonn\'e functor $\D$ on $\B$-admissible representations commutes with tensor products and duals, so that when $\O(P)$ is ind-$\B$-admissible, $\D(\O(P))$ is an algebra over the fixed field $F=\B^{G_K}$, and a right-comodule for $\D(\O(U))$, such that $\D(\O(P))\otimes_F\D(\O(P))\rightarrow\D(\O(U))\otimes_F\D(\O(P))$ is an isomorphism. Since $\D(\O(P))\neq0$, this implies that $\D(P):=\Spec(\D(\O(P))$ is a torsor under the unipotent group $\D(U)=\Spec(\D(\O(U)))$ over $F$ (remark \ref{rmk:affine_rings_of_admissible_representations}). The same argument as in remark \ref{rmk:affine_rings_of_admissible_representations} shows that $\D(P)(\Lambda)=P(\B\otimes_F\Lambda)^{G_K}$ for any $F$-algebra $\Lambda$. But then by triviality of the $\D(U)$-torsor $\D(P)$, we know that $P(\B)^{G_K}=\D(P)(F)\neq\emptyset$, as desired.
\end{proof}
\end{definition-lemma}

\begin{remark}
Definition-lemma \ref{def-lem:admissible_torsor} is a non-abelian version of the well-known fact that, under the correspondence $\H^1(G_K,V)\cong\Ext^1_{G_K}(\Q_p,V)$ for a $p$-adic representation $V$, the kernel of $\H^1(G_K,V)\rightarrow\H^1(G_K,\B\otimes V)$ corresponds to the class of those extensions of $\Q_p$ by $V$ which remain exact upon applying the associated Dieudonn\'e functor $\D$.
\end{remark}

Using this torsor-theoretic language, we are now able to justify our use of $\B_\st$ in place of $\B_\dR$ in the definition of $\H^1_g$, as promised. This is a non-abelian version of Hyodo's $\H^1_g=\H^1_\st$ theorem \cite{hyodo}.

\begin{lemma}\label{lem:our_H^1_g_is_right}
Let $U/\Q_p$ be a de Rham representation of $G_K$ on a finitely generated pro-unipotent group. Then the kernel of $\H^1(G_K,U(\Q_p))\rightarrow\H^1(G_K,U(\B_\dR))$ is $\H^1_g(G_K,U(\Q_p))$.
\begin{proof}
The $G_K$-equivariant inclusion $\B_\st\hookrightarrow\B_\dR$ induced from our choice of $p$-adic logarithm factors the map $\H^1(G_K,U(\Q_p))\rightarrow\H^1(G_K,U(\B_\dR))$ through $\H^1(G_K,U(\Q_p))$, and hence its kernel clearly contains $\H^1_g(G_K,U(\B_\st))$. For the converse implication, consider any $U$-torsor $P$ whose class lies in the kernel of $\H^1(G_K,U(\Q_p))\rightarrow\H^1(G_K,U(\B_\dR))$, i.e.\ such that $\O(P)$ is ind-de Rham.

By potential semistability (corollary \ref{cor:potential_semistability_p-adic}) we may pick a finite $L/K$ such that $\O(U)$ is ind-semistable as a representation of $G_L$. Then from the canonical isomorphism $\gr^J_\bullet\O(P)\cong\gr^J_\bullet\O(U)$ (see proposition \ref{prop:graded_isomorphism}), we see that $\gr^J_\bullet\O(P)$, and hence $\O(P)$, is an ind-semistable representation of $G_L$. This tells us that the class of $P|_{G_L}$ lies in $\H^1_g(G_L,U(\Q_p))$, and hence by lemma \ref{lem:bloch-kato_stable_under_base_change} the class of $P$ lies in $\H^1_g(G_K,U(\Q_p))$ as desired.
\end{proof}
\end{lemma}

\subsubsection{Isocrystal moduli interpretations for Bloch--Kato Selmer sets}

To conclude this section, let us briefly discuss a second moduli interpretation for the Bloch--Kato Selmer sets $\H^1_e$, $\H^1_f$ and $\H^1_g$, as well as the relative quotients $\H^1_{g/e}$, $\H^1_{f/e}$ and $\H^1_{g/f}$ which is of a more log-crystalline flavour. As in \S\ref{ss:extra_structure}, we will assume for simplicity that $U$ is semistable, so that $\O(\D_\st(U))$ is an ind-weakly admissible filtered $(\varphi,N)$-module and $\D_\st(U)$ carries a Frobenius $\varphi$ and monodromy operator $N$. In this case, definition-lemma \ref{def-lem:admissible_torsor} shows that $\H^1_g(G_K,U(\Q_p))$ has an interpretation as the set of isomorphism classes of torsors $P$ under $U$ such that $\O(P)$ is ind-semistable. Thus by the equivalence of symmetric monoidal categories between semistable representations and weakly admissible filtered $(\varphi,N)$-modules \cite[Theorem A]{colmez-fontaine}, we obtain an equivalent description in terms of $\D_\st(U)$.

\begin{lemma}\label{lem:isocrystal_interpretation_H^1_g}
Let $U/\Q_p$ be a semistable representation of $G_K$ on a finitely generated pro-unipotent group. Then there is a canonical bijection between $\H^1_g(G_K,U(\Q_p))$ and the set $\H^1(\MF_K^{\varphi,N,\wa},\D_\st(U)(K_0))$ of isomorphism classes of torsors $Q$ under $\D_\st(U)$ endowed with a ind-weakly admissible filtered $(\varphi,N)$-module structure on $\O(Q)$ compatible with the $\O(\D_\st(U))$-comodule algebra structure maps.
\end{lemma}

In this lemma and what follows, $K_0$ denotes the maximal absolutely unramified subfield of $K$, i.e.\ the fraction field of the ring of Witt vectors of the residue field of $K$.

Just as for $\D_\st(U)$, for such a torsor $Q$ the Frobenius and monodromy on $\O(Q)$ induce a Frobenius automorphism $\varphi\colon Q\isoarrow\varphi^*Q$ and monodromy vector field $N\colon Q\rightarrow TQ$ on $Q$, compatible with the corresponding structures on $\D_\st(U)$ under the action map $Q\times\D_\st(U)\rightarrow Q$ and satisfying $\varphi^*N\circ\varphi=p\cdot\dee\varphi\circ N$.

\begin{lemma}\label{lem:isocrystal_interpretation_H^1_g/e}
Let $U/\Q_p$ be a semistable representation of $G_K$ on a finitely generated pro-unipotent group. Then there is a canonical bijection between:
\begin{itemize}
	\item $\H^1_{g/e}(G_K,U(\Q_p))$;
	\item the set $\H^1(\Mod_K^{\varphi,N},\D_\st(U)(K_0))$ of isomorphism classes of torsors $Q$ under $\D_\st(U)$ endowed with a (non-filtered) $(\varphi,N)$-module structure on $\O(Q)$ compatible with the $\O(\D_\st(U))$-comodule algebra structure maps; and
	\item the set of isomorphism classes of torsors $Q$ under $\D_\st(U)$ endowed with a Frobenius $\varphi\colon Q\isoarrow\varphi^*Q$ and monodromy vector field $N\colon Q\rightarrow TQ$ satisfying $\varphi^*N\circ\varphi=p\cdot\dee\varphi\circ N$ and compatible with the corresponding structures on $\D_\st(U)$ under the action map $Q\times\D_\st(U)\rightarrow Q$.
\end{itemize}
\begin{proof}
The bijection between the latter two kinds of structures on a torsor $Q$ arises from the usual bijection between $\varphi$-semilinear variety automorphisms of $Q$ (resp.\ vector fields on $Q$) and $\varphi$-semilinear algebra automorphisms of $\O(Q)$ (resp.\ derivations on $\O(Q)$) -- one checks straightforwardly that the conditions imposed on the variety side exactly correspond to those on the algebraic side.

For the equivalence of the first point with the other two, we consider the map\[\H^1_g(G_K,U(\Q_p))\isoarrow\H^1(\MF_K^{\varphi,N,\wa},\D_\st(U)(K_0))\rightarrow\H^1(\Mod_K^{\varphi,N},\D_\st(U)(K_0)),\]where the isomorphism is the one from lemma \ref{lem:isocrystal_interpretation_H^1_g}, and the second map is the one which forgets the Hodge structure on the affine ring of a $\D_\st(U)$-torsor. It suffices to show that this map is surjective and that its fibres are exactly the $\sim_{\H^1_e}$ equivalence classes, so that it factors through a bijection\[\H^1_{g/e}(G_K,U(\Q_p))\isoarrow\H^1(\Mod_K^{\varphi,N},\D_\st(U)(K_0)).\]

To show surjectivity, consider any $Q$ with a $(\varphi,N)$-module structure on $\O(Q)$. It follows by definition \ref{def:J-filtration} that the conilpotency filtration $J_\bullet\O(Q)$ is a filtration by sub-$(\varphi,N)$-modules, and that the canonical isomorphism $\gr^J_\bullet(\O(Q))\isoarrow\gr^J_\bullet(\O(\D_\st(U)))$ is an isomorphism of $(\varphi,N)$-modules in each degree.

If we pick any $q\in Q(K_0)$, we obtain a trivialisation $\D_\st(U)\isoarrow Q$, and may put a Hodge filtration on $K\otimes_{K_0}\O(Q)$ by declaring it to be the filtration corresponding to the Hodge filtration on $K\otimes_{K_0}\O(\D_\st(U))$ under the induced isomorphism $\O(Q)\isoarrow\O(\D_\st(U))$. By construction, the canonical isomorphism $\gr^J_\bullet(\O(Q))\isoarrow\gr^J_\bullet(\O(\D_\st(U)))$ is then a filtered isomorphism for the Hodge filtration, so that the $J$-graded pieces of $\O(Q)$, with their induced Hodge filtrations, are weakly admissible. Hence, by \cite[Proposition 3.4]{colmez-fontaine}, $\O(Q)$ itself is ind-weakly admissible, so that the class of $Q$ lies in the image of $\H^1(\MF_K^{\varphi,N,\wa},\D_\st(U)(K_0))\rightarrow\H^1(\Mod_K^{\varphi,N},\D_\st(U)(K_0))$, as desired.

Finally, to show that the fibres of $\H^1_g(G_K,U(\Q_p))\rightarrow\H^1(\Mod_K^{\varphi,N},\D_\st(U)(K_0))$ are exactly the $\sim_{\H^1_e}$ equivalence classes, consider any two semistable torsors $P$, $P'$ under $U$, and write $\D_\st(P)=\Spec(\D_\st(\O(P)))$, respectively $\D_\st(P')$, for the associated $\D_\st(U)$-torsors. $P$ and $P'$ lie in the same fibre iff $\D_\st(P)$ and $\D_\st(P')$ are isomorphic as $\D_\st(U)$-torsors compatibly with the $(\varphi,N)$-module structure on their affine rings. This occurs iff $\D_\st(P)_{\B_\st}=P_{\B_\st}$ and $\D_\st(P')_{\B_\st}=P'_{\B_\st}$ are isomorphic as $\D_\st(U)_{\B_\st}=U_{\B_\st}$-torsors, compatibly with the Galois action, Frobenius and monodromy on the affine rings. This, in turn, occurs iff there is a $G_K$-equivariant isomorphism $P_{\B_\cris^{\varphi=1}}\simeq P'_{\B_\cris^{\varphi=1}}$ of $U_{\B_\cris^{\varphi=1}}$-torsors. But this final condition says exactly that the images of $P$ and $P'$ in $\H^1(G_K,U(\B_\cris^{\varphi=1}))$ are equal, i.e.\ that they lie in the same $\sim_{\H^1_e}$-equivalence class, which is what we wanted to show.
\end{proof}
\end{lemma}
\section{Homotopical algebraic background}\label{c:homotopical_algebra}

Our approach to the analysis of non-abelian Bloch--Kato Selmer sets relies on a suitable non-abelian generalisation of the homological algebra of cochain complexes and their cohomology, so that we can effectively transcribe arguments from the abelian setting to the non-abelian. This generalisation is provided by the theory of cosimplicial groups and their cohomotopy, as defined for instance in \cite{bousfield-kan}, and which subsumes the more well-known theory of non-abelian group cohomology \cite[Section I.5]{serre}. For the convenience of the reader, we will recall the basic setup here as expounded in \cite{pirashvili}.

\subsection{Cosimplicial groups and cohomotopy}

\begin{definition}[Cosimplicial objects]\label{def:cosimplicial_objects}
Let $\mcal C$ be a category. A \emph{cosimplicial object} of $\mcal C$ is a covariant functor $X^\bullet\colon\Delta\rightarrow\mcal C$, where $\Delta$ is the simplex category whose objects are $[n]=\{0,1,\dots,n\}$ for $n\geq0$ and whose morphisms are the order-preserving maps of sets (we usually write $X^n$ for $X([n])$). A morphism of cosimplicial objects is just a natural transformation of functors, so the category of cosimplicial objects of $\mcal C$ is the functor category $\mcal C^\Delta$.

If we replace the simplex category $\Delta$ with the strict simplex category $\Delta^+$ consisting of the monomorphisms in $\Delta$, then we have the notion of a \emph{semi-cosimplicial object}. When $\mcal C$ has finite products, we define the cosimplicial object $\Kan^\bullet(X^\bullet)$ \emph{cogenerated by} a semi-cosimplicial object $X^\bullet\colon\Delta^+\rightarrow\mcal C$ to be, at some $[n]\in\ob(\Delta)$\[\Kan^n(X^\bullet):=\prod_{[n]\twoheadrightarrow[k]}X^k\]where the product is taken over all order-preserving surjections from the set $[n]$. The map $\Kan^{n'}(X^\bullet)\rightarrow\Kan^n(X^\bullet)$ associated to a map $f\colon[n']\rightarrow[n]$ is defined to be the unique one making commute all squares
\begin{center}
\begin{tikzcd}
\prod_{[n']\twoheadrightarrow[k']}X^{k'} \arrow[r]\arrow[d,"\proj_{g'}"] & \prod_{[n]\twoheadrightarrow[k]} X^k \arrow[d,"\proj_g"] \\
X^{k'} \arrow[r,"X^\bullet(f')"] & X^k
\end{tikzcd}
\end{center}
for each $g\colon[n]\twoheadrightarrow[k]$, where $[n']\overset{g'}\twoheadrightarrow[k']\overset{f'}\hookrightarrow[k]$ is the epi-mono factorisation of $gf$. There is a natural morphism $\Kan^\bullet(X^\bullet)\rightarrow X^\bullet$ of semi-cosimplicial objects, which is terminal among all such maps to $X^\bullet$ from a cosimplicial object.
\end{definition}

\begin{remark}
Often, instead of considering all morphisms appearing in a cosimplicial object $X^\bullet$, we just consider the special generating set given by the \emph{coface maps} $d^i\colon X^{n-1}\rightarrow X^n$ and \emph{codegeneracy maps} $s^i\colon X^{n+1}\rightarrow X^n$ for $0\leq i\leq n$, which are those induced from the unique injection $[n-1]\hookrightarrow[n]$ (resp.\ surjection $[n+1]\twoheadrightarrow[n]$) which misses $i\in[n]$ (resp.\ hits $i\in[n]$ twice). These are related by the \emph{cosimplicial identities}:
\begin{itemize}
	\item $d^jd^i=d^id^{j-1}$ for $i<j$;
	\item $s^js^i=s^is^{j+1}$ for $i\leq j$; and
	\item $s^jd^i=\begin{cases}d^is^{j-1} & \text{for $i<j$,}\\\id & \text{for $i=j,j+1$,}\\ d^{i-1}s^j & \text{for $i>j+1$;}\end{cases}$
\end{itemize}
and, conversely, maps between the $X^n$ satisfying these identities uniquely determine the structure of a cosimplicial object.
\end{remark}

\begin{remark}\label{rmk:dold-kan_correspondence}
The functor $\Kan^\bullet\colon\mcal C^{\Delta^+}\rightarrow\mcal C^\Delta$ is, when $\mcal C$ is abelian, essentially one of the two functors giving the equivalence in the cosimplicial Dold--Kan correspondence \cite[Corollary III.2.3]{goerss-jardine}\footnote{The Dold--Kan correspondence in \cite{goerss-jardine} is given for simplicial objects rather than cosimplicial objects, but since the dual of an abelian category is again an abelian category, the dual statement holds just as well. This is spelled out explicitly in \cite{grunenfelder-mastnak}.}. More precisely, if we precompose with the functor $\CoCh_+(\mcal C)\rightarrow\mcal C^{\Delta^+}$ taking a coconnected cochain complex $A^\bullet$ to the semi-cosimplicial object with the same objects and whose coface maps $A^{n-1}\rightarrow A^n$ all zero except $d^n:=(-1)^nd$, then the composite $\mathrm{CoCh}_+(\mcal C)\rightarrow\mcal C^\Delta$ is one half of the equivalence of categories in the Dold--Kan correspondence. Because of this, $\Kan^\bullet$ is often called the \emph{denormalisation} functor (e.g.\ in \cite{olsson,kriz}), and can be variously denoted by $\mathbf D$, $D$ or $K$ in different sources\footnote{There is an apparent discrepancy in the definition of the denormalisation functor in \cite[Section 1.4]{kriz} and the dual to the definition in \cite[Section III.2]{goerss-jardine}, even though they are both supposed to be quasi-inverses to the same normalisation functor. To save confusion, we will work exclusively with the latter description.}.
\end{remark}

The preceding remark shows that the category of coconnected cochain complexes of abelian groups is a full subcategory of the category of cosimplicial groups, and hence one can regard cosimplicial groups as a generalisation of cochain complexes more suited to non-abelian results. Accordingly, one hopes for suitable non-abelian generalisations of the cohomology of complexes, at least in small codimension, and this is provided by the \emph{cohomotopy sets/groups} of Bousfield--Kan.

\begin{definition}\cite[Section 2]{pirashvili}\label{def:cohomotopy}
Let $U^\bullet$ be a cosimplicial or semi-cosimplicial group. The $0$th \emph{cohomotopy group} $\pi^0(U^\bullet)$ and $1$st \emph{cohomotopy pointed set} $\pi^1(U^\bullet)$ are defined to be:
\begin{itemize}
	\item $\pi^0(U^\bullet)$ is the equaliser of $d^0,d^1\colon U^0\rightarrow U^1$;
	\item $\pi^1(U^\bullet):=\Cycle^1(U^\bullet)/U^0$ where\[\Cycle^1(U^\bullet)=\{u_1\in U^1\:|\:d^1(u_1)=d^2(u_1)d^0(u_1)\}\]is the set of \emph{cocycles} and $U^0$ acts on this set by twisted conjugation $u_0\colon u_1\mapsto d^1(u_0)^{-1}u_1d^0(u_0)$. The distinguished point is just the class of $1\in\Cycle^1(U^\bullet)$.
\end{itemize}
When $U^\bullet$ is abelian, this can be extended further to give abelian cohomotopy groups $\pi^i(U^\bullet)$ for all $i\geq0$, defined as the cohomology groups of the (unnormalised) Moore complex $C^\bullet(U^\bullet)$, whose terms are $C^n(U^\bullet)=U^n$ and whose differential is the alternating sum $\sum_k(-1)^kd^k$ of the coface maps. It is straightforward to check this agrees with the above definition of $\pi^0$ and $\pi^1$.

Note that $\pi^0(U^\bullet)$ makes sense for cosimplicial objects in any category with equalisers, so will inherit e.g.\ topologies or group actions when $U^\bullet$ has them.
\end{definition}

We will see many variants on the following example throughout this paper.

\begin{example}[Double-coset spaces]\label{ex:pairs_of_subgroups}
Let $U$ be a group with subgroups $U',U''\leq U$, and consider the cosimplicial group $U^\bullet$ cogenerated by the semi-cosimplicial group\[U'\times U''\rightrightarrows U,\]where $d^0(u',u'')=u'$ and $d^1(u',u'')=u''$ (and subsequent terms are the trivial group). Then (for example using remark \ref{rmk:non-abelian_dold-kan} shortly) we canonically identify the cohomotopy groups of $U^\bullet$ as\[\pi^i(U^\bullet)\cong\begin{cases}U'\cap U''&\text{if $i=0$,}\\U''\backslash U/U'&\text{if $i=1$,}\\1&\text{if $i>1$ and $U$ abelian.}\end{cases}\]
\end{example}

The following lemmas are essentially restatements of the Dold--Kan and Eilenberg--Zilber theorems, which will be useful in computations.

\begin{lemma}\label{lem:dold-kan}
Let $A^\bullet$ be a semi-cosimplicial abelian group (or semi-cosimplicial object in any abelian category), and $\Kan^\bullet(A^\bullet)$ the cosimplicial abelian group cogenerated by it. Then the canonical map $\Kan^\bullet(A^\bullet)\rightarrow A^\bullet$ of semi-cosimplicial groups induces an isomorphism on cohomotopy groups.
\end{lemma}

\begin{remark}\label{rmk:non-abelian_dold-kan}
Lemma \ref{lem:dold-kan} is also true when $A^\bullet$ is replaced by a cosimplicial group $U^\bullet$, although in this case we can only say that $\Kan^\bullet(U^\bullet)\rightarrow U^\bullet$ induces an isomorphism on $\pi^0$ and $\pi^1$. This is proved by a straightforward chasing of definitions, which we leave to the reader.
\end{remark}

\begin{lemma}\label{lem:eilenberg-zilber}
Let $A^{\bullet,\bullet}$ be a bi-semi-cosimplicial abelian group (i.e.\ a functor $\Delta^+\times\Delta^+\rightarrow\Ab$). Let $C^{\bullet,\bullet}(A^{\bullet,\bullet})$ be its Moore bicomplex (formed by taking the unnormalised Moore complex in the horizontal and vertical directions), and let $\Kan^{\bullet,\bullet}(A^{\bullet,\bullet})$ be the bi-cosimplicial abelian group formed by taking the cosimplicial abelian group cogenerated by the semi-cosimplicial abelian groups in each of the horizontal and vertical directions. Then there is a canonical isomorphism\[\pi^j(d^\bullet(\Kan^{\bullet,\bullet}(A^{\bullet,\bullet})))\cong\H^j(\Tot^\bullet(C^{\bullet,\bullet}(A^{\bullet,\bullet})))\]where $d^\bullet(\Kan^{\bullet,\bullet}(A^{\bullet,\bullet}))$ is the diagonal cosimplicial abelian group.
\end{lemma}

\subsection{Long exact sequences of cohomotopy}

For us, the main advantage of using cosimplicial groups is that they allow us to mechanistically produce long exact sequences (in the sense of definition \ref{def:exact_sequence}) in much the same way as one produces long exact sequences in the cohomology of cochain complexes from short exact sequences of cochains. Unlike the abelian case, 
of course, the exact sequences we obtain will in general be finite in length, and moreover the length of the sequence we obtain will depend on how well-behaved the terms of the short exact sequence are. For our purposes, the following four cases will suffice.

\begin{theorem}\label{thm:long_exact_seqs}
\leavevmode
\begin{enumerate}
	\item Suppose $U^\bullet$ is a cosimplicial group acting degree-wise transitively on a cosimplicial set $Q^\bullet$ from the right, and let $Z^\bullet\hookrightarrow U^\bullet$ be the stabiliser of a chosen basepoint $*\rightarrow Q^\bullet$. Then there is a functorially assigned exact sequence\[1\rightarrow\pi^0(Z^\bullet)\rightarrow\pi^0(U^\bullet)\actsarrow\pi^0(Q^\bullet)\rightarrow\pi^1(Z^\bullet)\rightarrow\pi^1(U^\bullet).\]
	\item Suppose\[1\rightarrow Z^\bullet\rightarrow U^\bullet\rightarrow Q^\bullet\rightarrow1\]is a degree-wise exact sequence of cosimplicial groups. Then there is a functorially assigned exact sequence\[1\rightarrow\pi^0(Z^\bullet)\rightarrow\pi^0(U^\bullet)\rightarrow\pi^0(Q^\bullet)\actsarrow\pi^1(Z^\bullet)\rightarrow\pi^1(U^\bullet)\rightarrow\pi^1(Q^\bullet).\]
	\item Suppose\[1\rightarrow Z^\bullet\centarrow U^\bullet\rightarrow Q^\bullet\rightarrow1\]is a degree-wise central extension of cosimplicial groups. Then there is a functorially assigned exact sequence\[1\rightarrow\pi^0(Z^\bullet)\centarrow\pi^0(U^\bullet)\rightarrow\pi^0(Q^\bullet)\rightarrow\pi^1(Z^\bullet)\actsarrow\pi^1(U^\bullet)\rightarrow\pi^1(Q^\bullet)\rightarrow\pi^2(Z^\bullet).\]
	\item Suppose\[1\rightarrow Z^\bullet\rightarrow U^\bullet\rightarrow Q^\bullet\rightarrow1\]is a degree-wise exact sequence of cosimplicial abelian groups. Then there is a functorially assigned exact sequence of abelian groups
	\[\cdots\rightarrow\pi^{r-1}(Q^\bullet)\rightarrow\pi^r(Z^\bullet)\rightarrow\pi^r(U^\bullet)\rightarrow\pi^r(Q^\bullet)\rightarrow\pi^{r+1}(Z^\bullet)\rightarrow\cdots.\]
\end{enumerate}
\begin{proof}
Part $(2)$ is \cite[Proposition 2.6.1]{pirashvili} (originally \cite[Lemma 3.2 and Proposition 3.2.1]{duflot-marak}) and \cite[Lemma 2.6.2]{pirashvili}, while part $(3)$ is \cite[Proposition 2.8.2]{pirashvili} (also originally in \cite{duflot-marak}). Neither of these sources explicitly mention the action of $\pi^1(Z^\bullet)$ on $\pi^1(U^\bullet)$, but this is just induced by the multiplication action of $\Cycle^1(Z^\bullet)$ on $\Cycle^1(U^\bullet)$. Part $(4)$ is well-known, in that this is just the exact sequence on cohomology for the unnormalised Moore complexes.

Part $(1)$ we do by hand, following \cite[Section I.5.4]{serre}. With the exception of the coboundary map $\delta\colon\pi^0(Q^\bullet)\rightarrow\pi^1(Z^\bullet)$, the maps and group action in the sequence are just the ones induced from the maps and group actions between the $Z^j$, $U^j$ and $Q^j$. To define $\delta$ at some $q_0\in\pi^0(Q^\bullet)$, we choose some $u_0\in U^0$ such that $q_0=*u_0$. The $0$-cocycle condition for $q_0$ says exactly that $*d^0(u_0)=*d^1(u_0)$, so that $d^1(u_0)d^0(u_0)^{-1}\in Z^1$.

Since $d^1(u_0)d^0(u_0)^{-1}$ is the coboundary of $u_0^{-1}$ in $U^1$ and $Z^2\hookrightarrow U^2$ is injective, it is even a cocycle in $Z^1$. Moreover, since the choice of $u_0$ was unique up to left multiplication by $Z^0$, its class $[d^1(u_0)d^0(u_0)^{-1}]\in\pi^1(Z^\bullet)$ is independent of the choice of $u_0$, and we take this to be $\delta(q_0)$.

It remains to check exactness. This is mostly an exercise in chasing definitions, and is left to the reader.
\end{proof}
\end{theorem}

\begin{remark}
Each of the successive sequences in theorem \ref{thm:long_exact_seqs} is an extension of the previous sequences. To be precise:
\begin{itemize}
	\item the first five terms in $(2)$ agree with the sequence assigned in $(1)$ to the inverse left-multiplication action of $U^\bullet$ on $Q^\bullet$ given by $u\colon q\mapsto u^{-1}q$, and the action of $\pi^0(U^\bullet)$ on $\pi^0(Q^\bullet)$ is also given by inverse left-multiplication;
	\item the first six terms in $(3)$ agree with the sequence assigned in $(2)$, with the action of $\pi^0(Q^\bullet)$ on $\pi^1(Z^\bullet)$ being given by multiplication via the connecting homomorphism $\pi^0(Q^\bullet)\rightarrow\pi^1(Z^\bullet)$;
	\item the first seven terms in $(4)$ agree with the sequence assigned in $(3)$, with the action of $\pi^1(Z^\bullet)$ on $\pi^1(U^\bullet)$ being given by multiplication.
\end{itemize}
\end{remark}

We will also need one instance where one of the exact sequences in theorem \ref{thm:long_exact_seqs} can be extended with a final $1$, which intuitively should be thought of as a dimensional vanishing result for the non-existent $\pi^2$. Specifically, if $U^\bullet$ is a cosimplicial group, we denote by $U^\bullet_{\leq n}$ the semi-cosimplicial group formed by truncating $U^\bullet$ above degree $n$. The canonical map $U^\bullet\twoheadrightarrow U^\bullet_{\leq n}$ of semi-cosimplicial groups factors uniquely through $U^\bullet\rightarrow\Kan^\bullet(U^\bullet_{\leq n})$, and we say that $U^\bullet$ is \emph{cogenerated in degree $\leq n$} if this factored map is injective in each degree. In particular, if $U^\bullet$ is the cosimplicial group cogenerated by a semi-cosimplicial group supported in degrees $\leq n$, then it is cogenerated in degree $\leq n$ in the above sense.

\begin{lemma}\label{lem:codimensional_vanishing}
Suppose\[1\rightarrow Z^\bullet\rightarrow U^\bullet\rightarrow Q^\bullet\rightarrow1\]is a sequence of cosimplicial groups which is exact in each degree, and that $Z^\bullet$ is cogenerated in degree $\leq1$. Then the exact sequence in point $(2)$ of theorem \ref{thm:long_exact_seqs} can be extended by a final map $\pi^1(Q^\bullet)\rightarrow1$, i.e.\ the map $\pi^1(U^\bullet)\rightarrow\pi^1(Q^\bullet)$ is surjective.
\begin{proof}
Pick $q_1\in\Cycle^1(Q^\bullet)$, and lift it to an element $u_1\in U^1$. The fact that $q_1$ is a cocycle tells us that $z_2:=d^2(u_1)^{-1}d^1(u_1)d^0(u_1)^{-1}$ lies in $Z^2$. Now by a quick computation with the cosimplicial identities, we see that $s^0(z_2)=d^1s^0(u_1)^{-1}$, $s^1(z_2)=d^0s^0(u_1)^{-1}$ and $s^0s^0(z_2)=s^0s^1(z_2)=s^0(u_1)^{-1}$. Hence, by replacing $u_1$ with $u_1s^0(z_2)$ (which still maps to $q_1\in Q^1$ since $s^0(z_2)\in Z^1$), we may assume that $s^0(z_2)=s^1(z_2)=1$ and $s^0s^0(z_2)=1$. But this says exactly that $z_2$ is contained in the kernel of $Z^2\rightarrow\Kan^2(Z^\bullet_{\leq1})$, so that $z_2=1$ by assumption that $Z^\bullet$ is cogenerated in degree $\leq1$. This then shows that $d^1(u_1)=d^2(u_1)d^0(u_1)$, so that $u_1$ is a cocycle, and hence $[q_1]\in\pi^1(Q^\bullet)$ is the image of $[u_1]\in\pi^1(U^\bullet)$ as desired.
\end{proof}
\end{lemma}

\subsection{Twisting}

Just as one uses Serre twists to get more precise information out of long exact sequences of group cohomology, one can formulate the notion of the \emph{twist} ${}_\beta U^\bullet$ of a cosimplicial group by to a $1$-cocycle $\beta\in\Cycle^1(U^\bullet)$ \cite[Section 2.4]{pirashvili}. This has the same objects and codegeneracy maps as $U^\bullet$, and the only coface maps which are changed are the $d^0\colon U^{n-1}\rightarrow U^n$, which are replaced with the twisted maps\[\bar{d^0}\colon u_{n-1}\mapsto d^n\cdots d^2(\beta)\cdot d^0(u_{n-1})\cdot d^n\cdots d^2(\beta)^{-1}.\]

The twisting has two key properties generalising those of Serre twists, namely that it doesn't change the first cohomotopy, and that twisting by a coboundary does not change the cosimplicial group.

\begin{proposition}\label{prop:pi^1_of_twists}\cite[Proposition 2.4.3]{pirashvili}
Let $U^\bullet$ be a cosimplicial group and $\beta\in\Cycle^1(U^\bullet)$. There is a canonical bijection $\pi^1({}_\beta U^\bullet)\isoarrow\pi^1(U^\bullet)$ given by right-multiplication by $\beta$, taking the distinguished point of $\pi^1({}_\beta U^\bullet)$ to $[\beta]\in\pi^1(U^\bullet)$.
\end{proposition}

\begin{proposition}\label{prop:trivial_twists}
Let $U^\bullet$ be a cosimplicial group and $\beta'=d^1(u_0)^{-1}\beta d^0(u_0)$ two cocycles differing by the action of some $u_0\in U^0$. Then there is an isomorphism ${}_{\beta'}U^\bullet\isoarrow{}_\beta U^\bullet$ canonically determined by $u_0$ such that the induced map on $\pi^1$ is the composite\[\pi^1\left({}_{\beta'}U^\bullet\right)\isoarrow\pi^1\left(U^\bullet\right)\leftisoarrow\pi^1\left({}_\beta U^\bullet\right),\]and in particular is independent of $u_0$.
\begin{proof}
The isomorphism is given on ${}_{\beta'}U^n$ by\[v_n\mapsto d^n\cdots d^1(u_0)\cdot v_n\cdot d^n\cdots d^1(u_0)^{-1}.\]This is clearly degree-wise an isomorphism, so to check it gives an isomorphism of cosimplicial groups one just has to check that the coface and codegeneracy maps are preserved. This is a straightforward exercise using the cosimplicial identities and  their consequence that
\begin{align*}
s^id^n\cdots d^1 &= d^{n-1}\cdots d^1, \text{ and} \\
d^id^n\cdots d^1 &= \begin{cases}d^{n+1}\cdots d^2d^1 & \text{if $i\neq 0$,} \\d^{n+1}\cdots d^2d^0 & \text{if $i=0$.}\end{cases}
\end{align*}
That the induced map on $\pi^1$ is as described is a straightforward calculation using the explicit description in proposition \ref{prop:pi^1_of_twists}.
\end{proof}
\end{proposition}

\subsection{Group cohomology}

As explained in \cite[Section 2.9]{pirashvili}, cohomotopy of cosimplicial groups is a generalisation of non-abelian group cohomology. Specifically, if $G$ is a topological group acting continuously (from the left) on another topological group $U$ (we call $U$ a \emph{topological $G$-group}), one constructs a cosimplicial group of continuous cochains $C^\bullet(G,U)$, where $C^n(G,U)=\Map_\cts(G^n,U)$, where the coface maps $d^i\colon C^{n-1}(G,U)\rightarrow C^n(G,U)$ and the codegeneracy maps $s^i\colon C^{n+1}(G,U)\rightarrow C^n(G,U)$ are given by
\begin{align*}
d^i(f)\colon (g_1,\dots,g_n) &\mapsto \begin{cases} g_1\cdot f(g_2,\dots,g_n) & \text{if $i=0$,} \\ f(g_1,\dots,g_ig_{i+1},\dots,g_n) & \text{if $0<i<n$,} \\ f(g_1,\dots,g_{n-1}) & \text{if $i=n$;} \end{cases} \\
s^i(f)\colon (g_1,\dots,g_n) &\mapsto f(g_1,\dots,g_{i-1},1,g_i,\dots,g_n).
\end{align*}
By definition of non-abelian group cohomology, we have equalities\[\pi^i(C^\bullet(G,U))=\H^i(G,U)\]for $i=0,1$, and for $i>1$ when $U$ is abelian (in fact, we even have equalities on the level of cocycles). Moreover, these descriptions are compatible with twisting, in that for a continuous cocycle $\alpha\in\Cycle^1(G,U)=\Cycle^1(C^\bullet(G,U))$, the cosimplicial group $C^\bullet(G,{}_\alpha U)$ of continuous cochains for the Serre twist ${}_\alpha U$ is equal to the twisted cosimplicial group ${}_\alpha C^\bullet(G,U)$, and the canonical isomorphism $\H^1(G,{}_\alpha U)\isoarrow\H^1(G,U)$ is just the canonical isomorphism on $\pi^1$ described in proposition \ref{prop:pi^1_of_twists}.

With this in mind, we see that theorem \ref{thm:long_exact_seqs} is a generalisation of various well-known long exact sequences in non-abelian group cohomology, as for example in \cite[Section I.5]{serre}. Since the statements there only deal with the case of discrete topological $G$-groups, we will trouble the reader with a sufficiently general statement for our applications.

\begin{corollary}[to theorem \ref{thm:long_exact_seqs}]\label{cor:long_exact_seq_group_cohomology}
Let $G$ be a topological group.
\begin{enumerate}
	\item Suppose $U$ is a topological $G$-group with acting continuously, $G$-equivariantly and transitively on a topological $G$-set $Q$ from the right, and $Z\hookrightarrow U$ is the stabiliser of a $G$-fixed basepoint $*\in Q$, given the subspace topology from $U$. Then there is a functorially assigned exact sequence\[1\rightarrow Z^G\rightarrow U^G\actsarrow Q^G\rightarrow\H^1(G,Z)\rightarrow\H^1(G,U).\]
	\item Suppose\[1\rightarrow Z\rightarrow U\rightarrow Q\rightarrow1\]is an exact sequence of topological $G$-groups, such that $Z\subseteq U$ has the subspace topology and $U\rightarrow Q$ has a continuous splitting (as topological spaces). Then there is a functorially assigned exact sequence\[1\rightarrow Z^G\rightarrow U^G\rightarrow Q^G\actsarrow\H^1(G,Z)\rightarrow\H^1(G,U)\rightarrow\H^1(G,Q).\]
	\item If in the previous point $Z$ is central in $U$, then the exact sequence can be extended to a functorially assigned sequence
	\[1\rightarrow Z^G\centarrow U^G\rightarrow Q^G\rightarrow\H^1(G,Z)\actsarrow\H^1(G,U)\rightarrow\H^1(G,Q)\rightarrow\H^2(G,Z).\]
	\item If in the previous point $U$ is abelian, then the exact sequence can be extended to a functorially assigned exact sequence of abelian groups\[\cdots\rightarrow\H^{r-1}(G,Q)\rightarrow\H^r(G,Z)\rightarrow\H^r(G,U)\rightarrow\H^r(G,Q)\rightarrow\H^{r+1}(G,Z)\rightarrow\cdots.\]
\end{enumerate}
\begin{proof}
For points $(2)$--$(4)$, the given conditions ensure that the sequence\[1\rightarrow C^\bullet(G,Z)\rightarrow C^\bullet(G,U)\rightarrow C^\bullet(G,Q)\rightarrow1\]of cosimplicial groups satisfies the conditions for the corresponding point of theorem \ref{thm:long_exact_seqs}. For point $(1)$, the given conditions ensure that $C^\bullet(G,Z)$ is the stabiliser of the basepoint of $C^\bullet(G,Q)$ under the action of $C^\bullet(G,U)$, and that the action of $C^\bullet(G,U)$ on $C^\bullet(G,Q)$ is transitive in degree $0$. This was all our proof of $(1)$ required, so we are done.
\end{proof}
\end{corollary}


\subsection{Cosimplicial algebras and cdgas}

In order to construct the cosimplicial groups that will be of interest later, it will be necessary to have some tools to allow us to define certain cosimplicial algebras (the cosimplicial group then being produced by evaluating a group scheme on these algebras). In many cases, the construction in definition \ref{def:cosimplicial_objects} of a cosimplicial algebra from a semi-cosimplicial algebra will suffice, but a few of our constructions are best described in terms of (graded\nobreakdash-)commutative differential graded algebras following the cosimplicial monoidal Dold--Kan correspondence.

Recall from \cite{castiglioni-cortinas} that if $A^\bullet$ is a dga over some base ring $R$, then we can identify the terms of the cosimplicial $R$-module $\Kan^\bullet(A^\bullet)$ corresponding to $A^\bullet$ under the Dold--Kan correspondence as\[\Kan^n(A^\bullet)\cong\bigoplus_j\bigwedge\!{}^jR^n\otimes_R A^j\]and that this inherits a natural $R$-algebra structure\footnote{\cite{castiglioni-cortinas} works only over the base ring $\Z$, but it is easy to see that the ring structure they define is automatically an $R$-algebra structure when the dg-ring is replaced by an $R$-dga} from the algebra structures on $A^\bullet$ and $\bigwedge^\bullet R^n$. Moreover, when $A^\bullet$ is graded-commutative, it is easy to see that this algebra structure is commutative, so that according to \cite{castiglioni-cortinas} $\Kan^\bullet(A^\bullet)$ is then a cosimplicial (commutative) $R$-algebra.

\begin{remark}
For our purposes, we will treat this construction simply as a tool for producing cosimplicial algebras, but it is worth recalling that this fits into the wider context of the \emph{cosimplicial monoidal Dold--Kan correspondence}. Specifically, the denormalisation functor\[\Kan^\bullet\colon\CoCh_+(R)\rightarrow\Mod_R^\Delta\]is a lax symmetric monoidal functor, meaning that there is a natural transformation called the \emph{Eilenberg--Zilber} or \emph{shuffle} map\footnote{\emph{Eilenberg--Zilber} because the corresponding map $N^\bullet(\Kan^\bullet(A^\bullet)\otimes\Kan^\bullet(B^\bullet))\rightarrow N^\bullet\Kan^\bullet(A^\bullet\otimes B^\bullet)\cong A^\bullet\otimes B^\bullet\cong N^\bullet\Kan^\bullet(A^\bullet)\otimes N^\bullet\Kan^\bullet(B^\bullet)$ is the dual map to the simplicial Eilenberg--Zilber map, and \emph{shuffle} because the simplicial Eilenberg--Zilber map has an explicit description in terms of shuffles \cite[Theorem VIII.8.5]{maclane}.}\[\EZ\colon\Kan^\bullet(A^\bullet)\otimes\Kan^\bullet(B^\bullet)\rightarrow\Kan^\bullet(A^\bullet\otimes B^\bullet)\]compatible with interchange of $A^\bullet$ and $B^\bullet$ (with respect to the usual symmetric monoidal structure on both sides). For purely abstract reasons, this means that $\Kan^\bullet$ takes (c)dgas to cosimplicial (commutative) $R$-algebras, whose multiplication is given by the composite\[\Kan^\bullet(A^\bullet)\otimes\Kan^\bullet(A^\bullet)\overset{\EZ}\longrightarrow\Kan^\bullet(A^\bullet\otimes A^\bullet)\overset{\Kan^\bullet(\mathrm{mult})}\longrightarrow\Kan^\bullet(A^\bullet).\]

However, unlike the classical Dold--Kan correspondence, $\Kan^\bullet$ does not induce an equivalence of categories, and (for our purposes) worse, the canonical dga structure on $N^\bullet(A^\bullet)$ for a cosimplicial $R$-algebra $A^\bullet$ can fail to be commutative, even when $A^\bullet$ is (since the monoidal structure on $N^\bullet$, induced from the dual Alexander--Whitney map, is not symmetric). In order to circumvent these problems properly, one needs to pass to appropriate homotopy categories on either side, and replace $N^\bullet$ with the \emph{functor of Thom-Sullivan cochains}, as explained in \cite[Section 1.1]{katzarkov-pantev-toen}.
\end{remark}

\begin{example}\label{ex:B_st[epsilon]}
The particular $\Q_p$-cdga we will be interested in is $\B_\st[\boldsymbol\varepsilon]^\bullet$, whose underlying cochain complex is defined to be\[\B_\st\overset N\rightarrow\B_\st\]and whose graded algebra structure is the one induced by the obvious $\B_\st$-module structure on $\B_\st$.

Following through the definition of the algebra structure in \cite{castiglioni-cortinas}, we see that the objects of its associated cosimplicial topological $\Q_p$-algebra are\[\Kan^n(\B_\st[\boldsymbol\varepsilon]^\bullet)=\B_\st[\varepsilon_1,\dots,\varepsilon_n]\]where $\varepsilon_j$ are variables subject to the relation $\varepsilon_j\varepsilon_k=0$ for all $j,k$.

The cdga $\B_\st[\boldsymbol\varepsilon]^\bullet$ carries certain extra structures: a topology, continuous $G_K$-action, and continuous $G_K$-equivariant semilinear crystalline Frobenius~$\varphi^\bullet$. The topology and $G_K$-action are just those induced from the topology and action on~$\B_\st$. The crystalline Frobenius on $\B_\st[\boldsymbol\varepsilon]^\bullet$ is a little more subtle: it acts as the usual crystalline Frobenius~$\varphi$ on $\B_\st[\boldsymbol\varepsilon]^0=\B_\st$, but as $p\varphi$ on $\B_\st[\boldsymbol\varepsilon]^1=\B_\st$. The appearance of the extra factor of~$p$ here derives from the identity $N\varphi=p\varphi N$ satisfied by the crystalline Frobenius on~$\B_\st$: the factor of~$p$ is necessary to ensure that~$\varphi^\bullet$ is a cdga automorphism.

These extra structures induce corresponding structures on $\Kan^\bullet(\B_\st[\boldsymbol\varepsilon]^\bullet)$: it is a cosimplicial topological $\Q_p$-algebra endowed with a continuous action of~$G_K$ and a continuous $G_K$-equivariant semilinear crystalline Frobenius~$\varphi$. On $\Kan^n(\B_\st[\boldsymbol\varepsilon]^\bullet)=\B_\st[\varepsilon_1,\dots,\varepsilon_n]$, this crystalline Frobenius is given by
\[
\varphi\left(a+\sum_{i=1}^na_i\varepsilon_i\right) = \varphi(a)+p\sum_{i=1}^n\varphi(a_i)\varepsilon_i
\]
with~$\varphi$ also denoting the usual crystalline Frobenius on~$\B_\st$.
\end{example}

Finally, we will also need to hybridise our two constructions to produce cosimplicial algebras out of semi-cosimplicial cdgas $A^{\bullet,\bullet}$.

\begin{construction}\label{cons:cosimplicial_algebras_from_semi-cosimplicial_cdgas}
Let $A^{\bullet,\bullet}$ be a semi-cosimplicial $R$-cdga. We produce a cosimplicial $R$-algebra $d^\bullet(\Kan^{\bullet,\bullet}(A^{\bullet,\bullet}))$ as follows. We start by replacing $A^{\bullet,\bullet}$ with the cosimplicial $R$-cdga cogenerated by it as in definition \ref{def:cosimplicial_objects}, and then replacing each of its terms by its associated cosimplicial $R$-algebra as above. This produces us a bi-cosimplicial $R$-algebra $\Kan^{\bullet,\bullet}(A^{\bullet,\bullet})$, and we let $d^\bullet(\Kan^{\bullet,\bullet}(A^{\bullet,\bullet}))$ be the restriction to the diagonal.

It is easy to see that this simultaneously generalises both the earlier constructions, corresponding to the case when each of the individual cdgas is concentrated in degree $0$, and the case when the semi-cosimplicial cdga itself is concentrated in degree $0$.
\end{construction}
\section{Unipotent Bloch--Kato theory}
\label{c:bloch-kato}

Our aim in this section is to explain in detail how one can use cosimplicial groups and cosimplicial algebras to extend the analysis of local Bloch--Kato Selmer groups to the pro-unipotent setting. All of our arguments are then essentially just rewritings of abelian arguments in such a way that they apply equally well in the unipotent setting, and the reader who is less familiar with the non-abelian theory should be able to follow the trace of these unwritten abelian arguments in the proofs below. As remarked in the introduction, the unifying theme of this section will be the definition of certain cosimplicial $\Q_p$-algebras $\B^\bullet_*$ for $*\in\{e,f,g,g/e,f/e\}$ such that for a de Rham representation of $G_K$ on a finitely generated pro-unipotent group $U/\Q_p$, the corresponding local Bloch--Kato Selmer set or quotient $\H^1_*(G_K,U(\Q_p))$ is computed by the first cohomotopy of the cosimplicial group $U(\B^\bullet_*)^{G_K}$.

From the point of view of our application, the main result in this section is corollary \ref{cor:long_exact_seqs_of_quotients}, which gives a long exact sequence relating the Bloch--Kato quotients $\H^1_{g/e}$ for central extensions of representations on finitely generated pro-unipotent groups, and plays the same role in the proof of theorem \ref{thm:main_theorem_p-adic} as the long exact sequence in non-abelian Galois cohomology does in the proof of theorem \ref{thm:main_theorem_l-adic} in \S\ref{c:main_theorem_l-adic}.

\subsection{The cosimplicial algebras $\B^\bullet_e$, $\B^\bullet_f$ and $\B^\bullet_g$}

When $V$ is a de Rham (linear) representation of $G_K$, the Bloch--Kato Selmer group $\H^1_e(G_K,V)$ is typically studied via the Bloch--Kato exponential sequence
\[
0 \to V^{G_K} \to \D_\dR^+(V)\oplus\D_\cris^{\varphi=1}(V) \to \D_\dR(V) \to \H^1_e(G_K,V) \to 0,
\]
which  arises by tensoring $V$ with the short exact sequence
\[
0\rightarrow\Q_p\rightarrow\B_\dR^+\oplus\B_\cris^{\varphi=1}\rightarrow\B_\dR\rightarrow0
\]
and taking the long exact sequence in Galois cohomology \cite[Corollary 3.8.4]{bloch-kato}. Put another way, if one defines the $2$-term cochain complex $\mathsf C^\bullet_e$ to be\[\mathsf C^\bullet_e := (\B_\dR^+\oplus\B_\cris^{\varphi=1}\rightarrow\B_\dR)\](with the map sending $(x,y)$ to $x-y$), then one has a canonical identification\[\H^1\!\left((\mathsf C^\bullet_e\otimes V)^{G_K}\right)\cong\coker\left(\D_\dR^+(V)\oplus\D_\cris^{\varphi=1}(V)\to\D_\dR(V)\right)=\H^1_e(G_K,V).\]

The advantage of this latter cohomological point of view, as noted in \cite[Theorem 1.23]{nekovar}, is that one can treat both $\H^1_f(G_K,V)$ and $\H^1_g(G_K,V)$ on an equal footing by defining cochain complexes $\mathsf C^\bullet_f$ and $\mathsf C^\bullet_g$ by
\begin{align*}
\mathsf C^\bullet_f &:= (\B_\dR^+\oplus\B_\cris\rightarrow\B_\dR\oplus\B_\cris),\text{ and} \\
\mathsf C^\bullet_g &:= (\B_\dR^+\oplus\B_\st\rightarrow\B_\dR\oplus\B_\st\oplus\B_\st\rightarrow\B_\st),
\end{align*}
where the map in the former complex sends $(x,y)$ to $(x-y,(\varphi-1)y)$, and in the latter complex the first map sends $(x,y)$ to $(x-y,(\varphi-1)y,Ny)$ and the second map sends $(x,y,z)$ to $(p\varphi-1)z-Ny$. With these definitions, one then has a canonical identification \cite[Theorem 1.23]{nekovar}\[\H^1\!\left((\mathsf C^\bullet_g\otimes V)^{G_K}\right)\cong\H^1_g(G_K,V),\]and similarly for $\H^1_f$. Thus, in all cases, the cochain complex $(\mathsf C^\bullet_*\otimes V)^{G_K}$ is a homological-algebraic invariant associated to the representation $V$ which refines the Bloch--Kato Selmer groups $\H^1_*(G_K,V)$.

\begin{remark}\label{rmk:cochains_acyclic}
All three cochain complexes $\mathsf C^\bullet_e$, $\mathsf C^\bullet_f$, $\mathsf C^\bullet_g$ defined above are resolutions of $\Q_p$: their cohomology is $\Q_p$ in degree $0$ and vanishes in higher degrees. For $\mathsf C^\bullet_e$ and $\mathsf C^\bullet_f$, this is simply a rephrasing of the two well-known exact sequences \cite[Proposition 1.17]{bloch-kato}, except that we are using different conventions for the signs of the morphisms, while for $\mathsf C^\bullet_g$ the surjectivity of $N\colon\B_\st\rightarrow\B_\st$ reduces it to the case of $\mathsf C^\bullet_f$.
\end{remark}

In order to generalise this idea to the unipotent setting, we will replace the cochain complexes $\mathsf C^\bullet_*$ with certain cosimplicial $\Q_p$-algebras $\B^\bullet_*$ with $G_K$-action, such that when $U$ is a de Rham representation of $G_K$ on a finitely generated pro-unipotent group, the cosimplicial group $U(\B^\bullet_*)^{G_K}$ recovers the non-abelian Bloch--Kato Selmer set $\H^1_*(G_K,U(\Q_p))$ as its first cohomotopy. The algebras $\B^\bullet_*$ themselves are quite straightforward to define, being simply built out of copies of the standard period rings.

\begin{definition}\label{def:bloch-kato_algebras}
We define three cosimplicial $\Q_p$-algebras $\B^\bullet_e$, $\B^\bullet_f$ and $\B^\bullet_g$ as follows:
\begin{itemize}
	\item $\B^\bullet_e$ is the cosimplicial $\Q_p$-algebra cogenerated by the semi-cosimplicial $\Q_p$-algebra\[\B_\dR^+\times\B_\cris^{\varphi=1}\rightrightarrows\B_\dR\](with higher terms the zero ring) where $d^0(x,y)=x$ and $d^1(x,y)=y$;
	\item $\B^\bullet_f$ is the cosimplicial $\Q_p$-algebra cogenerated by the semi-cosimplicial $\Q_p$-algebra\[\B_\dR^+\times\B_\cris\rightrightarrows\B_\dR\times\B_\cris\]where $d^0(x,y)=(x,\varphi y)$, $d^1(x,y)=(y,y)$;
	\item $\B^\bullet_g$ is the cosimplicial $\Q_p$-algebra constructed from the semi-cosimplicial $\Q_p$-cdga\[\B_\dR^+\times\B_\st[\boldsymbol\varepsilon]^\bullet\rightrightarrows\B_\dR\times\B_\st[\boldsymbol\varepsilon]^\bullet\]as in construction \ref{cons:cosimplicial_algebras_from_semi-cosimplicial_cdgas}. Here $\B_\st[\boldsymbol\varepsilon]^\bullet$ is the $\Q_p$-cdga in example \ref{ex:B_st[epsilon]}, and the maps in the semi-cosimplicial $\Q_p$-cdga are given by $d^0(x,y)=(x,\varphi y)$ and $d^1(x,y)=(y,y)$.
\end{itemize}
All of these cosimplicial $\Q_p$-algebras inherit a topology and an action of $G_K$ from those on $\B_\dR$, $\B_\st$ and $\B_\cris$. This action is degree-wise continuous.
\end{definition}


One important property of the algebras $\B^\bullet_e$, $\B^\bullet_f$ and $\B^\bullet_g$ is that they are resolutions of $\Q_p$, in the sense that their $0$th cohomotopy is canonically identified with $\Q_p$ (as a topological $\Q_p$-algebra with $G_K$-action), and their higher cohomotopy vanishes. Moreover, we will also see that this immediately implies that $U(\B^\bullet_*)$ is a resolution of $U(\Q_p)$ whenever $U/\Q_p$ is a finitely generated pro-unipotent group -- this is analogous to the fact that tensoring with a linear representation is exact.

\begin{proposition}\label{prop:acyclic_algebras}
We have $\pi^0(\B_e^\bullet)=\pi^0(\B_f^\bullet)=\pi^0(\B_g^\bullet)=\Q_p$ (as a topological $\Q_p$-algebra with $G_K$-action), and $\pi^i(\B^\bullet_e)=\pi^i(\B^\bullet_f)=\pi^i(\B^\bullet_g)=0$ for $i\geq1$ (as a $\Q_p$-vector space/abelian group).
\begin{proof}
For $\B^\bullet_e$ and $\B^\bullet_f$, by lemma \ref{lem:dold-kan} their cohomotopy groups agree with the cohomotopy groups of the generating semi-cosimplicial $\Q_p$-algebra, which are, respectively, the cohomology groups of the cochains $\mathsf C^\bullet_e$ and $\mathsf C^\bullet_f$ above. But these are $\Q_p$ in degree $0$ and vanish in higher degrees by remark \ref{rmk:cochains_acyclic}.

For $\B^\bullet_g$, by lemma \ref{lem:eilenberg-zilber} its cohomotopy groups agree with the cohomotopy groups of the totalisation of the bicomplex
\begin{center}
\begin{tikzcd}[ampersand replacement=\&,column sep = large]
\B_\dR^+\oplus\B_\st \arrow[r,"{\begin{pmatrix} 1 & -1 \\ 0 & \varphi-1 \end{pmatrix}}"]\arrow[d,"(0\:N)"] \& \B_\dR\oplus\B_\st \arrow[d,"(0\:N)"] \\
\B_\st \arrow[r,"p\varphi-1"] \& \B_\st,
\end{tikzcd}
\end{center}
namely $\mathsf C^\bullet_g$. Again, this is a resolution of $\Q_p$ by remark \ref{rmk:cochains_acyclic}.

One checks easily that the algebra structure, topology and $G_K$-action on $\pi^0(\B^\bullet_e)$, $\pi^0(\B^\bullet_f)$ and $\pi^0(\B^\bullet_g)$ agree with those on $\Q_p$.
\end{proof}
\end{proposition}

\begin{lemma}\label{lem:acyclic_groups}
Let $U/\Q_p$ be a finitely generated pro-unipotent group, and $\B^\bullet$ a cosimplicial topological $\Q_p$-algebra such that $\pi^0(\B^\bullet)=\Q_p$ (as a topological $\Q_p$-algebra) and $\pi^i(\B^\bullet)=0$ for $i>0$ (as a $\Q_p$-vector space). Then $\pi^0(U(\B^\bullet))= U(\Q_p)$ (as a topological group), $\pi^1(U(\B^\bullet))=1$, and $\pi^i(U(\B^\bullet))=1$ for $i>1$ when $U$ is abelian.
\begin{proof}
If $U$ is abelian, then it is $\G_a^d$ for some $d$, so it suffices to consider the case $U=\G_a$. But then $U(\B^\bullet)\cong\B^\bullet$ as cosimplicial topological abelian groups, so the result follows by assumption.

For $U$ non-abelian and finite-dimensional, we proceed by induction, writing $U$ as a central extension\[1\rightarrow Z\centarrow U\rightarrow Q\rightarrow1\]where we already know the result for $Z$ and $Q$. Since the sequence\[1\rightarrow Z(\B^\bullet)\centarrow U(\B^\bullet)\rightarrow Q(\B^\bullet)\rightarrow1\]is a central extension in each degree, we have by \ref{thm:long_exact_seqs} an exact sequence\[1\rightarrow\pi^0(Z(\B^\bullet))\centarrow\pi^0(U(\B^\bullet))\rightarrow\pi^0(Q(\B^\bullet))\rightarrow\pi^1(Z(\B^\bullet))\actsarrow\pi^1(U(\B^\bullet))\rightarrow\pi^1(Q(\B^\bullet))\]
which by our inductive assumption gives the desired description of $\pi^0(U(\B^\bullet))$ and $\pi^1(U(\B^\bullet))$. That the topology on $\pi^0(U(\B^\bullet))$ is the desired one follows, since after identifying $U$ with its Lie algebra $\Lie(U)$ we see that $\pi^0(U(\B^\bullet))$ is homeomorphic to $\pi^0\left(\B^\bullet\otimes\Lie(U)\right)$, which is homeomorphic to $\Lie(U)\cong U(\Q_p)$ by the abelian case.

For general $U$, we write $U=\liminv U_n$ as an inverse limit of finite-dimensional quotients. Since $\pi^0$ commutes with inverse limits of topological groups, we have the desired description of $\pi^0(U(\B^\bullet))$. To prove triviality of $\pi^1(U(\B^\bullet))$, it suffices to prove that the map $\pi^1(U(\B^\bullet))\rightarrow\liminv\pi^1(U_n(\B^\bullet))$ has trivial kernel. Thus suppose $\alpha\in\Cycle^1(U(\B^\bullet))=\liminv\Cycle^1(U(\B^\bullet))$ is an inverse limit of coboundaries. For each $n$, we write $T_n$ for the set of elements of $U_n(\B^0)$ whose coboundary is the image of $\alpha$ in $\Cycle^1(U(\B^\bullet))$. $T_n$ is then a left torsor under $\pi^0(U_n(\B^\bullet))=U_n(\Q_p)$, so that the maps $T_m\rightarrow T_n$ for $m\geq n$ are all surjective and $\liminv T_n\neq\emptyset$. An element of the inverse limit is then an element $u\in\liminv U_n(\B^0)=U(\B^0)$ whose coboundary is $\alpha$, so that $[\alpha]=*$ as desired.
\end{proof}
\end{lemma}

\subsection{The non-abelian Bloch--Kato exponential}
\label{s:exponential}

Combining proposition \ref{prop:acyclic_algebras} and lemma \ref{lem:acyclic_groups} shows that, for $U/\Q_p$ a representation of $G_K$ on a finitely generated pro-unipotent group, there is a $G_K$-equivariant exact sequence\[1\rightarrow U(\Q_p)\rightarrow U(\B^0_*)\actsarrow\Cycle^1(U(\B^\bullet_*))\rightarrow1\]for $*\in\{e,f,g\}$. When $*=e$, the terms are given explicitly by\[1\rightarrow U(\Q_p)\rightarrow U(\B_\dR^+)\times U(\B_\cris^{\varphi=1})\actsarrow U(\B_\dR)\rightarrow1\](where the group structure on $U(\B_\dR)$ plays no role). This suggests, following the construction in the abelian case, that a non-abelian analogue of the Bloch--Kato exponential sequence could be obtained by taking the long exact sequence on non-abelian Galois cohomology.

In order to perform this construction, we will need the following non-abelian variant of \cite[Lemma 3.8.1]{bloch-kato}.

\begin{proposition}\label{prop:cohomological_injectivity}
Let $U/\Q_p$ be a de Rham representation of $G_K$ on a finitely generated pro-unipotent group. Then the map $\H^1(G_K,U(\B_\dR^+))\rightarrow\H^1(G_K,U(\B_\dR))$ has trivial kernel.
\begin{proof}
When $U$ is finite-dimensional, we may proceed inductively. The case that $U$ is abelian is \cite[Lemma 3.8.1]{bloch-kato}, so we need only consider the case when $U$ is a $G_K$-equivariant central extension\[1\rightarrow Z\centarrow U\rightarrow Q\rightarrow1\]and we already know the lemma for $Z$ and $Q$. Since all the varieties involved are affine spaces, this sequence has a splitting as varieties (not necessarily compatible with the group structure), so that $Z(\B_\dR^+)\subseteq U(\B_\dR^+)$ has the subspace topology and $U(\B_\dR^+)\rightarrow Q(\B_\dR^+)$ has a continuous splitting, and similarly for $\B_\dR$. Hence we obtain (part of) an exact sequence in cohomology:
\begin{center}
\begin{tikzcd}
\H^1(G_K,Z(\B_\dR^+)) \arrow[r,"\acts"]\arrow[d] & \H^1(G_K,U(\B_\dR^+)) \arrow[r]\arrow[d] & \H^1(G_K,Q(\B_\dR^+)) \arrow[d] \\
\H^1(G_K,Z(\B_\dR)) \arrow[r,"\acts"] & \H^1(G_K,U(\B_\dR)) \arrow[r] & \H^1(G_K,Q(\B_\dR)).
\end{tikzcd}
\end{center}

However, the assumption that $U$ is de Rham implies that $Z$ and $Q$ are too, and hence that the sequence\[1\rightarrow\D_\dR(Z)\centarrow\D_\dR(U)\rightarrow\D_\dR(Q)\rightarrow1\]is exact. Thus the coboundary map $\D_\dR(Q)(K)=\H^0(G_K,Q(\B_\dR))\rightarrow\H^1(G_K,Z(\B_\dR))$ is zero and $\H^1(G_K,Z(\B_\dR))\rightarrow\H^1(G_K,U(\B_\dR))$ has trivial kernel.

With this observation, it is a straightforward diagram-chase to verify as desired that $\H^1(G_K,U(\B_\dR^+))\rightarrow\H^1(G_K,U(\B_\dR))$ has trivial kernel. If $u_1\in\H^1(G_K,U(\B_\dR^+))$ is in the kernel, then its image in $\H^1(G_K,Q(\B_\dR^+))$ lies in the corresponding kernel, so that inductively this image is trivial and hence $u_1$ is the image of some $z_1\in\H^1(G_K,Z(\B_\dR^+))$. Since the map $\H^1(G_K,Z(\B_\dR))\rightarrow\H^1(G_K,U(\B_\dR))$ has trivial kernel, it follows that $z_1=*$ and hence $u_1=*$ also.

For general $U$, we write $U=\liminv U_n$ as an inverse limit of finite-dimensional $G_K$-equivariant quotients, which are de Rham. Since the composite\[\H^1(G_K,U(\B_\dR^+))\rightarrow\liminv\H^1(G_K,U_n(\B_\dR^+))\rightarrow\liminv\H^1(G_K,U_n(\B_\dR))\](where the latter map has trivial kernel by the finite-dimensional case) factors through $\H^1(G_K,U(\B_\dR^+))\rightarrow\H^1(G_K,U(\B_\dR))$, it would suffice for us to prove that the map $\H^1(G_K,U(\B_\dR^+))\rightarrow\liminv\H^1(G_K,U_n(\B_\dR^+))$ has trivial kernel.

To do this, suppose that $\alpha\in\Cycle^1(G_K,U(\B_\dR^+))$ is a continuous cocycle whose image in each $\Cycle^1(G_K,U_n(\B_\dR^+))$ is a coboundary, and write $T_n$ for the set of all elements of $U_n(\B_\dR^+)$ whose coboundary is the image of $\alpha$. Then $T_n$ is a left torsor under $U_n(\B_\dR^+)^{G_K}=\D_\dR^+(U_n)(K)$. Since the $\D_\dR^+(U_n)$ are unipotent groups over $K$, it follows that the sets $(T_n)_{n\in\N}$ satisfy the Mittag--Leffler condition, and hence $\liminv T_n\neq\emptyset$. By inspection, an element of $\liminv T_n\subseteq\liminv U_n(\B_\dR^+)=U(\B_\dR^+)$ is an element whose coboundary is $\alpha$, so that $[\alpha]\in\H^1(G_K,U(\B_\dR^+))$ is trivial, as desired.
\end{proof}
\end{proposition}

\begin{corollary}\label{cor:pi^1_is_bloch-kato}
Let $U/\Q_p$ be a de Rham representation of $G_K$ on a finitely generated pro-unipotent group, and let $\Cycle^1(U(\B^\bullet_g))$ be the $1$-cocycles of the cosimplicial group $U(\B^\bullet_g)$. Then the twisted conjugation action of $U(\B^0_g)$ on $\Cycle^1(U(\B^\bullet_g))$ is transitive with point-stabiliser $U(\Q_p)$, so that by corollary \ref{cor:long_exact_seq_group_cohomology} we have an exact sequence\[1\rightarrow U(\Q_p)^{G_K}\rightarrow U(\B^0_g)^{G_K}\actsarrow\Cycle^1(U(\B^\bullet_g))^{G_K}\overset\delta\rightarrow\H^1(G_K,U(\Q_p))\rightarrow\H^1(G_K,U(\B^0_g)).\]Furthermore, the image of the coboundary map is the Bloch--Kato cohomology set $\H^1_g(G_K,U(\Q_p))$. The same is true of $\B^\bullet_f$ (for $\H^1_f$) and $\B^\bullet_e$ (for $\H^1_e$).
\begin{proof}
We will just prove this for $\B_g^\bullet$, the other cases being similar. By proposition \ref{prop:acyclic_algebras}, the cosimplicial topological $\Q_p$-algebra $\B_g^\bullet$ satisfies the conditions of lemma \ref{lem:acyclic_groups}, which tells us that the action of $U(\B^0_g)$ on $\Cycle^1(U(\B^\bullet_g))$ is transitive, with point-stabiliser $U(\Q_p)$, so that we have the claimed exact sequence.

To conclude, we want to determine the kernel of the right-hand map in the exact sequence. Since $\B^0_g=\B_\dR^+\times\B_\st$, the kernel of this map is equal to the intersection of the kernels of the two maps
\[
\H^1(G_K,U(\Q_p)) \rightarrow \H^1(G_K,U(\B_\dR^+)) \hspace{0.4cm}\text{and}\hspace{0.4cm} \H^1(G_K,U(\Q_p)) \rightarrow \H^1(G_K,U(\B_\st)).
\]
But it follows from Proposition~\ref{prop:cohomological_injectivity} and Lemma~\ref{lem:our_H^1_g_is_right} that these two maps have the same kernel (which is also the kernel of $\H^1(G_K,U(\Q_p))\rightarrow\H^1(G_K,U(\B_\dR))$). Hence the kernel of $\H^1(G_K,U(\Q_p))\to\H^1(G_K,U(\B_g^0))$ is $\H^1_g(G_K,U(\Q_p))$ as desired.
\end{proof}
\end{corollary}

A convenient way to view the non-abelian Bloch--Kato exponential is as giving a canonical description of the cohomotopy of $U(\B^\bullet_*)^{G_K}$, as follows.

\begin{theorem}\label{thm:bloch-kato_sets_as_cohomotopy}
Let $U/\Q_p$ be a de Rham representation of $G_K$ on a finitely generated pro-unipotent group. Then the cohomotopy groups/sets of the cosimplicial groups $U(\B^\bullet_e)^{G_K}$, $U(\B^\bullet_f)^{G_K}$ and $U(\B^\bullet_g)^{G_K}$ are canonically identified as
\begin{align*}
\pi^i(U(\B^\bullet_e)^{G_K}) &= \begin{cases} U(\Q_p)^{G_K} & \text{if $i=0$,} \\ \H^1_e(G_K,U(\Q_p)) & \text{if $i=1$,} \\ 1 & \text{if $i>1$ and $U$ abelian;}\end{cases} \\
\pi^i(U(\B^\bullet_f)^{G_K}) &= \begin{cases} U(\Q_p)^{G_K} & \text{if $i=0$,} \\ \H^1_f(G_K,U(\Q_p)) & \text{if $i=1$,} \\ 1 & \text{if $i>1$ and $U$ abelian;}\end{cases} \\
\pi^i(U(\B^\bullet_g)^{G_K}) &= \begin{cases} U(\Q_p)^{G_K} & \text{if $i=0$,} \\ \H^1_g(G_K,U(\Q_p)) & \text{if $i=1$,} \\ \D_\cris^{\varphi=1}(U(\Q_p)^\dual(1))^\dual & \text{if $i=2$ and $U$ abelian,} \\ 1 & \text{if $i>2$ and $U$ abelian.}\end{cases}
\end{align*}
\begin{proof}
Let us do the case of $U(\B^\bullet_g)^{G_K}$, the other cases being similar (and simpler, since one need only work with $\Q_p$-algebras rather than $\Q_p$-cdgas). For $\pi^0$ and $\pi^1$, this is exactly the content of corollary \ref{cor:pi^1_is_bloch-kato} since $\Cycle^1(U(\B^\bullet_g))^{G_K}=\Cycle^1(U(\B^\bullet_g)^{G_K})$.

It remains to calculate the higher cohomotopy when $U=V$ is abelian. In this case, $U(\B^\bullet_g)=\B^\bullet_g\otimes V$ is the diagonal in the bi-cosimplicial $\Q_p$-representation cogenerated by the bi-semi-cosimplicial $\Q_p$-representation
\begin{center}
\begin{tikzcd}
(\B_\dR^+\otimes V)\oplus(\B_\st\otimes V) \arrow[r,shift right]\arrow[r,shift left]\arrow[d,swap,"(0\:0)",shift right]\arrow[d,"(0\:-N)",shift left] & (\B_\dR\otimes V)\oplus(\B_\st\otimes V) \arrow[d,swap,"(0\:0)",shift right]\arrow[d,"(0\:-N)",shift left] \\
\B_\st\otimes V \arrow[r,shift right]\arrow[r,shift left] & \B_\st\otimes V.
\end{tikzcd}
\end{center}
Since the operations of passing to the associated bi-cosimplicial representation and restricting to the diagonal commute with taking $G_K$-invariants, the same remains true after taking $G_K$-invariants, and hence by lemma \ref{lem:eilenberg-zilber} the cohomotopy of $(\B^\bullet_g\otimes V)^{G_K}$ is canonically isomorphic to the cohomology of the cochain $(\mathsf C^\bullet_g\otimes V)^{G_K}$, where $\mathsf C^\bullet_g$ is the cochain at the beginning of \S\ref{s:exponential}. Concretely, this is a cochain complex\[\D_\dR^+(V)\oplus\D_\st(V)\rightarrow\D_\dR(V)\oplus\D_\st(V)\oplus\D_\st(V)\rightarrow\D_\st(V),\]so the cohomology vanishes in degrees $>2$, and the cohomology in degree $2$ is the $K_0$-linear dual of the kernel of\[(0,p\varphi^*-1,-N^*)\colon\D_\st(V)^\dual\rightarrow\D_\dR(V)^\dual\oplus\D_\st(V)^\dual\oplus\D_\st(V)^\dual,\]which is $\left(\D_\st(V)^\dual\right)^{N=0,p\varphi=1}=\D_\cris^{\varphi=1}(V^\dual(1))$ as desired.
\end{proof}
\end{theorem}

\subsubsection{Compatibilities}\label{ss:compatibilities}

In our study of Bloch--Kato quotients, we will need to know various compatibilities between the descriptions in theorem \ref{thm:bloch-kato_sets_as_cohomotopy}. The implicit natural isomorphisms in that theorem already tell us that the descriptions are compatible with homomorphisms of representations, but we will also want to know that the descriptions of $\pi^i(U(\B^\bullet_e)^{G_K})$, $\pi^i(U(\B^\bullet_f)^{G_K})$ and $\pi^i(U(\B^\bullet_g)^{G_K})$ are compatible with one another, and that they are compatible with Serre twists. Since this section consists entirely of verifying that certain obvious compatibilities are in fact true, we advise the reader to skip this section on the first reading.

To make this first compatibility precise, we observe that there are canonical $G_K$-equivariant cosimplicial $\Q_p$-algebra homomorphisms $\B^\bullet_e\hookrightarrow\B^\bullet_f\hookrightarrow\B^\bullet_g$, where the second embedding is induced by the obvious embedding
\begin{center}
\begin{tikzcd}
\B_\dR^+\times\B_\cris \arrow[r,hook]\arrow[d,shift left]\arrow[d,shift right] & \B_\dR^+\times\B_\st[\boldsymbol\varepsilon]^\bullet \arrow[d,shift left]\arrow[d,shift right] \\
\B_\dR\times\B_\cris \arrow[r,hook] & \B_\dR\times\B_\st[\boldsymbol\varepsilon]^\bullet
\end{tikzcd}
\end{center}
of semi-cosimplicial $\Q_p$-cdgas using construction \ref{cons:cosimplicial_algebras_from_semi-cosimplicial_cdgas}. There is a minor subtlety in the definition of the morphism $\B^\bullet_e\hookrightarrow\B^\bullet_f$, in that it is not induced by a morphism of the defining semi-cosimplicial $\Q_p$-algebras. But by the universal property of cosimplicial algebras cogenerated by semi-cosimplicial algebras, to construct this morphism it suffices to define a morphism of semi-cosimplicial $\Q_p$-algebras from $\B^\bullet_e$ to the semi-cosimplicial $\Q_p$-algebra defining $\B^\bullet_f$. Such a morphism is provided in degrees $0$ and $1$ by the maps
\begin{align*}
\B^0_e=\B_\dR^+\times\B_\cris^{\varphi=1}&\longrightarrow\B_\dR^+\times\B_\cris \\
(x,y)&\longmapsto(x,y) \\
\B^1_e=\B_\dR^+\times\B_\cris^{\varphi=1}\times\B_\dR&\longrightarrow\B_\dR\times\B_\cris \\
(x,y,z)&\longmapsto(z,y)
\end{align*}
and in higher degrees by the zero morphism to the zero ring. It is then a straightforward chasing of definitions to check that this indeed defines a morphism of semi-cosimplicial $\Q_p$-algebras.

Equipped with these embeddings $\B^\bullet_e\hookrightarrow\B^\bullet_f\hookrightarrow\B^\bullet_g$, we can now state the compatibility result for theorem \ref{thm:bloch-kato_sets_as_cohomotopy}.

\begin{theorem}\label{thm:compatibility_for_bloch-kato_cohomotopy}
With respect to the descriptions in theorem \ref{thm:bloch-kato_sets_as_cohomotopy}, the maps on cohomotopy induced by the maps $U(\B^\bullet_e)^{G_K}\hookrightarrow U(\B^\bullet_f)^{G_K}\hookrightarrow U(\B^\bullet_g)^{G_K}$ coming from evaluating $U$ on $\B^\bullet_e\hookrightarrow\B^\bullet_f\hookrightarrow\B^\bullet_g$ are identified as:
\begin{itemize}
	\item the identity maps $U(\Q_p)^{G_K}=U(\Q_p)^{G_K}=U(\Q_p)^{G_K}$ on $\pi^0$;
	\item the natural inclusions $\H^1_e(G_K,U(\Q_p))\hookrightarrow\H^1_f(G_K,U(\Q_p))\hookrightarrow\H^1_g(G_K,U(\Q_p))$ on $\pi^1$;
	\item the zero maps $1=1\hookrightarrow\D_\cris^{\varphi=1}(U(\Q_p)^\dual(1))^\dual$ on $\pi^2$ when $U$ is abelian;
	\item the zero maps $1=1=1$ on the higher $\pi^i$ when $U$ is abelian.
\end{itemize}
\begin{proof}
The extra compatibility relations when $U$ is abelian are automatic. For $\pi^0$ and $\pi^1$, we just need to recall that the exact sequences assigned in theorem \ref{thm:long_exact_seqs} are functorial, and hence the exact sequences in corollary \ref{cor:pi^1_is_bloch-kato} fit together into a commuting diagram
\begin{center}
\begin{tikzcd}[column sep=small]
1 \arrow[r] & U(\Q_p)^{G_K} \arrow[r]\arrow[d,equals] & U(\B^0_e)^{G_K} \arrow[r,"\acts"]\arrow[d] & \Cycle^1(U(\B^\bullet_e))^{G_K} \arrow[r]\arrow[d] & \H^1(G_K,U(\Q_p)) \arrow[d,equals] \\
1 \arrow[r] & U(\Q_p)^{G_K} \arrow[r]\arrow[d,equals] & U(\B^0_f)^{G_K} \arrow[r,"\acts"]\arrow[d] & \Cycle^1(U(\B^\bullet_f))^{G_K} \arrow[r]\arrow[d] & \H^1(G_K,U(\Q_p)) \arrow[d,equals] \\
1 \arrow[r] & U(\Q_p)^{G_K} \arrow[r] & U(\B^0_g)^{G_K} \arrow[r,"\acts"] & \Cycle^1(U(\B^\bullet_g))^{G_K} \arrow[r] & \H^1(G_K,U(\Q_p)).
\end{tikzcd}
\end{center}
The induced maps on $\pi^1$ (resp.\ $\pi^0$) are exactly the induced maps on the cokernel (resp.\ kernel) of the middle horizontal arrow, and thus are $\H^1_e(G_K,U(\Q_p))\hookrightarrow\H^1_f(G_K,U(\Q_p))\hookrightarrow\H^1_g(G_K,U(\Q_p))$ (resp.\ the identity) as claimed.
\end{proof}
\end{theorem}

The second compatibility we will be interested in is compatibility with Serre twists, which is just a straightforward chasing of definitions.

\begin{lemma}\label{lem:twist_compatibility}
Let $U/\Q_p$ be a de Rham representation of $G_K$ on a finitely generated pro-unipotent group. Fix an element of $\H^1_g(G_K,U(\Q_p))\cong\pi^1\left(U(\B^\bullet_g)^{G_K}\right)$, and pick representatives as a Galois cocycle $\alpha\in\Cycle^1(G_K,U(\Q_p))$ and as a cocycle $\beta\in\Cycle^1(U(\B^\bullet_g)^{G_K})$. Then there is an isomorphism\[{}_\beta\left(U(\B^\bullet_g)^{G_K}\right)\isoarrow\left({}_\alpha U(\B^\bullet_g)^{G_K}\right)\]of cosimplicial groups such that the induced map on $\pi^1$ makes the diagram
\begin{center}
\begin{tikzcd}[column sep=small]
\pi^1\left({}_\beta\left(U(\B^\bullet_g)^{G_K}\right)\right) \arrow[r,"\sim"]\arrow[d,"\wr"] & \pi^1\left({}_\alpha U(\B^\bullet_g)^{G_K}\right) \arrow[r,"\sim"] & \H^1_g(G_K,{}_\alpha U(\Q_p)) \arrow[d,"\wr"] \\
\pi^1\left(U(\B^\bullet_g)^{G_K}\right) \arrow[rr,"\sim"] & & \H^1_g(G_K,U(\Q_p))
\end{tikzcd}
\end{center}
commute, where the rightmost horizontal isomorphisms are those from theorem \ref{thm:bloch-kato_sets_as_cohomotopy} (note that ${}_\alpha U$ is de Rham by proposition \ref{prop:admissibility_of_twists}) and the vertical maps are the canonical identifications.

The same is true for $\H^1_f$ (resp.\ $\H^1_e$) when $\B^\bullet_g$ is replaced with $\B^\bullet_f$ (resp.\ $\B^\bullet_e$).
\begin{proof}
We prove only the case of $\H^1_g$, the other cases being similar. Note that it suffices to prove the result for any single choice of representative $\beta$ of our fixed class, since then if $\beta'$ were another choice, we would have a cosimplicial group isomorphism ${}_{\beta'}\left(U(\B^\bullet_g)^{G_K}\right)\isoarrow{}_\beta\left(U(\B^\bullet_g)^{G_K}\right)$ by proposition \ref{prop:trivial_twists}, so an isomorphism ${}_\beta\left(U(\B^\bullet_g)^{G_K}\right)\isoarrow\left({}_\alpha U(\B^\bullet_g)^{G_K}\right)$ would allow us to produce a similar isomorphism for $\beta'$; commutativity of the desired rectangle for $\beta'$ follows from that for $\beta$ and proposition \ref{prop:trivial_twists}.

Thus fix any choice of $\alpha$. Since by assumption the class of $\alpha$ lies in $\H^1_g(G_K,U(\Q_p))=\ker\left(\H^1(G_K,U(\Q_p))\rightarrow\H^1(G_K,U(\B^0_g))\right)$ (see the proof of corollary \ref{cor:pi^1_is_bloch-kato}), we may choose some $u\in U(\B^0_g)$ whose coboundary is the Galois cocycle $\alpha$. We then consider the cocycle $\beta\in\Cycle^1(U(\B^\bullet_g))$ which is the coboundary of $u^{-1}$, i.e.\ $\beta:=d^1(u)d^0(u)^{-1}$. Since the coface and coboundary maps in $U(\B^\bullet_g)$ are $G_K$-equivariant, we thus have by straightforward calculation that\[\beta^{-1}\sigma(\beta)=d^0(u)d^1(\alpha(\sigma))d^0(\alpha(\sigma))^{-1}d^0(u)^{-1}=1\]for all $\sigma\in G_K$, where the final equality comes from the fact that $\alpha$ takes values in $U(\Q_p)$, the equaliser of $d^0,d^1\colon U(\B^0_g)\rightrightarrows U(\B^1_g)$. Hence $\beta$ is $G_K$-fixed, so is an element of $\Cycle^1(U(\B^\bullet_g))^{G_K}=\Cycle^1(U(\B^\bullet_g)^{G_K})$. Moreover, the description of the coboundary map in theorem \ref{thm:long_exact_seqs} shows that the class of $\beta$ maps to that of $\alpha$ under our isomorphism $\pi^1\left(U(\B^\bullet_g)\right)\isoarrow\H^1_g(G_K,U(\Q_p))$.

Thus we need only prove the result for this $\beta$ and this $\alpha$, using $u$ to construct the desired isomorphism. Indeed, it suffices to construct from $u$ a $G_K$-equivariant isomorphism\[{}_\beta\left(U(\B^\bullet_g)\right)\isoarrow{}_\alpha U(\B^\bullet_g)\]of cosimplicial groups, since the $G_K$-invariants of the left-hand cosimplicial group is ${}_\beta\left(U(\B^\bullet_g)^{G_K}\right)$ as $\beta$ is $G_K$-fixed.

To do this, since $\beta$ is the coboundary of $u^{-1}$, proposition \ref{prop:trivial_twists} provides an isomorphism of cosimplicial groups\[\psi_u\colon{}_\beta\left(U(\B^\bullet_g)\right)\isoarrow U(\B^\bullet_g)\]which need not be $G_K$-equivariant (as $u^{-1}$ need not be $G_K$-fixed). However, ${}_\alpha U(\B^\bullet_g)$ has the same underlying cosimplicial group as $U(\B^\bullet_g)$, and we claim that $\psi_u$ is $G_K$-equivariant when its codomain is given the $\alpha$-twisted $G_K$-action, so that by taking $G_K$-fixed points we obtain the desired result.

Verifying this is a simple chasing of definitions. If $g\in G_K$ and $u_n\in{}_\beta\left(U(\B^n_g)\right)=U(\B^n_g)$, then we have
\begin{align*}
\psi_u(gu_n) &= d^n\cdots d^1(u^{-1})\cdot gu_n\cdot d^n\cdots d^1(u^{-1})^{-1} \\
 &= d^n\cdots d^1(u^{-1}g(u))\cdot g(\psi_u(u_n))\cdot d^n\cdots d^1(u^{-1}g(u))^{-1} \\
 &= d^n\cdots d^0(\alpha(g))\cdot g(\psi_u(u_n))\cdot d^n\cdots d^0(\alpha(g))^{-1}
\end{align*}
which is the $\alpha$-twisted action on ${}_\alpha U$, since the composite $d^n\cdots d^0\colon\pi^0(U(\B^\bullet_g))\rightarrow U(\B^n_g)$ is just the canonical inclusion $U(\Q_p)\hookrightarrow U(\B^n_g)$. This establishes the $G_K$-equivariance of $\psi_u$, and hence concludes the proof.
\end{proof}
\end{lemma}

\subsection{Bloch--Kato quotients}

We can also use similar cosimplicial techniques to analyse the Bloch--Kato quotients $\H^1_{g/e}$ and $\H^1_{f/e}$ in the non-abelian setting (the author is unfortunately currently unaware of how to make the right definitions to conduct a similar analysis of $\H^1_{g/f}$, but understanding the two other quotients will suffice for our purposes). Our preferred method for carrying out this analysis is to define cosimplicial algebras $\B^\bullet_{g/e}$ and $\B^\bullet_{f/e}$, so that the corresponding Bloch--Kato quotients can be described in terms of the first cohomotopy of $U(\B^\bullet_*)^{G_K}$, just as theorem \ref{thm:bloch-kato_quotients_as_cohomotopy} gave the corresponding description for local Bloch--Kato Selmer sets.

\begin{definition}
\label{def:bloch-kato_quotient_algebras}
We define two cosimplicial $\Q_p$-algebras $\B^\bullet_{f/e}$ and $\B^\bullet_{g/e}$ as follows:
\begin{itemize}
	\item $\B^\bullet_{f/e}$ is the cosimplicial $\Q_p$-algebra cogenerated by the semi-cosimplicial topological $\Q_p$-algebra\[\B_\cris\rightrightarrows\B_\cris\]where $d^0=\varphi$ and $d^1=1$;
	\item $\B^\bullet_{g/e}$ is  the cosimplicial $\Q_p$-algebra constructed from the semi-cosimplicial $\Q_p$-cdga\[\B_\st[\boldsymbol\varepsilon]^\bullet\rightrightarrows\B_\st[\boldsymbol\varepsilon]^\bullet\]as in construction \ref{cons:cosimplicial_algebras_from_semi-cosimplicial_cdgas}. Here $\B_\st[\boldsymbol\varepsilon]^\bullet$ is the $\Q_p$-cdga in example \ref{ex:B_st[epsilon]}, and $d^0=\varphi$ and $d^1=1$.
\end{itemize}
Both of these cosimplicial $\Q_p$-algebras inherit an action of $G_K$ from those on $\B_\st$ and $\B_\cris$. They also carry a topology, though the results we have proven up to this point mean that we will not have to concern ourselves with their topologies.
\end{definition}

\begin{theorem}\label{thm:bloch-kato_quotients_as_cohomotopy}
Let $U/\Q_p$ be a de Rham representation of $G_K$ on a finitely generated pro-unipotent group. Then the cohomotopy groups/sets of the cosimplicial topological groups $U(\B^\bullet_{f/e})^{G_K}$ and $U(\B^\bullet_{g/e})^{G_K}$ are canonically identified as
\begin{align*}
\pi^i(U(\B^\bullet_{f/e})^{G_K}) &= \begin{cases}\D_\cris^{\varphi=1}(U)(\Q_p) & \text{if $i=0$,} \\ \H^1_{f/e}(G_K,U(\Q_p)) & \text{if $i=1$,} \\ 1 & \text{if $i>1$ and $U$ abelian.} \end{cases} \\
\pi^i(U(\B^\bullet_{g/e})^{G_K}) &= \begin{cases}\D_\cris^{\varphi=1}(U)(\Q_p) & \text{if $i=0$,} \\ \H^1_{g/e}(G_K,U(\Q_p)) & \text{if $i=1$,} \\ \D_\cris^{\varphi=1}(U(\Q_p)^\dual(1))^\dual & \text{if $i=2$ and $U$ abelian,} \\ 1 & \text{if $i>2$ and $U$ abelian.} \end{cases}
\end{align*}
Moreover, with respect to the descriptions in theorem \ref{thm:bloch-kato_sets_as_cohomotopy}, the natural maps on cohomotopy induced from $U(\B^\bullet_f)^{G_K}\rightarrow U(\B^\bullet_{f/e})^{G_K}$ and $U(\B^\bullet_g)^{G_K}\rightarrow U(\B^\bullet_{g/e})^{G_K}$ are the canonical inclusion $U(\Q_p)^{G_K}\hookrightarrow\D_\cris^{\varphi=1}(U)(\Q_p)$ in degree $0$, the defining quotient map $\H^1_f(G_K,U(\Q_p))\twoheadrightarrow\H^1_{f/e}(G_K,U(\Q_p))$ (and similarly for $g/e$) in degree $1$, and the identity in higher degrees.
\end{theorem}

Before we embark on a proof of this theorem, let us collect a few straightforward consequences. Firstly, we obtain a natural criterion for the equality of $\H^1_e$ and $\H^1_f$, thereby making good on our promise in remark \ref{rmk:what_about_f/e}.

\begin{corollary}\label{cor:H^1_e=H^1_f}
Let $U/\Q_p$ be a de Rham representation of $G_K$ on a unipotent group. Then $\H^1_e(G_K,U(\Q_p))=\H^1_f(G_K,U(\Q_p))$ if and only if $\D_\cris^{\varphi=1}(U)(\Q_p)=1$.
\begin{proof}
Note that $\D_\cris^{\varphi=1}(U)(\Q_p)$ and $\H^1_{f/e}(G_K,U(\Q_p))$ are, respectively, the $0$th and $1$st cohomotopy of the cosimplicial group $U(\B_{f/e}^\bullet)^{G_K}$, which is cogenerated by the semi-cosimplicial group\[\D_\cris(U)(K_0)\rightrightarrows\D_\cris(U)(K_0)\]with coface maps $d^0=\varphi$ and $d^1=1$, as defined in \S\ref{ss:extra_structure}. Hence by remark \ref{rmk:non-abelian_dold-kan} we see that $\H^1_{f/e}(G_K,U(\Q_p))=1$ iff the twisted conjugation action $u\colon w\mapsto u^{-1}w\varphi(u)$ of $\D_\cris(U)(K_0)$ on itself is transitive, and that $\H^1_e(G_K,U(\Q_p))=\H^1_f(G_K,U(\Q_p))$ iff it has trivial point-stabiliser.

Thus it would suffice to show that for any unipotent group $D/K_0$ with a semilinear Frobenius automorphism $\varphi$, the twisted conjugation self-action of $D$ is transitive iff it has trivial point-stabiliser. This is established by the same argument as in proposition \ref{prop:H^1_nr=1}.
\end{proof}
\end{corollary}

\begin{remark}\label{rmk:D_cris^phi=1_is_trivial}
Just as in remark \ref{rmk:H^1_nr=1}, we have that $\D_\cris^{\varphi=1}(U)(\Q_p)=1$ whenever $U$ is the $\Q_p$-pro-unipotent fundamental group of a smooth connected variety. This is again a consequence of a suitable weight--monodromy property for the fundamental group \cite[Theorem~1.3(1)]{me-daniel:weight-monodromy}. Again, in the case that $U$ is the fundamental group of~$L^\times=L\setminus0$ for $L$ a line bundle on an abelian variety~$A/K$, this can be argued more directly. By exact sequence \ref{eq:pi_1_of_L} $\D_\cris^{\varphi=1}(U)(\Q_p)$ sits in an exact sequence\[1\rightarrow\D_\cris^{\varphi=1}(\Q_p(1))\rightarrow\D_\cris^{\varphi=1}(U)(\Q_p)\rightarrow\D_\cris^{\varphi=1}(V_pA)\]in which the outer terms vanish by \cite[Examples 3.9 \& 3.11]{bloch-kato} and corollary \ref{cor:H^1_e=H^1_f}.
\end{remark}

The second consequence of theorem \ref{thm:bloch-kato_quotients_as_cohomotopy} is that we can make good on our promise in remark \ref{rmk:inverse_limits_of_H^1_g} and prove that $\H^1_{g/e}$ preserves inverse limits.

\begin{lemma}\label{lem:inverse_limits_of_H^1_g}
Let $U/\Q_p$ be a representation of $G_K$ on a finitely generated pro-unipotent group, and write $U=\liminv U_n$ as an inverse limit of finite-dimensional $G_K$-equivariant quotients. Then the natural map\[\H^1_g(G_K,\liminv U_n(\Q_p))\rightarrow\liminv\H^1_g(G_K,U_n(\Q_p))\]is bijective. The same holds with $\H^1_g$ replaced with $\H^1_f$, $\H^1_e$, $\H^1_{g/e}$ or $\H^1_{f/e}$.
\begin{proof}
Let us prove the lemma for $\H^1_{g/e}$, the other cases being similar. By theorem \ref{thm:bloch-kato_quotients_as_cohomotopy}, it suffices to prove that the natural map\[\pi^1\left(U(\B^\bullet_{g/e})^{G_K}\right)\rightarrow\liminv\pi^1\left(U_n(\B^\bullet_{g/e})^{G_K}\right)\]is bijective.

To show surjectivity, we pick an element of $\liminv\pi^1\left(U_n(\B^\bullet_{g/e})^{G_K}\right)$ and, for each $n$, write $S_n$ for the set of elements of $\Cycle^1\left(U_n(\B^\bullet_{g/e})^{G_K}\right)$ whose class is the chosen element of $\pi^1\left(U_n(\B^\bullet_{g/e})^{G_K}\right)$. Each $S_n$ comes with a transitive action of $U_n(\B^0_{g/e})^{G_K}=\D_\st(U_n)(K_0)$. Since the system $\left(\D_\st(U_n)(K_0)\right)_{n\in\N}$ satisfies the Mittag--Leffler condition, so too does $(S_n)_{n\in\N}$, so that $\liminv S_n\neq\emptyset$. But any element of $\liminv S_n\subseteq\Cycle^1\left(U(\B^\bullet_{g/e})^{G_K}\right)$ represents an element of $\pi^1\left(U(\B^\bullet_{g/e})^{G_K}\right)$ mapping to the chosen element of $\liminv\pi^1\left(U_n(\B^\bullet_{g/e})^{G_K}\right)$, proving surjectivity.

To show that the map has trivial kernel, pick an element $\alpha\in\Cycle^1\left(U(\B^\bullet_{g/e})^{G_K}\right)$ whose class lies in the kernel, i.e.\ such that the image of $\alpha$ in each $\Cycle^1\left(U_n(\B^\bullet_{g/e})^{G_K}\right)$ is a coboundary. Writing $T_n\subseteq U_n(\B^0_{g/e})^{G_K}$ for the set of elements whose coboundary is the image of $\alpha$, we observe that each $T_n$ is a left torsor under $\pi^0\left(U_n(\B^\bullet_{g/e})^{G_K}\right)=\D_\cris^{\varphi=1}(U_n)(\Q_p)$, so that as before $(T_n)_{n\in\N}$ satisfies the Mittag--Leffler condition and $\liminv T_n\neq\emptyset$. Again, an element of $\liminv T_n\subseteq U(\B^0_{g/e})^{G_K}$ is an element whose coboundary is $\alpha$, so we have proved triviality of the kernel.

Finally, to show injectivity, as in the proof of lemma \ref{lem:cohomology_of_inverse_limits}, it suffices to show that the map also has trivial kernel when the cosimplicial group $U(\B^\bullet_{g/e})^{G_K}$ is replaced by the twist by a cocycle $\alpha\in\Cycle^1\left(U(\B^\bullet_{g/e})^{G_K}\right)$ (and each $U_n(\B^\bullet_{g/e})^{G_K}$ is compatibly twisted). For this, the same argument as above applies once we observe that each $U_n(\B^\bullet_{g/e})^{G_K}$ and hence ${}_\alpha\left(U_n(\B^\bullet_{g/e})^{G_K}\right)$ is the $\Q_p$-points of a cosimplicial unipotent group $\D_{g/e}^\bullet(U_n)/\Q_p$, so that the subgroups $\pi^0\left({}_\alpha\left(U_n(\B^\bullet_{g/e})^{G_K}\right)\right)$ are the $\Q_p$-points of (unipotent) algebraic subgroups of ${}_\alpha\D_{g/e}^\bullet(U_n)$, and hence satisfy the Mittag--Leffler condition.
\end{proof}
\end{lemma}

The final consequence of theorem \ref{thm:bloch-kato_quotients_as_cohomotopy} is two exact sequences relating the Bloch--Kato quotients in central extensions, which will allow us to analyse these sets in the same way we analysed non-abelian Galois cohomology in \S\ref{c:main_theorem_l-adic}.

\begin{corollary}\label{cor:long_exact_seqs_of_quotients}
Let\[1\rightarrow Z\centarrow U\rightarrow Q\rightarrow1\]be a central extension of de Rham representations of $G_K$ on unipotent groups over $\Q_p$. Then there are functorially assigned exact sequences
\begin{center}
\begin{tikzcd}[column sep=small]
1 \arrow[r] & Z(\Q_p)^{G_K} \arrow[r,"\cent"] & U(\Q_p)^{G_K} \arrow[r] \arrow[d, phantom, ""{coordinate, name=Z}] & Q(\Q_p)^{G_K} \arrow[dll,rounded corners,to path={ -- ([xshift=2ex]\tikztostart.east) |- (Z) [near end]\tikztonodes -| ([xshift=-2ex]\tikztotarget.west) -- (\tikztotarget)}] \\
 & \H^1_g(G_K,Z(\Q_p)) \arrow[r,"\acts"] & \H^1_g(G_K,U(\Q_p)) \arrow[r] \arrow[d, phantom, ""{coordinate, name=X}] & \H^1_g(G_K,Q(\Q_p)) \arrow[dll,rounded corners,to path={ -- ([xshift=2ex]\tikztostart.east) |- (X) [near end]\tikztonodes -| ([xshift=-2ex]\tikztotarget.west) -- (\tikztotarget)}] \\
 & \D_\cris^{\varphi=1}(Z(\Q_p)^\dual(1))^\dual & {} & 
\end{tikzcd}
\end{center}
and
\begin{center}
\begin{tikzcd}[column sep=small]
1 \arrow[r] & \D_\cris^{\varphi=1}(Z)(\Q_p) \arrow[r,"\cent"] & \D_\cris^{\varphi=1}(U)(\Q_p) \arrow[r] \arrow[d, phantom, ""{coordinate, name=Z}] & \D_\cris^{\varphi=1}(Q)(\Q_p) \arrow[dll,rounded corners,to path={ -- ([xshift=2ex]\tikztostart.east) |- (Z) [near end]\tikztonodes -| ([xshift=-2ex]\tikztotarget.west) -- (\tikztotarget)}] \\
 & \H^1_{g/e}(G_K,Z(\Q_p)) \arrow[r,"\acts"] & \H^1_{g/e}(G_K,U(\Q_p)) \arrow[r] \arrow[d, phantom, ""{coordinate, name=X}] & \H^1_{g/e}(G_K,Q(\Q_p)) \arrow[dll,rounded corners,to path={ -- ([xshift=2ex]\tikztostart.east) |- (X) [near end]\tikztonodes -| ([xshift=-2ex]\tikztotarget.west) -- (\tikztotarget)}] \\
 & \D_\cris^{\varphi=1}(Z(\Q_p)^\dual(1))^\dual. & {} & 
\end{tikzcd}
\end{center}
Moreover, when $U$ and $Q$ are abelian, both exact sequences continue on with the terms $\D_\cris^{\varphi=1}(U(\Q_p)^\dual(1))^\dual\rightarrow\D_\cris^{\varphi=1}(Q(\Q_p)^\dual(1))^\dual\rightarrow1$.
\begin{proof}
Take the long exact sequence in cohomotopy of the two exact sequences (which are central extensions) in proposition \ref{prop:long_exact_seqs_of_quotients_prep} which follows.
\end{proof}
\end{corollary}

\begin{proposition}\label{prop:long_exact_seqs_of_quotients_prep}
Let\[1\rightarrow Z\rightarrow U\rightarrow Q\rightarrow 1\]be an exact sequence of de Rham representations of $G_K$ on finitely generated pro-unipotent groups. Then the sequences\[1\rightarrow Z(\B^\bullet_g)^{G_K}\rightarrow U(\B^\bullet_g)^{G_K}\rightarrow Q(\B^\bullet_g)^{G_K}\rightarrow1\]\[1\rightarrow Z(\B^\bullet_{g/e})^{G_K}\rightarrow U(\B^\bullet_{g/e})^{G_K}\rightarrow Q(\B^\bullet_{g/e})^{G_K}\rightarrow1\]are exact in each degree.
\begin{proof}
We will just prove this for $\B^\bullet_g$, the other case being similar. Observe that, by construction, each term of $U(\B^\bullet_g)$ is a product of terms of the form $U(\B_\dR)$, $U(\B_\dR^+)$ and $U(\B_\st[\boldsymbol\varepsilon]^n)$, where $\B_\st[\boldsymbol\varepsilon]^n=\B_\st[\varepsilon_1,\dots,\varepsilon_n]$ is as in example \ref{ex:B_st[epsilon]}, so that it suffices to show that the sequences
\[1\rightarrow Z(\B_\dR)^{G_K}\rightarrow U(\B_\dR)^{G_K}\rightarrow Q(\B_\dR)^{G_K}\rightarrow1\]
\[1\rightarrow Z(\B_\dR^+)^{G_K}\rightarrow U(\B_\dR^+)^{G_K}\rightarrow Q(\B_\dR^+)^{G_K}\rightarrow1\]
\[1\rightarrow Z(\B_\st[\boldsymbol\varepsilon]^n)^{G_K}\rightarrow U(\B_\st[\boldsymbol\varepsilon]^n)^{G_K}\rightarrow Q(\B_\st[\boldsymbol\varepsilon]^n)^{G_K}\rightarrow1\]
are exact.

But now we can identify these sequences with the sequences
\[1\rightarrow\D_\dR(Z)(K)\rightarrow\D_\dR(U)(K)\rightarrow\D_\dR(Q)(K)\rightarrow1\]
\[1\rightarrow\D_\dR^+(Z)(K)\rightarrow\D_\dR^+(U)(K)\rightarrow\D_\dR^+(Q)(K)\rightarrow1\]
\[1\rightarrow\D_\st(Z)(K_0[\boldsymbol\varepsilon]^n)\rightarrow\D_\st(U)(K_0[\boldsymbol\varepsilon]^n)\rightarrow\D_\st(Q)(K_0[\boldsymbol\varepsilon]^n)\rightarrow1\]
where $K_0[\boldsymbol\varepsilon]^n=K_0[\varepsilon_1,\dots,\varepsilon_n]$ subject to $\varepsilon_i\varepsilon_j=0$ for all $i,j$. The exactness of these sequences is the content of lemma \ref{lem:seqs_of_de_Rham_reps}.
\end{proof}
\end{proposition}

The remainder of this section is devoted to a proof of theorem \ref{thm:bloch-kato_quotients_as_cohomotopy}. We will restrict our attention to the calculation of the cohomotopy of $U(\B^\bullet_{g/e})^{G_K}$, the case of $U(\B^\bullet_{f/e})^{G_K}$ being similar (and simpler).

In making the argument precise, it will be convenient for this section only to permit our $\Q_p$-algebras (or more accurately, their homomorphisms) to be non-unital\footnote{We follow the usual convention that ``non-unital'' means ``not necessarily unital''.} and to define, for a non-unital $\Q_p$-algebra $\B$ and unipotent group $U/\Q_p$, the group $U(\B)$ to be the kernel\[U(\B):=\ker\left(U(\Q_p\oplus\B)\rightarrow U(\Q_p)\right),\]where $\Q_p\oplus\B$ is given the obvious algebra structure with unit $(1,0)$, and $\Q_p\oplus\B\rightarrow\Q_p$ is the natural projection. This is clearly functorial in $\B$ with respect to non-unital $\Q_p$-algebra homomorphisms, and extends the usual definition on unital algebras, since for a unital $\B$ there is a canonical isomorphism $\Q_p\oplus\B\cong\Q_p\times\B$.

With this definition, we will let $\B_\dR^\bullet$ denote the (non-unital) cosimplicial $\Q_p$-algebra cogenerated by the semi-cosimplicial $\Q_p$-algebra\[\B_\dR^+\rightrightarrows\B_\dR\]where $d^0=0$ and $d^1$ is the canonical inclusion. There are then $G_K$-equivariant sequences\[0\rightarrow\B^\bullet_\dR\rightarrow\B^\bullet_e\rightarrow\B_\cris^{\varphi=1}\rightarrow0\]\[0\rightarrow\B^\bullet_\dR\rightarrow\B^\bullet_g\rightarrow\B^\bullet_{g/e}\rightarrow0\]of cosimplicial $\Q_p$-algebras, where $\B_\cris^{\varphi=1}$ denotes the constant cosimplicial $\Q_p$-algebra which is $\B_\cris^{\varphi=1}$ in each degree. We claim these sequences are levelwise exact, and split by a $G_K$-equivariant $\Q_p$-algebra homomorphism in each degree.

To see why this is, at least for the second sequence (the first being similar), we consider the diagram
\begin{center}
\begin{tikzcd}
0 \arrow[r] &\B_\dR^+ \arrow[r,hook]\arrow[d,shift left]\arrow[d,shift right] & \B_\st[\boldsymbol\varepsilon]^\bullet\times\B_\dR^+ \arrow[r,two heads]\arrow[d,shift left]\arrow[d,shift right] & \B_\st[\boldsymbol\varepsilon]^\bullet \arrow[d,shift left]\arrow[d,shift right] \arrow[r] & 0 \\
0 \arrow[r] & \B_\dR \arrow[r,hook] & \B_\st[\boldsymbol\varepsilon]^\bullet\times\B_\dR \arrow[r,two heads] & \B_\st[\boldsymbol\varepsilon]^\bullet \arrow[r] & 0
\end{tikzcd}
\end{center}
of semi-cosimplicial $\Q_p$-cdgas giving rise to $0\rightarrow\B^\bullet_\dR\rightarrow\B^\bullet_g\rightarrow\B^\bullet_{g/e}\rightarrow0$ through construction \ref{cons:cosimplicial_algebras_from_semi-cosimplicial_cdgas}. It is clear that each row is exact and $G_K$-equivariantly split by a homomorphism of $\Q_p$-cdgas, so that the corresponding sequence of bi-cosimplicial $\Q_p$-algebras is $G_K$-equivariantly split in each bidegree by a homomorphism of $\Q_p$-algebras. Hence the same is true in each degree for the diagonal sequence $0\rightarrow\B^\bullet_\dR\rightarrow\B^\bullet_g\rightarrow\B^\bullet_{g/e}\rightarrow0$ as claimed.

Using this it follows that the sequences\[1\rightarrow U(\B^\bullet_\dR)\rightarrow U(\B^\bullet_e)\rightarrow U(\B_\cris^{\varphi=1})\rightarrow1\]\[1\rightarrow U(\B^\bullet_\dR)\rightarrow U(\B^\bullet_g)\rightarrow U(\B^\bullet_{g/e})\rightarrow1\]are exact and $G_K$-equivariantly split in each degree, and hence that they remain exact after taking $G_K$-invariants. By taking the long exact sequence in cohomotopy of the $G_K$-invariants of the first sequence, we obtain a description of the cohomotopy groups of $U(\B_\dR^\bullet)^{G_K}$ as:
\begin{itemize}
	\item $\pi^0(U(\B_\dR^\bullet)^{G_K})=1$;
	\item $\pi^1(U(\B_\dR^\bullet)^{G_K})$ sits in an exact sequence\setcounter{equation}{0}\begin{equation}\label{eq:long_exact_seq_for_dR}1\!\rightarrow\! U(\Q_p)^{G_K}\!\rightarrow\!\D_\cris^{\varphi=1}(U)(\Q_p)\!\actsarrow\!\pi^1(U(\B_\dR^\bullet)^{G_K})\!\rightarrow\!\H^1_e(G_K,U(\Q_p))\!\rightarrow\!1;\end{equation}
	\item $\pi^i(U(\B_\dR^\bullet)^{G_K})=1$ for $i>1$ when $U$ is abelian.
\end{itemize}
Here we are using the description of the cohomotopy of $U(\B^\bullet_e)^{G_K}$ in theorem \ref{thm:bloch-kato_sets_as_cohomotopy} and the fact that the higher cohomotopy of $U(\B_\cris^{\varphi=1})^{G_K}$ is trivial (as it is a constant cosimplicial group).

Taking the long exact sequence in cohomotopy of the $G_K$-invariants of the second sequence, we find that we have an exact sequence (with the final $1$ coming from lemma \ref{lem:codimensional_vanishing})

\begin{equation}\label{eq:long_exact_seq_for_g/e}
\begin{tikzcd}
 1 \arrow[r] & U(\Q_p)^{G_K} \arrow[r]
	\arrow[d,draw=none,""{coordinate, name=Z}, anchor=center]{}
 & \pi^0(U(\B^\bullet_{g/e})^{G_K}) \arrow[rounded corners,
	to path={ -- ([xshift=2ex]\tikztostart.east)
		|- (Z) \tikztonodes
		-| ([xshift=-2ex]\tikztotarget.west)
		-- (\tikztotarget)},dll,pos=1.03,swap,"\acts"] & \\
\pi^1(U(\B^\bullet_\dR)^{G_K}) \rar & \H^1_g(G_K,U(\Q_p)) \arrow[r] & \pi^1(U(\B^\bullet_{g/e})^{G_K}) \arrow[r] & 1
\end{tikzcd}
\end{equation}
and that $\pi^2(U(\B^\bullet_{g/e})^{G_K})=\D_\cris^{\varphi=1}(U(\Q_p)^\dual(1))^\dual$ and $\pi^i(U(\B^\bullet_{g/e})^{G_K})=1$ for $i>2$ when $U$ is abelian (this proves all the extra claims in the abelian case).

Now we observe that the natural map $\B^\bullet_e\hookrightarrow\B^\bullet_g$ defined in \S\ref{ss:compatibilities} fits into a diagram
\begin{center}
\begin{tikzcd}
0 \arrow[r] & \B^\bullet_\dR \arrow[r]\arrow[d, equals] & \B^\bullet_e \arrow[r]\arrow[d, hook] & \B_\cris^{\varphi=1} \arrow[r]\arrow[d, hook] & 0 \\
0 \arrow[r] & \B^\bullet_\dR \arrow[r] & \B^\bullet_g \arrow[r] & \B^\bullet_{g/e} \arrow[r] & 0,
\end{tikzcd}
\end{center}
from which we obtain a morphism of exact sequences
\begin{center}
\begin{tikzcd}[column sep=small]
1 \arrow[r] & U(\Q_p)^{G_K} \arrow[r]\arrow[d,equals] & \D_\cris^{\varphi=1}(U)(\Q_p) \arrow[r,"\acts"]\arrow[d] & \pi^1(U(\B_\dR^\bullet)^{G_K}) \arrow[r]\arrow[d,equals] & \H^1_e(G_K,U(\Q_p)) \arrow[r]\arrow[d,hook] & 1 \\
1 \arrow[r] & U(\Q_p)^{G_K} \arrow[r] & \pi^0(U(\B^\bullet_{g/e})^{G_K}) \arrow[r,"\acts"] & \pi^1(U(\B_\dR^\bullet)^{G_K}) \arrow[r] & \H^1_g(G_K,U(\Q_p)) & 
\end{tikzcd}
\end{center}
whose top row is sequence \ref{eq:long_exact_seq_for_dR} and whose bottom row is part of sequence \ref{eq:long_exact_seq_for_g/e} (using theorem \ref{thm:compatibility_for_bloch-kato_cohomotopy} for the description of the vertical arrows). A quick five-lemma-style diagram chase shows that the map $\D_\cris^{\varphi=1}(U)(\Q_p)\rightarrow\pi^0(U(\B^\bullet_{g/e})^{G_K})$ is an isomorphism, giving us the desired description of $\pi^0$.

For $\pi^1$, we note that surjectivity of the upper-right horizontal arrow in the above diagram shows that the image of $\pi^1(U(\B_\dR^\bullet)^{G_K})\rightarrow\H^1_g(G_K,U(\Q_p))$ is $\H^1_e(G_K,U(\Q_p))$, and hence we obtain from sequence \ref{eq:long_exact_seq_for_g/e} an exact sequence of pointed sets\[1\rightarrow\H^1_e(G_K,U(\Q_p))\rightarrow\H^1_g(G_K,U(\Q_p))\rightarrow\pi^1(U(\B^\bullet_{g/e})^{G_K})\rightarrow1.\]
As usual, this doesn't imply that $\H^1_g(G_K,U(\Q_p))\twoheadrightarrow\pi^1\left(U(\B^\bullet_{g/e})^{G_K}\right)$ induces an isomorphism $\H^1_{g/e}(G_K,U(\Q_p))\isoarrow\pi^1(U(\B^\bullet_{g/e})^{G_K})$, and we have to consider twists of $U$ to complete the proof.

Thus let $\alpha$ be a de Rham cocycle of $U$, i.e.\ one whose class lies in $\H^1_g(G_K,U(\Q_p))$. As per the proof of lemma \ref{lem:twist_compatibility}, we may pick some $u\in U(\B^0_g)$ of which $\alpha$ is the coboundary, and write $\beta=d^1(u)d^0(u)^{-1}\in\Cycle^1\left(U(\B^\bullet_g)^{G_K}\right)$ for the coboundary of $u^{-1}$ in the cosimplicial group $U(\B^\bullet_g)$. Writing $\bar u\in U(\B^\bullet_{g/e})$ and $\bar\beta\in\Cycle\left(U(\B^\bullet_{g/e})^{G_K}\right)$ for the images of $u$ and $\beta$ respectively, the same argument as in the proof of lemma \ref{lem:twist_compatibility}	shows that $\bar u$ induces a $G_K$-equivariant isomorphism\[{}_{\bar\beta}\left(U(\B^\bullet_{g/e})\right)\isoarrow{}_\alpha U(\B^\bullet_{g/e}),\]compatible with the corresponding isomorphism for $\B^\bullet_g$. Taking $G_K$-invariants and $\pi^1$, we obtain a commuting diagram
\begin{center}
\begin{tikzcd}[column sep=small]
\H^1_g(G_K,{}_\alpha U(\Q_p)) \arrow[r,"\sim"]\arrow[dd,"\wr"] & \pi^1\left({}_\alpha U(\B^\bullet_g)^{G_K}\right)\arrow[r,two heads] & \pi^1\left({}_\alpha U(\B^\bullet_{g/e})^{G_K}\right) \\
& \pi^1\left({}_\beta\left(U(\B^\bullet_g)^{G_K}\right)\right) \arrow[r,two heads]\arrow[d,"\wr"]\arrow[u,swap,"\wr"] & \pi^1\left({}_{\bar\beta}\left(U(\B^\bullet_{g/e})^{G_K}\right)\right) \arrow[d,"\wr"]\arrow[u,swap,"\wr"] \\
\H^1_g(G_K,U(\Q_p)) \arrow[r,"\sim"] & \pi^1\left(U(\B^\bullet_g)^{G_K}\right) \arrow[r,two heads] & \pi^1\left(U(\B^\bullet_{g/e})^{G_K}\right).
\end{tikzcd}
\end{center}
The upper-right square commutes already on the level of cosimplicial groups by the above discussion, and the lower-right square commutes by naturality of the twisting isomorphism in proposition \ref{prop:pi^1_of_twists}; that the left-hand rectangle commutes is part of lemma \ref{lem:twist_compatibility}.

Since by proposition \ref{prop:admissibility_of_twists} ${}_\alpha U$ is also de Rham, the composite of the top row has kernel $\H^1_e(G_K,{}_\alpha U(\Q_p))$, so that by the commutativity of the above diagram and lemma \ref{lem:twists_give_cosets}, the fibre of $\H^1_g(G_K,U(\Q_p))\twoheadrightarrow\pi^1\left(U(\B^\bullet_{g/e})^{G_K}\right)$ containing $[\alpha]$ is exactly its $\sim_{\H^1_e}$ equivalence class. Thus the map $\H^1_g(G_K,U(\Q_p))\twoheadrightarrow\pi^1\left(U(\B^\bullet_{g/e})^{G_K}\right)$ factors through an isomorphism $\H^1_{g/e}(G_K,U(\Q_p))\isoarrow\pi^1(U(\B^\bullet_{g/e})^{G_K})$ as desired.\hfill\qed
\section{The main theorem ($p$-adic version)}
\label{c:main_theorem_p-adic}

With the assembled theory of the preceding sections, we are now able to exactly translate the proof of theorem \ref{thm:main_theorem_l-adic} to give a proof of its $p$-adic analogue, theorem \ref{thm:main_theorem_p-adic}. Let us fix an abelian variety $A/K$, and let us denote by $(L,\tilde0)$ a pair of a varying line bundle $L/A$ and a basepoint $\tilde0\in L^\times(K)$ in the complement $L^\times=L\setminus0$ of the zero section lying over $0\in A(K)$. We will write $U/\Q_p$ for the $\Q_p$-pro-unipotent fundamental group of $L^\times$ based at $\tilde0$.

As in section \ref{c:main_theorem_l-adic}, $U$ is a $G_K$-equivariant central extension\[1\rightarrow\Q_p(1)\centarrow U\rightarrow V_pA\rightarrow1\]arising from the fibration sequence $(\G_m,1)\hookrightarrow(L^\times,\tilde0)\twoheadrightarrow(A,0)$.

By \cite[Theorem~1.4(1)]{me:local_constancy}, all the groups involved are de Rham, so that by corollary \ref{cor:long_exact_seqs_of_quotients} we have an exact sequence\[\D_\cris^{\varphi=1}(V_pA)\rightarrow\H^1_{g/e}(G_K,\Q_p(1))\actsarrow\H^1_{g/e}(G_K,U(\Q_p))\rightarrow\H^1_{g/e}(G_K,V_pA).\]Since $\H^1_g$ and $\H^1_e$ agree for the Tate module of an abelian variety, the final term vanishes, which implies by the usual dimension formulae for abelian local Bloch--Kato Selmer groups that the first term vanishes too, and thus the middle map $\H^1_{g/e}(G_K,\Q_p(1))\rightarrow\H^1_{g/e}(G_K,U(\Q_p))$ is bijective. We know e.g.\ by \cite[Theorem~1.2]{me:local_constancy} that the image of the non-abelian Kummer map is contained in $\H^1_g$, so we have a well-defined composite\[\lambda_L\colon L^\times(K)\rightarrow\H^1_{g/e}(G_K,U(\Q_p))\leftisoarrow\H^1_{g/e}(G_K,\Q_p(1))\isoarrow\Q_p,\]where the final isomorphism arises from the explicit description of the local Bloch--Kato Selmer groups of $\Q_p(1)$.

It remains to show that the $\lambda_L$ are the N\'eron log-metrics on our varying line bundle $L$, for which it suffices to show that the conditions of lemma \ref{lem:neron_log-metrics} are satisfied. The argument for conditions \ref{condn:additivity}--\ref{condn:normalisation} is formal and proceeds exactly as in \S\ref{c:main_theorem_l-adic}, replacing $\ell$ with $p$ and $\H^1$ with $\H^1_{g/e}$ throughout. For condition \ref{condn:local_constancy}, see \cite[Theorem~1.2]{me:local_constancy}.
\section{The main theorem (archimedean version)}
\label{c:main_theorem_archimedean}

To complete the circle of ideas sketched in the introduction, let us now prove theorem \ref{thm:main_theorem_archimedean}, the archimedean instance of our results relating local heights to Kummer maps. One can prove this through direct calculation, but we will instead imitate the proofs of our other two main theorems and show that the description in theorem \ref{thm:main_theorem_archimedean} gives a well-defined family of functions $(\lambda_L)_L$ which satisfies a certain list of conditions uniquely defining the archimedean N\'eron log-metric.

Central to this discussion is the mixed Hodge structure carried by the fundamental groupoids of smooth complex varieties, constructed by Hain \cite{hain_MHS} using Chen's non-abelian de Rham theorem\footnote{Strictly speaking, \cite{hain_MHS} only defines the mixed Hodge structure when $z=y$; \cite[Proposition 3.21]{hain-zucker} provides the construction in general.}. More precisely, if $Y/\C$ is a smooth connected variety, Hain constructs an ind-mixed Hodge structure\label{term:ind-MHS}\footnote{By an \emph{ind-mixed Hodge structure} we simply mean an ind-object in the category of mixed Hodge structures \cite[Definition 3.1]{peters-steenbrink} over an appropriate ring (for us, always $\Z$, $\Q$ or $\R$). Since morphisms of mixed Hodge structures are strict for the weight and Hodge filtrations \cite[Corollary 3.6]{peters-steenbrink}, an ind-mixed Hodge structure $V$ is equivalently a $\Z$-module (resp.\ $\Q$- or $\R$-vector space) endowed with a weight filtration on $V_\Q$ and a Hodge filtration on $V_\C$ such that $V$ is a filtered union of finitely generated submodules $V_n$ such the restriction of the weight and Hodge filtrations to $V_n$ define a mixed Hodge structure in the usual sense. Morphisms of ind-mixed Hodge structures are then just $\Z$-module (resp.\ $\Q$- or $\R$-vector space) homomorphisms compatible with the weight and Hodge filtrations -- these are automatically strict for both filtrations by the corresponding result for mixed Hodge structures \cite[Corollary 3.6]{peters-steenbrink}.} on the decompleted dual\[\Z\pi_1(Y(\C);y,z)^{\wedge\reddual}:=\limdir\Hom(\Z\pi_1(Y(\C);y,z)/J^{n+1},\Z)\]of the completed path-torsor coalgebra for each pair of points $y,z\in Y(\C)$. Here $J^\bullet$ denotes the filtration on $\Z\pi_1(Y(\C);y,z)$ induced by the filtration on $\Z\pi_1(Y(\C);y)$ (equivalently $\Z\pi_1(Y(\C);z)$) defined by powers of the augmentation ideal. This mixed Hodge structure is compatible with the algebra structure on $\Z\pi_1(Y(\C);y,z)^{\wedge\reddual}$ and with the path-composition and path-reversal maps, so that in particular when $y=z$, $\Z\pi_1(Y(\C);y)^{\wedge\reddual}$ carries an ind-mixed Hodge structure compatible with its Hopf algebra structure.

In this setup, the ind-mixed Hodge structures on the path-torsors $\pi_1(Y(\C);y,z)$ for varying $z$ are compatible in an appropriate sense with that on $\pi_1(Y(\C);y)$, so that there is a natural non-abelian Kummer map\[Y(\C)\rightarrow\H^1(\MHS_\Z,\pi_1(Y(\C);y))\]to the moduli set of such torsors, which we will later identify as the inverse limit of the higher Albanese maps of \cite{hain-zucker}. However, it turns out that when $Y=L^\times$ is the complement of zero in a line bundle on an abelian variety, this non-abelian Kummer map will be bijective \cite[Remarks and Examples 5.45]{hain-zucker}\footnote{\cite{hain} only asserts this for $L$ with non-trivial first Chern class, but in fact this is true without this restriction.}, so is too fine an invariant to be useful to us.

Instead, we will work with the real Mal\u cev completion\[U=\Spec(\R\pi_1(Y(\C);y)^{\wedge\reddual})\]of the Betti fundamental group and its associated non-abelian Kummer map\[Y(\C)\rightarrow\H^1(\MHS_\R,U(\R)),\]and it will turn out that this map will be sufficiently coarse to allow us to recover exactly the N\'eron log-metric on such $\G_m$-torsors $L^\times$.

\subsection{Pro-unipotent groups with mixed Hodge structures}
\label{s:unipotent_MHS}

Before analysing the non-abelian Kummer map for the real unipotent fundamental group in more detail, let us briefly set out a few basic properties of unipotent groups $U$ with mixed Hodge structure in the abstract. The main objective here is to show that, under a certain natural condition, the moduli set $\H^1(\MHS_\R,U(\R))$ parametrising torsors under $U$ with compatible mixed Hodge structure canonically has the structure of a real manifold, which we will construct via an explicit description of this set.

\begin{definition}\label{def:MHS_on_groups}
Let $U/\R$ (resp.\ $U/\Q$) be a finitely generated pro-unipotent group. By a \emph{mixed Hodge structure} on $U$ we shall mean an ind-mixed Hodge structure on $\O(U)$ compatible with the Hopf algebra operations. By a \emph{torsor (with mixed Hodge structure)} under such a group $U$, we shall mean a torsor $P$ under the underlying group scheme endowed with an ind-mixed Hodge structure on $\O(P)$ which is compatible with the $\O(U)$-comodule algebra structure. By a morphism of such groups or torsors, we shall mean a morphism of the underlying group schemes or torsors which respects the weight and Hodge filtrations on the affine rings.
\end{definition}

\begin{proposition}\label{prop:J_n_is_sub-MHS}
Let $U/\R$ (resp.\ $U/\Q$) be a finitely generated pro-unipotent group with mixed Hodge structure. Then each $J_\bullet\O(U)$ is a mixed Hodge substructure of $\O(U)$. If $P$ is a torsor under $U$, then each $J_\bullet\O(P)$ is a mixed Hodge substructure of $\O(P)$, and the canonical isomorphism $\gr^J_\bullet\O(P)\isoarrow\gr^J_\bullet\O(U)$ (see proposition \ref{prop:graded_isomorphism}) is an isomorphism of ind-mixed Hodge structures.
\begin{proof}
It is obvious from definition \ref{def:J-filtration} that each $J_n\O(U)$ and each $J_n\O(P)$ is stable under the action of the Tannaka group of $\MHS_\R$ (resp.\ $\MHS_\Q$), so are mixed Hodge substructures. Since the isomorphism $\gr^J_\bullet\O(P)\isoarrow\gr^J_\bullet\O(U)$ from proposition \ref{prop:graded_isomorphism} is independent of the choice of element $\gamma$, it is equivariant for the action of this Tannaka group, and hence is an isomorphism of ind-mixed Hodge structures.
\end{proof}
\end{proposition}

\subsection{Weight and Hodge filtrations on torsors}
\label{ss:MHS_filtrations}

Central to our explicit determination of the moduli set $\H^1(\MHS_\R,U(\R))$ is a description of mixed Hodge structures on torsors directly on the torsors themselves, as opposed to on their affine rings. The key definition that allows us to study mixed Hodge structures is as follows.

\begin{definition}\label{defn:reduction_of_structure_on_torsors}
Let $U/\R$ be a finitely generated pro-unipotent group with mixed Hodge structure, and $P$ a torsor under $U$ with mixed Hodge structure. We define subsets $W_0P(\R)\subseteq P(\R)=\O(P)^{\dual,\gplike}$, $W_0P(\C)\subseteq P(\C)=\O(P)_\C^{\dual,\gplike}$ and $F^0W_0P(\C)\subseteq P(\C)=\O(P)_\C^{\dual,\gplike}$ by:
\begin{align*}
W_0P(\R) &= \O(P)^{\dual,\gplike}\cap W_0; \\
W_0P(\C) &= \O(P)_\C^{\dual,\gplike}\cap W_0; \\
F^0W_0P(\C) &= \O(P)_\C^{\dual,\gplike}\cap F^0\cap W_0.
\end{align*}
We clearly have $W_0P(\R)\subseteq W_0P(\C)\supseteq F^0W_0P(\C)$.

We will see in proposition \ref{prop:filtrations_on_torsors} that in the particular case that $P=U$, the subsets $W_0U(\R)$, $W_0U(\C)$ and $F^0W_0U(\C)$ are subgroups, and that in general the subsets $W_0P(\R)$, $W_0P(\C)$ and $F^0W_0P(\C)$ are torsors under each of these respective groups.

We also make the same definitions over the base field $\Q$ by replacing $\R$ with $\Q$ throughout.
\end{definition}

Since the proof of this result (in particular, the non-emptiness of $W_0P(\R)$ and $F^0W_0P(\C)$) is relatively technical, let us for the time being describe a converse construction.

\begin{construction}\label{cons:MHS_from_reduction_of_structure}
Let $U/\R$ be a finitely generated pro-unipotent group and $P$ a torsor under $U$. Suppose that a mixed Hodge structure on $U$ is given, along with subsets $W_0P(\R)\subseteq P(\R)$, $W_0P(\C)\subseteq P(\C)$ and $F^0W_0P(\C)\subseteq P(\C)$ which are torsors under $W_0U(\R)$, $W_0U(\C)$ and $F^0W_0U(\C)$ respectively, and such that $W_0P(\R)\subseteq W_0P(\C)\supseteq F^0W_0P(\C)$. We will describe how to construct an ind-mixed Hodge structure on $\O(P)$ which is compatible with the $\O(U)$-comodule algebra operations, thereby making $P$ into a torsor under $U$ with mixed Hodge structure.

To construct the weight filtration on $\O(P)$, pick an element $q\in W_0P(\R)$, which gives rise to an isomorphism of schemes $U\isoarrow P$ by the action on $q$. We declare the weight filtration on $\O(P)$ to be the filtration corresponding to the weight filtration on $\O(U)$ under the induced isomorphism $\O(P)\isoarrow\O(U)$.

This filtration is independent of the choice of $q$, for, if $q'$ were another choice, then it would differ from $q$ by the action of some $u\in W_0U(\R)$, and hence the two isomorphisms $\O(P)\isoarrow\O(U)$ would differ by the algebra automorphism of $\O(U)$ induced by right-multiplication by $u$ on $U$. Since $u\in W_0U(\R)$, this automorphism strictly preserves the weight filtration (this is easiest to see on the dual $\O(U)^\dual$, where $u^{\pm1}\in W_0\O(U)^\dual$), and hence $q$ and $q'$ induce the same weight filtration on $\O(P)$.

Similarly, to construct the Hodge filtration on $\O(P)_\C$ we pick any element $r\in F^0W_0P(\C)$ and declare the Hodge filtration on $\O(P)_\C$ to be the filtration corresponding to the Hodge filtration on $\O(U)_\C$ under the induced isomorphism $\O(P)_\C\isoarrow\O(U)_\C$. The same proof as above shows that this is independent of the choice of $r$. It also follows that the isomorphism $\O(P)_\C\isoarrow\O(U)_\C$ induced by $r$ is (strictly) compatible with the weight filtration induced by our choice of $q$ above; this is since $q$ and $r$ both belong to the $W_0U(\C)$-torsor $W_0P(\C)$, and hence induce the same weight filtration by the same argument as above.

We will see in proposition \ref{prop:construction_gives_MHS} that these filtrations do indeed define a mixed Hodge structure on $P$. One may make the same construction with $\R$ replaced with $\Q$ throughout.
\end{construction}

\begin{proposition}\label{prop:MHS_on_torsors}
The constructions in definition \ref{defn:reduction_of_structure_on_torsors} and construction \ref{cons:MHS_from_reduction_of_structure} are mutually inverse. That is, let $U/\R$ be a finitely generated pro-unipotent group and $P$ a torsor under $U$, and suppose that a mixed Hodge structure on $U$ (but not $P$) is given. Then definition \ref{defn:reduction_of_structure_on_torsors} and construction \ref{cons:MHS_from_reduction_of_structure} provide a canonical bijection between:
\begin{itemize}
	\item ind-mixed Hodge structures on $\O(P)$ compatible with the $\O(U)$-comodule algebra operations; and
	\item subsets $W_0P(\R)\subseteq P(\R)$, $W_0P(\C)\subseteq P(\C)$ and $F^0W_0P(\C)\subseteq P(\C)$ which are torsors under $W_0U(\R)$, $W_0U(\C)$ and $F^0W_0U(\C)$ respectively, such that $W_0P(\R)\subseteq W_0P(\C)\supseteq F^0W_0P(\C)$.
\end{itemize}
\begin{proof}
This is a straightforward verification from the definitions. In the one direction, given subsets $W_0P(\R)$ and $F^0W_0P(\C)$ as above, the weight and Hodge filtrations on $\O(P)$ were constructed from choices of elements $q\in W_0P(\R)$ and $r\in F^0W_0P(\C)$ in such a way that $q\in W_0\O(P)^\dual$ and $r\in F^0W_0\O(P)^\dual_\C$. Hence $W_0P(\R)$ meets $\O(P)^{\dual,\gplike}\cap W_0$ non-trivially, and hence they are equal since they are both torsors under $W_0U(\R)$. By exactly the same argument, we also have $F^0W_0P(\C)=\O(P)^{\dual,\gplike}_\C\cap F^0\cap W_0$, so that $W_0P(\R)$ and $F^0W_0P(\C)$ are indeed the subsets in definition \ref{defn:reduction_of_structure_on_torsors} associated to the ind-mixed Hodge structure constructed by construction \ref{cons:MHS_from_reduction_of_structure}.

In the other direction, suppose that we are given an ind-mixed Hodge structure on $\O(P)$ as above. Since the map $\O(P)^\dual\hat\otimes\O(U)^\dual\isoarrow\O(P)^\dual\hat\otimes\O(P)^\dual$ given by $(q,u)\mapsto(q,qu)$ is a bijective morphism of pro-mixed Hodge structures\footnote{That is, pro-objects in the category of mixed Hodge structures.}, it is strict for the weight and Hodge filtrations, and hence both the action map $\O(P)^\dual\hat\otimes\O(U)^\dual\rightarrow\O(P)^\dual$ and the difference map $\O(P)^\dual\hat\otimes\O(P)^\dual\rightarrow\O(U)^\dual$ are morphisms of pro-mixed Hodge structures. In particular, if we choose any $q\in \O(P)^{\dual,\gplike}\cap W_0$ then the isomorphism $\O(P)\isoarrow\O(U)$ induced by $q$ \emph{strictly} preserves the weight filtration, so that the weight filtration on $\O(P)$ constructed by construction \ref{cons:MHS_from_reduction_of_structure} from the subset $\O(P)^\dual\cap W_0$ agrees with the original weight filtration. Exactly the same argument applies to the Hodge filtration, which completes the proof that the constructions are mutually inverse.
\end{proof}
\end{proposition}

\begin{corollary}\label{cor:explicit_H^1(MHS)}
Let $U/\R$ be a finitely generated pro-unipotent group with mixed Hodge structure. Then there is a canonical bijection\[\H^1(\MHS_\R,U(\R))\isoarrow W_0U(\R)\backslash W_0U(\C)/F^0W_0U(\C),\]where $\H^1(\MHS_\R,U(\R))$ denotes the set of isomorphism classes of torsors under $U$.

The same holds for $\R$ replaced with $\Q$ throughout.
\begin{proof}
The map sends a torsor $P$ under $U$ to the double coset containing the element $q^{-1}r\in U(\C)$ where $q\in W_0P(\R)$ and $r\in F^0P(\C)$. That this gives a bijection is left as an easy exercise to the reader.
\end{proof}
\end{corollary}

\begin{remark}
Corollary \ref{cor:explicit_H^1(MHS)} is a non-abelian analogue of the classical result that if $V$ is an $\R$-mixed Hodge structure then\[\Ext^1(\R(0),V)\cong\frac{W_0V_\C}{W_0V+F^0W_0V_\C}.\]See for instance \cite[Theorem 3.31]{peters-steenbrink} (with $\Z$ replaced by $\R$).
\end{remark}

It remains to prove that the subsets defined in definition \ref{defn:reduction_of_structure_on_torsors} are indeed subgroups and torsors as claimed, and that the weight and Hodge filtrations in construction \ref{cons:MHS_from_reduction_of_structure} do indeed define a mixed Hodge structure on the torsor $P$. The key technical input in the latter proof is a criterion which describes exactly when certain extensions of mixed Hodge structures are also mixed Hodge structures.

\begin{lemma}\label{lem:MHSs_in_exact_sequences}
Let\[0\rightarrow V''\rightarrow V\rightarrow V'\rightarrow0\]be an exact sequence of finite dimensional $\R$-vector spaces, each endowed with an increasing filtration $W_\bullet$, and with a decreasing filtration $F^\bullet$ on its complexification -- we assume both filtrations are separated and exhaustive. Suppose that the morphisms in the exact sequence preserve both filtrations, and that the filtrations define $\R$-mixed Hodge structures on $V''$ and $V'$ respectively. Then the following are equivalent:
\begin{enumerate}
	\item the filtrations define a mixed Hodge structure on $V$;
	\item the morphisms in the exact sequence are strict for the filtration $W_\bullet$, and the morphisms in the induced exact sequences\[0\rightarrow\gr^W_nV''_\C\rightarrow\gr^W_nV_\C\rightarrow\gr^W_nV'_\C\rightarrow0\]are strict for the filtration $F^\bullet$ (this implies that the original sequence was strict for the filtration $F^\bullet$ after tensoring with $\C$); and
	\item the exact sequence possesses a splitting after tensoring with $\C$ which is strict for both filtrations simultaneously.
\end{enumerate}

The same holds for $\R$ replaced with $\Q$ throughout.
\end{lemma}

\begin{remark}
A version of this lemma appears as \cite[Criterion 3.10]{peters-steenbrink}, where it is erroneously stated that a necessary and sufficient condition for $V$ to be a mixed Hodge structure is that the morphisms in the exact sequence be strict for the weight and Hodge filtrations separately. In fact, strictness for the filtrations is necessary but insufficient for $V$ to be a mixed Hodge structure.
\end{remark}

\begin{proof}[Proof of lemma \ref{lem:MHSs_in_exact_sequences}]
We will prove the lemma only over $\R$, the case of $\Q$ following by the same argument.

For the implication $(1)\Rightarrow(3)$, it is simplest to use the Deligne splitting \cite[Lemma-Definition 3.4]{peters-steenbrink}, which is a bigrading $V^{p,q}$ of $V_\C$ that simultaneously splits the weight and Hodge filtrations in that\[F^pV_\C=\bigoplus_{i\geq p}V^{i,j}\]and\[W_n V_\C=\bigoplus_{i+j\leq n}V^{i,j}.\]The Deligne splitting is functorial with respect to morphisms of mixed Hodge structures, and hence\[0\rightarrow V''_\C\rightarrow V_\C\rightarrow V'_\C\rightarrow0\]is an exact sequence of bigraded vector spaces with respect to the Deligne splittings; taking any splitting of the sequence compatible with the bigradings gives a splitting compatible with the filtrations.

For the implication $(3)\Rightarrow(2)$, the existence of the splitting ensures that the morphisms in the sequence\[0\rightarrow V''_\C\rightarrow V_\C\rightarrow V'_\C\rightarrow0\]are strict for the filtration $W_\bullet$, and hence so too are the morphisms in the original sequence before tensoring with $\C$. The splitting also ensures that\[0\rightarrow\gr^W_nV''_\C\rightarrow\gr^W_nV_\C\rightarrow\gr^W_nV'_\C\rightarrow0\]is split compatibly with the filtration $F^\bullet$, and hence the morphisms are strict for $F^\bullet$.

Finally, for the implication $(2)\Rightarrow(1)$, it suffices to prove the special case that $V=\gr^W_nV$, so that $V'$ and $V''$ are pure Hodge structures of weight $n$, and we want to prove that so too is $V$.

We define $V^{p,q}:=F^pV_\C\cap\bar{F^qV_\C}$ for $p+q=n$ (and similarly for $V'$ and $V''$), so that $V^{p,q}$ is canonically isomorphic to the kernel of the map $F^pV_\C\oplus\bar{F^qV_\C}\rightarrow V_\C$ given by the sum of the inclusions. Applying the snake lemma to the diagram
\begin{center}
\begin{tikzcd}
0 \arrow[r] & F^pV''_\C\oplus\bar{F^q V''_\C} \arrow[r]\arrow[d] & F^pV_\C\oplus\bar{F^q V_\C} \arrow[r]\arrow[d] & F^pV'_\C\oplus\bar{F^q V'_\C} \arrow[r]\arrow[d] & 0 \\
0 \arrow[r] & V''_\C \arrow[r] & V_\C \arrow[r] & V'_\C \arrow[r] & 0
\end{tikzcd}
\end{center}
we obtain a short exact sequence
\[
0\rightarrow (V'')^{p,q}\rightarrow V^{p,q}\rightarrow (V'')^{p,q}\rightarrow0,
\]
where the final $0$ comes from the fact that $F^pV''_\C+\bar{F^qV''_\C}=V''_\C$ since $V''$ is a pure Hodge structure of weight $n$. Since $V'_\C=\bigoplus_{p+q=n}(V')^{p,q}$ and similarly for $V''$, we deduce from the above exact sequences that $V_\C=\bigoplus_{p+q=n} V^{p,q}$ also, so that $V$ is a pure Hodge structure of weight $n$, as desired.
\end{proof}

\begin{proposition}\label{prop:construction_gives_MHS}
The weight and Hodge filtrations constructed in construction \ref{cons:MHS_from_reduction_of_structure} define a mixed Hodge structure on $P$ (i.e.\ an ind-mixed Hodge structure on $\O(P)$ compatible with the $\O(U)$-comodule algebra structure).
\begin{proof}
Compatibility with the $\O(U)$-comodule algebra structure morphisms is clear from the construction, since the isomorphisms $\O(P)\isoarrow\O(U)$ and $\O(P)_\C\isoarrow\O(U)_\C$ induced from our choices of $q$ and $r$ are isomorphisms of $\O(U)$-comodule algebras.

It remains to show that the weight and Hodge filtrations define an ind-mixed Hodge structure on $\O(P)$. To do this, note that by construction the induced weight and Hodge filtrations on $\gr^J_\bullet\O(P)$ agree with those on $\gr^J_\bullet\O(U)$ under the canonical isomorphism $\gr^J_\bullet\O(P)\isoarrow\gr^J_\bullet\O(U)$, and hence define an ind-mixed Hodge structure on $\gr^J_\bullet\O(P)$. We now show by induction on $n$ that each $J_n\O(P)$ is a mixed Hodge structure (starting from the trivial base case $n=-1$), using the exact sequences
\begin{equation}\label{eq:J-filtration_on_torsors}
0\rightarrow J_{n-1}\O(P)\rightarrow J_n\O(P)\rightarrow \gr^J_n\O(P)\rightarrow0.
\end{equation}
The right-hand term is a mixed Hodge structure, and our inductive hypothesis ensures that so too is the left-hand term. Now our choice of $q$ gives rise to an isomorphism $\O(P)_\C\isoarrow\O(U)_\C$ strictly compatible with the weight, Hodge and conilpotency filtrations, and hence the sequence \ref{eq:J-filtration_on_torsors} becomes isomorphic to the corresponding sequence
\begin{equation}\label{eq:J-filtration_on_groups}
0\rightarrow J_{n-1}\O(U)\rightarrow J_n\O(U)\rightarrow \gr^J_n\O(U)\rightarrow0
\end{equation}
after tensoring with $\C$, in a manner strictly compatible with the weight and Hodge filtrations. Since the sequence \ref{eq:J-filtration_on_groups} splits over $\C$ compatibly with both filtrations, so too does the sequence \ref{eq:J-filtration_on_torsors}, and hence $J_n\O(P)$ is a mixed Hodge structure by lemma \ref{lem:MHSs_in_exact_sequences}. This completes the construction of the mixed Hodge structure on $P$.
\end{proof}
\end{proposition}

\begin{proposition}\label{prop:filtrations_on_torsors}
Let $U/\R$ be a finitely generated pro-unipotent group with mixed Hodge structure and $P$ a torsor under $U$ (with mixed Hodge structure). Then the subsets $W_0U(\R)\subseteq U(\R)$, $W_0U(\C)\subseteq U(\C)$ and $F^0W_0U(\C)\subseteq U(\C)$ defined in definition \ref{defn:reduction_of_structure_on_torsors} are subgroups, and the subsets $W_0P(\R)$, $W_0P(\C)$ and $F^0W_0P(\C)$ are torsors under $W_0U(\R)$, $W_0U(\C)$ and $F^0W_0U(\C)$ respectively.

The same holds when $\R$ is replaced with $\Q$ throughout.
\begin{proof}
We will just deal with the case of $F^0W_0U(\C)$ and $F^0W_0P(\C)$, the other cases being simpler. To see that $F^0W_0U(\C)\subseteq U(\C)$ is a subgroup, we simply note that the group operations on $U(\C)=\O(U)_\C^{\dual,\gplike}$ are induced by the multiplication and antipode in the Hopf algebra $\O(U)_\C^\dual$, and that these operations preserve both the weight and Hodge filtrations, so that $F^0W_0U(\C)=\O(U)_\C^{\dual,\gplike}\cap F^0\cap W_0$ is a subgroup.

To show that $F^0W_0P(\C)$ is a torsor under $F^0W_0U(\C)$, we note firstly that the coalgebra isomorphism $\O(P)^\dual\hat\otimes\O(U)^\dual\isoarrow\O(P)^\dual\hat\otimes\O(P)^\dual$ induced by the isomorphism $P\times U\isoarrow P\times P$ given by $(q,u)\mapsto(q,qu)$ is a morphism of pro-mixed Hodge structures, and hence is strict for the Hodge and weight filtrations. Thus for $q\in F^0W_0P(\C)$ we have that $qu\in F^0W_0P(\C)$ if and only if $u\in F^0W_0U(\C)$, so that the action of $F^0W_0U(\C)$ on $P(\C)$ restricts to a transitive action on $F^0W_0P(\C)$.

It remains to show that $F^0W_0P(\C)$ is non-empty, i.e.\ that there is an algebra homomorphism $\O(P)_\C\rightarrow\C$ compatible with both the weight and Hodge filtrations. For $n\in\N_0$ let $U_n$ denote the maximal $n$-step unipotent quotient of $U$ and $P_n:=P\times^UU_n$ the pushout. It follows from lemma \ref{lem:J-filtration_is_descending_central_series} that $\O(P_n)_\C\leq\O(P)_\C$ is the subalgebra generated by $J_n\O(P)_\C$, and hence by proposition \ref{prop:J_n_is_sub-MHS} is a sub-ind-mixed Hodge structure. We claim that each inclusion $\O(P_{n-1})_\C\hookrightarrow\O(P_n)_\C$ admits a retraction which is an algebra homomorphism compatible with the weight and Hodge filtrations. This completes the proof of the proposition, since then the algebra inclusion $\C=\O(P_0)_\C\hookrightarrow\O(P)_\C=\bigcup_n\O(P_n)_\C$ also admits such a retraction, as desired.

To prove the claim, we choose by lemma \ref{lem:MHSs_in_exact_sequences}\footnote{In the analogous proof for $W_0P(\R)$ one uses instead the fact that an exact sequence of $\R$-mixed Hodge structures is split compatibly with the weight filtration, since all the morphisms are strict.} a retraction of the inclusion $\O(P_{n-1})_\C\cap J_n\O(P)_\C\hookrightarrow J_n\O(P)_\C$ which preserves both the weight and Hodge filtrations. The chosen retraction then induces a morphism\[\phi\colon\Sym^\bullet\left(J_n\O(P)_\C\right)\rightarrow\O(P_{n-1})_\C\]of algebras, compatible with the weight and Hodge filtrations. Now if $x\in J_i\O(P)$ and $y\in J_j\O(P)$, we have that $\phi([x][y])=\phi([x])\phi([y])=\phi([xy])$ -- if $i,j<n$, this is since by construction $\phi$ restricts to the inclusion $\O(P_{n-1})_\C\cap J_n\O(P)_\C\hookrightarrow\O(P_{n-1})_\C$, while if $i=0$ or $j=0$, this is since $J_0\O(P)_\C=\C$ and $\phi$ is $\C$-linear. We also trivially have that $\phi([1])=1=\phi(1)$.

It thus follows by lemma \ref{lem:J-filtration_is_descending_central_series} (and the fact that $\O(P)_\C\simeq\O(U)_\C$ as $J$-filtered algebras) that $\phi$ factors uniquely through an algebra homomorphism $\bar\phi\colon\O(P_n)_\C\rightarrow\O(P_{n-1})_\C$. Since the quotient map $\Sym^\bullet(J_n\O(P)_\C)\twoheadrightarrow\O(P_n)_\C$ is the complexification of a morphism of ind-mixed Hodge structures, it is strict for the weight and Hodge filtrations, and hence $\bar\phi$ also preserves these filtrations. Finally, we note by construction that $\bar\phi$ restricts to the canonical inclusion $J_{n-1}\O(P)_\C\hookrightarrow\O(P_{n-1})_\C$, and hence is a retraction of the inclusion $\O(P_{n-1})_\C\hookrightarrow\O(P_n)_\C$, as desired.
\end{proof}
\end{proposition}

\subsubsection{Cosimplicial groups and twists for mixed Hodge structures}

\begin{remark}\label{rmk:cosimplicial_MHS}
If $U$ is a finitely generated pro-unipotent group with mixed Hodge structure, then the description of $\H^1(\MHS_\R,U(\R))$ in corollary \ref{cor:explicit_H^1(MHS)} can be expressed in cosimplicial language, by considering the cosimplicial group $U_\MHS^\bullet(\R)$ cogenerated by the semi-cosimplicial group\[F^0W_0U(\C)\times W_0U(\R)\rightrightarrows W_0U(\C)\]where $d^0(u,w)=u$ and $d^1(u,w)=w$. In this language, we have a canonical identification of the cohomotopy as\[\pi^i(U_\MHS^\bullet(\R))=\begin{cases}F^0W_0U(\R)&\text{if $i=0$,}\\\H^1(\MHS_\R,U(\R))&\text{if $i=1$,}\\1&\text{if $i\geq2$ and $U$ abelian.}\end{cases}\]

Thus one is tempted to call an element $\alpha\in W_0U(\C)$ a \emph{cocycle} for $U$, and define the \emph{twist} ${}_\alpha U$ of $U$ by $\alpha$ to be the finitely generated pro-unipotent group with mixed Hodge structure whose underlying group scheme is $U$, whose weight filtration on $\O({}_\alpha U)=\O(U)$ is that of $U$, and whose Hodge filtration on $\O({}_\alpha U)_\C$ is given by $F^p\O({}_\alpha U)_\C=\Ad_\alpha^*F^p\O(U)_\C$, where $\Ad_\alpha\colon U_\C\isoarrow U_\C$ is (left\nobreakdash-)conjugation by $\alpha$. To check that this defines a mixed Hodge structure on ${}_\alpha U$, we note that $\Ad_\alpha^*$ is a Hopf algebra automorphism which acts as the identity on $\gr^J_\bullet(\O(U))$, so the induced weight and Hodge filtrations of $\gr^J_\bullet(\O({}_\alpha U))$ define a mixed Hodge structure on each graded piece. To show that the filtrations in fact define a mixed Hodge structure on each $J_n\O({}_\alpha U)$, one argues inductively as in construction \ref{cons:MHS_from_reduction_of_structure}, using the fact that $\Ad_\alpha^*$ sets up an isomorphism $\O({}_\alpha U)_\C\isoarrow\O(U)_\C$ compatible with the weight and Hodge filtrations simultaneously.

It follows from proposition  \ref{prop:pi^1_of_twists} that, as expected, there is a canonical bijection $\H^1(\MHS_\R,{}_\alpha U(\R))\isoarrow\H^1(\MHS_\R,U(\R))$ (induced from the map ${}_\alpha U(\C)\isoarrow U(\C)$ given by right-multiplication by $\alpha$).
\end{remark}

Using this cosimplicial description, we immediately deduce a result describing how the sets $\H^1(\MHS_\R,U(\R))$ behave in central extensions, generalising the long exact sequence on Ext-groups $\Ext^i(\R(0),V)$ for usual (abelian) mixed Hodge structures $V$. The non-abelian version does not, to the best of the author's knowledge, appear in the literature, though undoubtedly results of this kind are already known to experts.

\begin{corollary}\label{cor:exact_sequence_archimedean}
Let\[1\rightarrow Z\centarrow U\rightarrow Q\rightarrow1\]be a central extension of unipotent groups over $\R$ with mixed Hodge structure. Then there is a functorially assigned exact sequence
\begin{center}
\begin{tikzcd}[column sep=small]
1 \arrow[r] & F^0W_0Z(\R) \arrow[r,"\msf z"] & F^0W_0U(\R) \arrow[r] \arrow[d, phantom, ""{coordinate, name=Z}] & F^0W_0Q(\R) \arrow[dll,rounded corners,to path={ -- ([xshift=2ex]\tikztostart.east) |- (Z) [near end]\tikztonodes -| ([xshift=-2ex]\tikztotarget.west) -- (\tikztotarget)}] & \\
 & \H^1(\MHS_\R,Z(\R)) \arrow[r,"\acts"] & \H^1(\MHS_\R,U(\R)) \arrow[r] & \H^1(\MHS_\R,Q(\R)) \arrow[r] & 1
\end{tikzcd}
\end{center}
where the horizontal arrows are the induced maps.
\begin{proof}
Consider the sequence of cosimplicial groups\[1\rightarrow Z_\MHS^\bullet\centarrow U_\MHS^\bullet\rightarrow Q_\MHS^\bullet\rightarrow1,\]where $U_\MHS^\bullet$, $Q_\MHS^\bullet$ and $Z_\MHS^\bullet$ are the cosimplicial groups described in remark \ref{rmk:cosimplicial_MHS}. Identifying $W_0U(\R)$ (resp.\ $W_0U(\C)$, resp.\ $F^0W_0U(\C)$) with $W_0\Lie(U)$ (resp.\ $W_0\Lie(U)_\C$, resp.\ $F^0W_0\Lie(U)_\C$) via the logarithm map, we see that the above sequence is a degree-wise central extension. The long exact sequence on cohomotopy from theorem \ref{thm:long_exact_seqs} is then the desired long exact sequence.
\end{proof}
\end{corollary}

\subsubsection{Manifold structures on moduli of torsors}

Armed with the explicit description in corollary \ref{cor:explicit_H^1(MHS)}, we now come to the main point of this section: that the set $\H^1(\MHS_\R,U(\R))$ for a unipotent group $U/\R$ with mixed Hodge structure carries a canonical topology, namely the quotient topology from its description as a double-coset space. Most crucially, under mild conditions on $U$, it even has canonically the structure of a $C^\infty$ real manifold.

\begin{definition}\label{def:negative_weights}
Let $U/\R$ (resp.\ $U/\Q$) be a pro-unipotent group with mixed Hodge structure. We say that $U$ \emph{has only negative weights} just when:
\begin{itemize}
	\item $W_0\O(U)=\R$ (resp.\ $=\Q$); or equivalently
	\item $W_{-1}\Lie(U)=\Lie(U)$.
\end{itemize}
\end{definition}

\begin{remark}\label{rmk:only_negative_weights}
The condition that $U$ has only negative weights is automatically satisfied for the $\R$-pro-unipotent Betti fundamental group of any smooth connected $\C$-variety (for instance by the description in \cite[Section 5]{hain_MHS}). In particular, any finite-dimensional mixed-Hodge-theoretic quotient of such a $U$ also only has negative weights.
\end{remark}

It is this condition of negativity of weights which provides the $C^\infty$ real manifold structure on $\H^1(\MHS_\R,U(\R))$, as follows.

\begin{proposition}\label{prop:H^1(MHS)_is_manifold}
Let $U/\R$ be a unipotent group with mixed Hodge structure with only negative weights. Then the $C^\infty$ Lie group action of $F^0U(\C)\times U(\R)$ on $U(\C)$ by $x\cdot(u,w)=w^{-1}xu$ is free and proper, so that by the quotient manifold theorem \cite[Theorem 21.10]{lee} $\H^1(\MHS_\R,U(\R))\cong U(\R)\backslash U(\C)/F^0U(\C)$ is canonically a $C^\infty$ real manifold.
\begin{proof}
In the case that $U$ is abelian, the condition on the weights implies that $U(\R)$ and $F^0U(\C)$ meet trivially in $U(\C)$. The given action is the translation by the subgroup $F^0U(\C)\oplus U(\R)$, which is necessarily free and proper.

In general we may proceed inductively, writing $U$ as a central extension\[1\rightarrow Z\centarrow U\rightarrow Q\rightarrow1\]of unipotent groups with mixed Hodge structures, where we know the result for $Q$ and for $Z$. The action is free, since if some $\mbf u\in F^0U(\C)\times U(\R)$ had a fixed point in $U(\C)$, its image in $F^0Q(\C)\times Q(\R)$ would have a fixed point, and hence $\mbf u$ would lie in $F^0Z(\C)\times Z(\R)=F^0Z(\C)\oplus Z(\R)$. But $U(\C)\twoheadrightarrow Q(\C)$ is a $Z(\C)$-torsor by translation, so that the only way $\mbf u$ could have a non-trivial fixed point is if it were the identity already. This proves freeness of the action.

To prove properness, it is equivalent by \cite[Proposition 21.5]{lee} to prove that whenever we are given sequences $(x_i)$ in $U(\C)$ and $(\mbf u_i)$ in $F^0U(\C)\times U(\R)$ such that $(x_i)$ and $(x_i\cdot\mbf u_i)$ converge, there is a convergent subsequence of $(\mbf u_i)$. Given such sequences, our inductive hypothesis shows that the image of $(\mbf u_i)$ in $F^0Q(\C)\times Q(\R)$ has a convergent subsequence, so we may suppose without loss of generality that this image sequence converges. Since the map $F^0U(\C)\times U(\R)\twoheadrightarrow F^0Q(\C)\times Q(\R)$, being the map on $\R$-points of a morphism of affine spaces over $\R$, is continuously split, we may thus write $\mbf u_i=\mbf u_i'\mbf z_i$, where $(\mbf u_i')$ is a convergent sequence in $F^0U(\C)\times U(\R)$ and $\mbf z_i\in F^0Z(\C)\times Z(\R)$.

Now the sequences $(x_i')=(x_i\cdot\mbf u_i')$ and $(x_i'\cdot\mbf z_i)=(x_i\cdot\mbf u_i)$ both converge. But the action of $F^0Z(\C)\times Z(\R)$ on $U(\C)$ can be identified with the translation action of the abelian Lie subgroup $F^0Z(\C)\times Z(\R)$ on $U(\C)$ (up to an automorphism of $F^0Z(\C)\times Z(\R)$), and hence is proper, so that $\mbf z_i$ has a convergent subsequence. Hence $(\mbf u_i)=(\mbf u_i'\mbf z_i)$ has a convergent subsequence, as desired.
\end{proof}
\end{proposition}

Before we conclude this section, let us briefly analyse some of the functoriality properties of $\H^1(\MHS_\R,-)$, for use in our later proof of theorem \ref{thm:main_theorem_archimedean}.

\begin{lemma}\label{lem:maps_on_H^1_smooth}
If $U\rightarrow W$ is a morphism of unipotent groups over $\R$ with mixed Hodge structure, both with only negative weights, then the induced map\[\H^1(\MHS_\R,U(\R))\rightarrow\H^1(\MHS_\R,W(\R))\]is a $C^\infty$ morphism of $C^\infty$ manifolds. If $U\rightarrow W$ is injective (resp.\ surjective), then it is an immersion (resp.\ a submersion)\footnote{In fact, the map is even an injective immersion (resp.\ surjective submersion)}.
\begin{proof}
Since $U(\C)\rightarrow\H^1(\MHS_\R,U(\R))$ is a locally trivial torsor for $F^0U(\C)\times U(\R)$ (for example from the proof of \cite[Theorem 21.10]{lee}), the first point follows from the fact that $U(\C)\rightarrow W(\C)$ is $C^\infty$ (even algebraic). Now suppose that $f\colon U\twoheadrightarrow W$ is surjective. If $\alpha\in U(\C)$ is any point, then we have a commuting square
\begin{center}
\begin{tikzcd}
\H^1(\MHS_\R,{}_\alpha U(\R)) \arrow[r]\arrow[d,"\wr"] & \H^1(\MHS_\R,{}_{f(\alpha)}W(\R)) \arrow[d,"\wr"] \\
\H^1(\MHS_\R,U(\R)) \arrow[r] & \H^1(\MHS_\R,W(\R))
\end{tikzcd}
\end{center}
of $C^\infty$ manifolds, where ${}_\alpha U$ and ${}_{f(\alpha)}W$ are the twists described in remark \ref{rmk:cosimplicial_MHS}. Since the leftmost vertical map takes the distinguished point to $[\alpha]$, to prove that $\H^1(\MHS_\R,U(\R))\rightarrow\H^1(\MHS_\R,W(\R))$ is a submersion it suffices to prove (for general $U$, $W$) that it is surjective on tangent spaces at the distinguished point. But the map on tangent spaces is\[\Lie(U)_\C/(F^0+\Lie(U)_\R)\rightarrow\Lie(W)_\C/(F^0+\Lie(W)_\R),\]which is clearly surjective.

If instead $f$ were injective, by the same argument it would suffice to prove that the above map is injective. But we can identify the sides as the Yoneda Ext-groups $\Ext^1_{\MHS_\R}(\R,\Lie(U))$ (resp.\ $\Ext^1_{\MHS_\R}(\R,\Lie(W))$), so that the kernel of the map is a surjective image of $\Hom_{\MHS_\R}(\R,\Lie(W)/\Lie(U))=F^0W_0\left(\Lie(W)/\Lie(U)\right)_\R$. But since $\Lie(W)/\Lie(U)$ has only negative weights, this vector space vanishes, and hence the desired map was indeed injective.
\end{proof}
\end{lemma}

\begin{remark}
When discussing the topology on $\H^1(\MHS_\R,U(\R))$, we have focused on unipotent groups rather than finitely generated pro-unipotent ones, even though $\R$-pro-unipotent fundamental groups will in general be infinite-dimensional (though still finitely generated). The principal reason for this is that one can reduce the study of $\H^1(\MHS_\R,U(\R))$ for $U$ finitely generated pro-unipotent to the finite-dimensional case: if we write such a $U=\liminv U_n$ as an inverse limit of finite-dimensional quotients with compatible mixed Hodge structures, then the natural map\[\H^1(\MHS_\R,U(\R))\rightarrow\liminv\H^1(\MHS_\R,U_n(\R))\]is bijective (the same proof as lemma \ref{lem:inverse_limits_of_H^1_g} works, using the cosimplicial description in remark \ref{rmk:cosimplicial_MHS}). Hence if $U$ has only negative weights, $\H^1(\MHS_\R,U(\R))$ is canonically a pro-object in the category of $C^\infty$ manifolds (this structure is independent of the choice of $U_n$), whose underlying topology can be checked to agree with the quotient topology on $\H^1(\MHS_\R,U(\R))=U(\R)\backslash U(\C)/F^0U(\C)$ induced from the pro-analytic topology on $U(\C)$.
\end{remark}

\subsection{Higher Albanese manifolds and the $\R$-pro-unipotent Kummer map}
\label{s:Kummer_map_archimedean}

In the previous section, we examined $\Q$- and $\R$-unipotent groups with mixed Hodge structure, but one can equally well study an integral version of this theory, by studying finitely generated (discrete) groups $\Pi$ with a ind-mixed Hodge structure on the decompleted dual $\Z\Pi^{\wedge\reddual}:=\limdir\Hom\left(\Z\Pi/J^{n+1},\Z\right)$ compatible with the Hopf algebra operations, or equivalently a $\Q$-mixed Hodge structure on the $\Q$-Mal\u cev completion $\Pi_\Q:=\Spec(\Q\Pi^{\wedge\reddual})$ -- we call such a structure a \emph{mixed Hodge structure} on $\Pi$. Here, as usual, $J$ denotes the augmentation ideal of $\Z\Pi$.

If $P$ is a torsor under such a group, we define a \emph{mixed Hodge structure} on $P$ to be an ind-mixed Hodge structure on $\Z P^{\wedge\reddual}:=\limdir\Hom\left(\Z P/J^{n+1},\Z\right)$ compatible with the algebra operations and $\Z\Pi^{\wedge\reddual}$-comodule structure, where $J^{n+1}$ denotes the filtration of $\Z P$ arising from the conilpotency filtration on $\Z\Pi$ under any trivialisation $\Pi\isoarrow P$ of $\Pi$-torsors. We will refer to a $P$ endowed with such a structure as a \emph{torsor under $\Pi$}, where the mixed Hodge structure is taken as implied.

With these definitions, one can develop an integral version of the theory of the previous section, and again we obtain an explicit description of the classifying space for torsors.

\begin{proposition}\label{prop:explicit_H^1_integral}
Let $\Pi$ be a finitely generated nilpotent group with mixed Hodge structure. Then there is a canonical identification of the set $\H^1(\MHS_\Z,\Pi)$ of isomorphism classes of $\Pi$-torsors with the quotient of $\Pi(\Q)\times W_0\Pi(\C)$ by the action of $F^0W_0\Pi(\C)\times W_0\Pi(\Q)\times\Pi$ given by $(x,y)\cdot(u,v,w)=(\bar w^{-1}xv,v^{-1}yu)$, where $\bar w$ denotes the image of $w$ under the map $\Pi\rightarrow\Pi(\Q)$. Here $\Pi(\Q)$ (resp.\ $\Pi(\C)$) denotes the $\Q$-points (resp.\ $\C$-points) of the Mal\u cev completion of $\Pi$.

In particular, if $\Pi$ has only negative weights, then we have a canonical bijection\[\H^1(\MHS_\Z,\Pi)\cong\Pi\backslash\Pi(\C)/F^0\Pi(\C).\]
\begin{proof}
Let us first describe the classifying map\[\H^1(\MHS_\Z,\Pi)\rightarrow\left(\Pi(\Q)\times W_0\Pi(\C)\right)/\left(F^0W_0\Pi(\C)\times W_0\Pi(\Q)\times\Pi\right).\] Given a $\Pi$-torsor $P$ with mixed Hodge structure, we may choose by proposition \ref{prop:MHS_on_torsors} elements $q\in P$, $x\in\Pi_\Q(\Q)$ and $y\in W_0\Pi_\Q(\C)$ such that
\[
W_0P_\Q(\Q)=\bar qxW_0\Pi_\Q(\Q)
\text{ and }
F^0W_0P_\Q(\C)=\bar qxyF^0W_0\Pi_\Q(\C),
\]
where $\bar q\in P_\Q(\Q)$ denotes the image of $q$ under the canonical map $P\rightarrow P_\Q(\Q)$. We declare the image $[P]$ of $P$ under the classifying map to be the orbit of the pair $(x,y)$ under the specified action -- this is easily seen to be independent of choices.

It remains to check that the classifying map is bijective. It is certainly surjective, since for any $(x,y)\in\Pi(\Q)\times W_0\Pi(\C)$, the above construction with $P=\Pi$ (as a $\Pi$-torsor) and $q=1$ specifies a mixed Hodge structure on $P_\Q$ (and hence on $P$) by proposition \ref{prop:MHS_on_torsors}. For injectivity, suppose that we are given two $\Pi$-torsors $P$ and $P'$ with mixed Hodge structure, and choose elements $q,x,y$ and $q',x',y'$ as above. If $[P]=[P']$, then $(x',y')=(\bar w^{-1}xv,v^{-1}yu)$ for some $u\in F^0W_0\Pi(\C)$, $v\in W_0\Pi(\Q)$ and $w\in\Pi$, and we let $\phi\colon P\isoarrow P'$ be the unique isomorphism of $\Pi$-torsors taking $q$ to $q'w^{-1}$. It follows that
\[\phi\left(W_0P_\Q(\Q)\right)=\phi(\bar qx)W_0\Pi_\Q(\Q)=\bar q'x'v^{-1}W_0\Pi_\Q(\Q)=W_0P'_\Q(\Q)
\]
and similarly for $F^0W_0P_\Q(\C)$, so that $\phi$ is an isomorphism of $\Pi$-torsors with mixed Hodge structures. This proves injectivity, as desired.
\end{proof}
\end{proposition}

\begin{remark}
Proposition \ref{prop:explicit_H^1_integral} is a non-abelian generalisation of the explicit description of Ext-groups of torsion-free $\Z$-mixed Hodge structures, in particular that\[\Ext^1(\Z(0),V)\cong\frac{V_\Q\oplus W_0V_\C}{F^0W_0V_\C+W_0V_\Q+V_\Z}\]for any torsion-free $\Z$-mixed Hodge structure $V$, where $W_0V_\Q$ is embedded in $V_\Q\oplus W_0V_\C$ diagonally.\footnote{This description disagrees with the description in \cite[Theorem 3.31]{peters-steenbrink}, which erroneously describes the Ext-groups as $W_0V_\C/(F^0W_0V_\C+V)$. This latter group in fact classifies extensions of $\Z(0)$ by $V$ which admit integral splittings compatible with the weight filtration (which isn't automatic in general).}
\end{remark}

\begin{corollary}\label{cor:realification_of_H^1}
Let $\Pi$ be a finitely generated nilpotent group with mixed Hodge structure with only negative weights. Then\[\H^1(\MHS_\Z,\Pi)\cong \Pi\backslash\Pi(\C)/F^0\Pi(\C)\]is canonically a complex manifold, and the natural forgetful map $\H^1(\MHS_\Z,\Pi)\rightarrow\H^1(\MHS_\R,\Pi(\R))$ is a $C^\infty$ map of $C^\infty$ real manifolds.

Explicitly, the isomorphism $\H^1(\MHS_\Z,\Pi)\isoarrow\Pi\backslash\Pi(\C)/F^0\Pi(\C)$ sends the class of a torsor $P$ to $q^{-1}r\in\Pi(\C)$, where $q\in P$ and $r\in F^0\Pi(\C)$ are any choices of elements.
\begin{proof}
Let $\Pi^\tf=\Pi/\Pi_\tors$ denote the image of $\Pi$ under the natural map $\Pi\rightarrow\Pi(\Q)$. It is a discrete subgroup of the Lie group $\Pi(\R)$, so that the action of $\Pi^\tf\times F^0\Pi(\C)$ on $\Pi(\C)$ is free and proper by proposition \ref{prop:H^1(MHS)_is_manifold}. Hence\[\Pi\backslash\Pi(\C)/F^0\Pi(\C)=\Pi^\tf\backslash\Pi(\C)/F^0\Pi(\C)\]is canonically a complex manifold. The fact that the forgetful map is $C^\infty$ is clear, and the explicit description of the map follows from the proof of proposition \ref{prop:explicit_H^1_integral}.
\end{proof}
\end{corollary}

In fact, these classifying manifolds for torsors with integral mixed Hodge structures have already been studied in great detail by Hain and Zucker in \cite{hain-zucker} (see also \cite{hain}). When $\Pi_n$ is the $n$-step nilpotent quotient of the fundamental group of a smooth connected variety $Y/\C$ at a basepoint $y\in Y(\C)$ (which has only negative weights), Hain and Zucker define a series of \emph{higher Albanese manifolds}\[\Alb^n(Y):=\Pi_n\backslash\Pi_n(\C)/F^0\Pi_n(\C)\]and \emph{higher Albanese maps} $Y(\C)\rightarrow\Alb^n(Y)$, generalising the usual Albanese variety and map of $Y$. Let us explicitly make the link between our description and theirs.

\begin{proposition}\label{prop:higher_Albanese_is_Kummer}
Let $Y/\C$ be a smooth connected variety, and let $\Pi_n$ be the $n$-step nilpotent quotient of the fundamental group of $Y$ based at a point $y\in Y(\C)$, so that we have a canonical identification $\H^1(\MHS_\Z,\Pi_n)\cong\Pi_n\backslash\Pi_n(\C)/F^0\Pi_n(\C)=\Alb^n(Y)$. Then, under this identification, the non-abelian Kummer map\[Y(\C)\rightarrow\H^1(\MHS_\Z,\Pi_n)\](that is, the map assigning to each point $z\in Y(\C)$ the torsor of paths\footnote{I.e.\ set of homotopy classes of paths, viewed as a torsor under the fundamental group.} from $y$ to $z$ in $Y(\C)$, pushed out along $\pi_1^\topp(Y(\C);y)\twoheadrightarrow\Pi_n$ and endowed with the mixed Hodge structure given in \cite[Section 3]{hain-zucker}) is identified with the higher Albanese map $Y(\C)\rightarrow\Alb^n(Y)$.

In particular, the non-abelian Kummer map is complex analytic, and the $\R$-unipotent non-abelian Kummer map\[Y(\C)\rightarrow\H^1(\MHS_\R,\Pi_n(\R))\](the composite of the above map with the map $\H^1(\MHS_\Z,\Pi_n)\twoheadrightarrow\H^1(\MHS_\R,\Pi_n(\R))$) is $C^\infty$.
\begin{proof}
Let us recall first Hain and Zucker's construction of the higher Albanese map. We let $\mathbb V_n$ denote the $n$th canonical variation on $Y$ with basepoint $y$. It is a variation of mixed Hodge structure, whose fibre at $z$ is canonically identified with $\Z\pi_1(Y;y,z)/J^{n+1}$ with its mixed Hodge structure. The corresponding holomorphic vector bundle $V_n=\O_Y\otimes_\Z\mathbb V_n$ on $Y(\C)$ then carries a flat holomorphic connection $\nabla_\C$ and Hodge filtration by holomorphic subbundles. Hain and Zucker prove in \cite[Proposition 5.24]{hain-zucker} that there is a $C^\infty$ frame for $V_n$, compatible with the weight and Hodge filtrations, such that for any path $\gamma$ from $y$ to any $z$, the parallel transport of the frame for $z^*V_n$ to $y^*V_n=\C\Pi_n/J^{n+1}$ along $\gamma^{-1}$ is given by left-multiplying\footnote{In \cite[Proposition 5.24]{hain-zucker}, the action is on the right, since their conventions for path-composition are opposite to ours.} the frame for $y^*V_n$ by some $\tau\in\Pi_n(\C)$. The image of $z$ under the higher Albanese map is then declared to be the double-coset of $\tau$.

Let now $P_n$ denote the torsor of paths from $y$ to $z$, pushed out to $\Pi_n$, so that there is a canonical identification $z^*V_n=\C P_n/J^{n+1}$. It is easy to see that the parallel transport $\C P_n/J^{n+1}\isoarrow\C\Pi_n/J^{n+1}$ along $\gamma^{-1}$ is just given by left-multiplication by $[\gamma]^{-1}$. It follows from the preceding paragraph that the identification $\C\Pi_n/J^{n+1}\isoarrow\C P_n/J^{n+1}$ arising from the $C^\infty$ frame for $V_n$ is given by left-multiplication by some $r\in P_n(\C)$ -- the $\C$-points of the Mal\u cev completion of $P_n$. Then $\tau=[\gamma]^{-1}r$, with $[\gamma]\in P_n(\Z)$ and $r\in F^0P_n(\C)$ (since the $C^\infty$ frame for $V_n$ respects the Hodge filtration), and hence by the description in corollary \ref{cor:realification_of_H^1} the class of the torsor $P_n$ is also represented by $\tau$.

That the non-abelian Kummer map is complex analytic follows from the corresponding assertion for the higher Albanese map \cite[Section 5]{hain-zucker}, and that the $\R$-unipotent non-abelian Kummer map is $C^\infty$ follows from corollary \ref{cor:realification_of_H^1}.
\end{proof}
\end{proposition}

\subsection{Proof of the main theorem}

Having set up all this abstract theory and made the link to Hain and Zucker's, we are now almost ready to prove theorem \ref{thm:main_theorem_archimedean} in an exactly analogous way to the proofs of theorems \ref{thm:main_theorem_l-adic} and \ref{thm:main_theorem_p-adic}. The only part remaining before we can copy our earlier proof methods is a characterisation of the archimedean N\'eron log-metric, analogous to the non-archimedean characterisation in lemma \ref{lem:neron_log-metrics}.

\begin{lemma}\label{lem:archimedean_neron_log-metrics}
Let $A/\C$ be an abelian variety. Then the N\'eron log-metrics\[\lambda_L\colon L^\times(\C)\rightarrow\R\]associated to pairs $(L,\tilde0)$ consisting of a line bundle $L/A$ and a basepoint $\tilde0\in L^\times(\C)$ lying over $0\in A(\C)$,are uniquely characterised by the following properties:
\begin{enumerate}\setcounter{enumi}{-1}
	\item $\lambda_L$ only depends on $(L,\tilde0)$ up to isomorphism;
	\item\label{condn:archimedean_local_boundedness} $\lambda_L$ is locally bounded (for the analytic topology on $L^\times(\C)$);
	\item\label{condn:archimedean_additivity} $\lambda_{L_1\otimes L_2}(P_1\otimes P_2)=\lambda_{L_1}(P_1)+\lambda_{L_2}(P_2)$;
	\item\label{condn:archimedean_pullbacks} $\lambda_{[2]^* L}(P)=\lambda_L([2]P)$, where by abuse of notation we also denote by $[2]$ the topmost map in the pullback square
	\begin{center}
	\begin{tikzcd}
	{[2]}^*L^\times \arrow[r]\arrow[d] & L^\times \arrow[d] \\
	A \arrow[r,"{[2]}"] & A;
	\end{tikzcd}
	\end{center}
	\item\label{condn:archimedean_trivial_bundle} when $L=A\times_\C\A^1$ is the trivial line bundle, $\lambda_L$ is the composite\[L^\times(\C)=A(\C)\times\C^\times\rightarrow\C^\times\overset v\rightarrow\R\]where $v=-\log|\cdot|$; and
	\item\label{condn:archimedean_normalisation} $\lambda_L(\tilde 0)=0$.
\end{enumerate}
In fact, the N\'eron log-metrics are even $C^\infty$, not just locally bounded.
\begin{proof}
That the N\'eron log-metric satisfies the above properties follows from the corresponding properties of N\'eron functions \cite[Theorem 11.1.1]{lang}.

That the N\'eron log-metric is uniquely characterised by these properties proceeds in exactly the same way as in the proof of lemma \ref{lem:neron_log-metrics}, with the small adjustment that the function $\delta_L$ is locally bounded with bounded image instead of being locally constant with finite image.
\end{proof}
\end{lemma}
\vspace{-0.5ex}
With this characterisation of the N\'eron log-metric, theorem \ref{thm:main_theorem_archimedean} can now be proved by a more-or-less direct translation of the proofs in \S\ref{c:main_theorem_l-adic}. Specifically, given a line bundle $L$ on an abelian variety $A/\C$, the fibration $\G_m\rightarrow L^\times\rightarrow A$ gives us a central extension\[1\rightarrow\R(1)\centarrow U\rightarrow VA\rightarrow1\]of $\R$-unipotent fundamental groups, where $VA=\H_1(A(\C),\R)$, and hence by corollary \ref{cor:exact_sequence_archimedean} an exact sequence\[F^0W_0VA\rightarrow\H^1(\MHS_\R,\R(1))\actsarrow\H^1(\MHS_\R,U(\R))\rightarrow\H^1(\MHS_\R,VA).\]

But $VA$, being pure of weight $-1$, satisfies $F^0W_0VA=0$ and $\H^1(\MHS_\R,VA)=VA_\C/F^0VA_\C\oplus VA=0$, and hence the map $\H^1(\MHS_\R,\R(1))\rightarrow\H^1(\MHS_\R,U(\R))$ is bijective. Indeed, it is a $C^\infty$ diffeomorphism, since it is an immersion by lemma \ref{lem:maps_on_H^1_smooth} and hence a local $C^\infty$ diffeomorphism by counting dimensions.

Now we also have a canonical $C^\infty$ diffeomorphism $\H^1(\MHS_\R,\R(1))\isoarrow\R(1)\backslash\C\cong\R$ (the final isomorphism being minus the obvious one) and so we define a map $\lambda_L$ as the composite\[\lambda_L\colon L^\times(\C)\rightarrow\H^1(\MHS_\R,U(\R))\leftisoarrow\H^1(\MHS_\R,\R(1))\isoarrow\R.\]It remains to show that the $\lambda_L$ are the N\'eron log-metrics, by verifying the properties in lemma \ref{lem:archimedean_neron_log-metrics}.

We've already checked that the $\lambda_L$ are continuous (in fact $C^\infty$), and the verification of the remaining properties is purely formal, following exactly the pattern in the proof of theorem \ref{thm:main_theorem_l-adic} (with $\H^1(\MHS_\R,-)$ replacing the non-abelian Galois cohomology). The only thing we need to check is that we have normalised the isomorphism $\H^1(\MHS_\R,\R(1))\isoarrow\R$ correctly, for the proof of point \ref{condn:archimedean_normalisation}. But this follows immediately from the fact that the abelian Kummer map $\C^\times\rightarrow\H^1(\MHS_\Z,\Z(1))\cong\Z(1)\backslash\C$ is the logarithm map, so that the composite $\C^\times\rightarrow\H^1(\MHS_\Z,\Z(1))\rightarrow\H^1(\MHS_\R,\R(1))\isoarrow\R$ is given by $z\mapsto-\Re\log(z)=-\log|z|$, which is the desired normalisation.


\bibliography{references}

\begin{thebibliography}{BDCKW14}

\bibitem[BD16]{balakrishnan-dogra_1}
J.S. Balakrishnan and N.~Dogra.
\newblock Quadratic {C}habauty and rational points {I}: $p$-adic heights, 2016.
\newblock to appear in Duke Mathematical Journal (preprint available as arXiv:
  1601.00388 -- v1 accessed).

\bibitem[BD17]{balakrishnan-dogra_2}
J.S. Balakrishnan and N.~Dogra.
\newblock Quadratic {C}habauty and rational points {II}: generalised height
  functions on {S}elmer varieties, 2017.
\newblock arXiv: 1705.00401 (preprint only) -- v1 accessed.

\bibitem[BDCKW14]{minhyong-etal:bsd_conjecture}
J.S. Balakrishnan, I.~Dan-Cohen, M.~Kim, and S.~Wewers.
\newblock A non-abelian conjecture of {Birch} and {Swinnerton-Dyer} type for
  hyperbolic curves, 2014.
\newblock arXiv: 1209.0640 (preprint only) -- v3 accessed.

\bibitem[Ber02]{berger}
L.~Berger.
\newblock Repr{\'e}sentations $p$-adiques et {\'e}quations diff{\'e}rentielles.
\newblock {\em Inventiones Mathematicae}, 148(2):219--284, 2002.

\bibitem[Bet18]{thesis}
L.A. Betts.
\newblock {\em Heights via anabelian geometry and local Bloch--Kato Selmer
  sets}.
\newblock PhD thesis, University of Oxford, 2018.

\bibitem[Bet22]{me:local_constancy}
L.A. Betts.
\newblock Local constancy of pro-unipotent {K}ummer maps, 2022.
\newblock arXiv:2203.03701 (v1 accessed).

\bibitem[BG06]{bombieri-gubler}
E.~Bombieri and W.~Gubler.
\newblock {\em {Heights} in {Diophantine} {Geometry}}.
\newblock Number~4 in New Mathematical Monographs. Cambridge University Press,
  2006.

\bibitem[BK72]{bousfield-kan}
A.K. Bousfield and D.M. Kan.
\newblock {\em {Homotopy} {Limits}, {Completions} and {Localizations}}.
\newblock Number 304 in Lecture Notes in Mathematics. Springer-Verlag, 1972.

\bibitem[BK90]{bloch-kato}
S.~Bloch and K.~Kato.
\newblock {$L$-Functions} and {Tamagawa} {Numbers} of {Motives}.
\newblock In {\em The Grothendieck Festschrift, Volume I}, volume~88 of {\em
  Modern Birkh{\"a}user Classics}, pages 333--400. Birkh{\"a}user Verlag, 1990.

\bibitem[BL19]{me-daniel:weight-monodromy}
L.A. Betts and D.~Litt.
\newblock Semisimplicity of the {F}robenius action on $\pi_1$.
\newblock 2019.
\newblock arXiv:1912.02167 (v2 accessed).

\bibitem[Bre83]{breen}
L.~Breen.
\newblock {\em Fonctions th{\^e}ta et th{\'e}or{\`e}me du cube}.
\newblock Number 980 in Lecture Notes in Mathematics. Springer-Verlag, 1983.

\bibitem[CC04]{castiglioni-cortinas}
J.L. Castiglioni and G.~Corti{\~n}as.
\newblock Cosimplicial versus {DG-rings}: a version of the {Dold--Kan}
  correspondence.
\newblock {\em Journal of Pure and Applied Algebra}, 191:119--142, 2004.

\bibitem[CF00]{colmez-fontaine}
P.~Colmez and J-M. Fontaine.
\newblock Construction des repr{\'e}sentations $p$-adiques semi-stables.
\newblock {\em Inventiones Mathematicae}, 140(1):1--43, 2000.

\bibitem[CK10]{minhyong-coates}
J.~Coates and M.~Kim.
\newblock {S}elmer varieties for curves with {CM} {J}acobians.
\newblock {\em Kyoto Journal of Mathematics}, 50(4):827--852, 2010.

\bibitem[Col85]{coleman}
R.F. Coleman.
\newblock Effective {C}habauty.
\newblock {\em Duke Mathematical Journal}, 52(3):765--770, 1985.

\bibitem[Cus02]{cushman}
M.~Cushman.
\newblock {\em Morphisms of curves and the fundamental group}, volume~5 of {\em
  Nankai {T}racts in {M}athematics}, pages 23--38.
\newblock 2002.

\bibitem[Del89]{deligne}
P.~Deligne.
\newblock Le groupe fondamental de la droite projective moins trois points.
\newblock In {\em Galois Groups over $\mathbb Q$}, volume~16 of {\em
  Mathematical Sciences Research Institute Publications}, pages 79--297.
  Springer-Verlag, 1989.

\bibitem[DM00]{duflot-marak}
J.~Duflot and C.T. Marak.
\newblock A filtration in algebraic {K}-theory.
\newblock {\em Journal of Pure and Applied Algebra}, 151:135--162, 2000.

\bibitem[DN18]{deglise-niziol}
F.~D{\'e}glise and W.~Nizio{\l}.
\newblock On $p$-adic absolute {H}odge cohomology and syntomic coefficients, i.
\newblock {\em Commentarii Mathematici Helvetici}, 93(1):71--131, 2018.

\bibitem[EH17]{ellenberg-hast}
J.S. Ellenberg and D.R. Hast.
\newblock Rational points on solvable curves over {$\Q$} via non-abelian
  {C}habauty, 2017.
\newblock arXiv: 1706.00525 (preprint only) -- v1 accessed.

\bibitem[Fal83]{faltings}
G.~Faltings.
\newblock Endlichkeitss{\"a}tze f{\"u}r abelsche {V}ariet{\"a}ten {\"u}ber
  {Z}ahlk{\"o}rpern.
\newblock {\em Inventiones mathematicae}, 73:349--366, 1983.

\bibitem[FG]{javier}
J.~Fres\'an and J.I.~Burgos Gil.
\newblock {\em Multiple zeta values: from numbers to motives}.
\newblock Clay Mathematics Proceedings.
\newblock to appear.

\bibitem[Fon94a]{fontaine2}
J-M. Fontaine.
\newblock Le corps des p{\'e}riodes $p$-adiques.
\newblock {\em Ast{\'e}risque}, 223: P{\'e}riodes $p$-adiques:59--101, 1994.

\bibitem[Fon94b]{fontaine3}
J-M. Fontaine.
\newblock Repr{\'e}sentations $p$-adiques semi-stables.
\newblock {\em Ast{\'e}risque}, 223: P{\'e}riodes $p$-adiques:113--184, 1994.

\bibitem[GJ09]{goerss-jardine}
P.G. Goerss and J.F. Jardine.
\newblock {\em {Simplicial} {Homotopy} {Theory}}.
\newblock Modern Birkh{\"a}user Classics. Birkh{\"a}user Verlag, 2009.
\newblock Originally published as volume 174 of Progress in Mathematics, 1999.

\bibitem[GM04]{grunenfelder-mastnak}
L.~Grunenfelder and M.~Mastnak.
\newblock Cohomology of abelian matched pairs and the {Kac} sequence.
\newblock {\em Journal of Algebra}, 276:706--736, 2004.

\bibitem[GRL67]{SGA7.1}
A.~Grothendieck, M.~Raynaud, and D.S. Lim.
\newblock {\em Groupes de Monodromie en G{\'e}ometrie Alg{\'e}brique ({SGA
  7.1})}.
\newblock Number 288 in Lecture Notes in Mathematics. Springer-Verlag, 1967.

\bibitem[Gro97]{grothendieck_letter}
A.~Grothendieck.
\newblock Brief an {G}. {F}altings (27th {J}une 1983).
\newblock In {\em Geometric Galois Actions 1: Around Grothendieck's Esquisse
  d'un Programme}, number 242 in London Mathematical Society Lecture Note
  Series, pages 49--58. Cambridge University Press, 1997.

\bibitem[Hai87a]{hain_MHS}
R.~Hain.
\newblock The geometry of the mixed {H}odge structure on the fundamental group.
\newblock In {\em Algebraic Geometry--Bowdoin 1985, Part 2}, volume~46 of {\em
  Proceedings of Symposia in Mathematics}, pages 247--282. American
  Mathematical Society, 1987.

\bibitem[Hai87b]{hain}
R.~Hain.
\newblock Higher {A}lbanese manifolds.
\newblock In {\em Hodge Theory}, number 1246 in Lecture Notes in Mathematics,
  pages 84--91. Springer-Verlag, 1987.

\bibitem[Hir75]{hironaka}
H.~Hironaka.
\newblock Triangulations of algebraic sets.
\newblock In {\em Algebraic Geometry -- {A}rcata 1974}, volume~29 of {\em
  Proceedings of Symposia in Pure Mathematics}, pages 165--185. American
  Mathematical Society, 1975.

\bibitem[HK94]{hyodo-kato}
O.~Hyodo and K.~Kato.
\newblock Semi-stable reduction and crystalline cohomology with logarithmic
  poles.
\newblock {\em Ast{\'e}risque}, 223: P{\'e}riodes $p$-adiques:221--268, 1994.

\bibitem[HK04]{kim-hain}
R.~Hain and M.~Kim.
\newblock A {D}e {R}ham-{W}itt approach to crystalline rational homotopy
  theory.
\newblock {\em Compositio Mathematicae}, 140(5):1245--1276, 2004.

\bibitem[Hyo]{hyodo}
O.~Hyodo.
\newblock {$\H^1_g=\H^1_\st$}.
\newblock Unpublished manuscript.

\bibitem[HZ87]{hain-zucker}
R.~Hain and S.~Zucker.
\newblock Unipotent variations of mixed {H}odge structure.
\newblock {\em Inventiones Mathematicae}, 88:83--124, 1987.

\bibitem[Kat89]{kato_log-structures}
K.~Kato.
\newblock Logarithmic structures of {F}ontaine--{I}llusie.
\newblock In {\em Algebraic analysis, geometry, and number theory}, pages
  191--224. The John Hopkins University Press, 1989.

\bibitem[Kim05]{minhyong:siegel}
M.~Kim.
\newblock The motivic fundamental group of {$\mathbf
  P^1\setminus\{0,1,\infty\}$} and the theorem of {S}iegel.
\newblock {\em Inventiones Mathematicae}, 161(3):629--656, 2005.

\bibitem[Kim09]{minhyong:selmer}
M.~Kim.
\newblock The unipotent {A}lbanese map and {S}elmer varieties for curves.
\newblock {\em Publications of the Research Institute for Mathematical
  Sciences}, 45:89--133, 2009.

\bibitem[Kim10]{minhyong:cm}
M.~Kim.
\newblock $p$-adic {$L$}-functions and {S}elmer varieties associated to
  elliptic curves with complex multiplication.
\newblock {\em Annals of Mathematics}, 172(1):751--759, 2010.

\bibitem[KPT08]{katzarkov-pantev-toen}
L.~Katzarkov, T.~Pantev, and B.~To{\"e}n.
\newblock Schematic homotopy types and non-abelian {Hodge} theory.
\newblock {\em Compositio Mathematica}, 144(3):582--632, 2008.

\bibitem[Kri93]{kriz}
I.~Kriz.
\newblock $p$-adic homotopy theory.
\newblock {\em Topology and its Applications}, 52(3):279--308, 1993.

\bibitem[Lan63]{maclane}
S.~Mac Lane.
\newblock {\em Homology}.
\newblock Number 114 in Grundlehren der {M}athematischen {W}issenschaften.
  Springer-Verlag, 1963.

\bibitem[Lan83]{lang}
S.~Lang.
\newblock {\em Fundamentals of {D}iophantine Geometry}.
\newblock Springer-Verlag, 1983.

\bibitem[Lee03]{lee}
J.M. Lee.
\newblock {\em Introduction to Smooth Manifolds, Second Edition}.
\newblock Number 218 in Graduate Texts in Mathematics. Springer-Verlag, 2003.

\bibitem[LM85]{lubotzky-magid}
A.~Lubotzky and A.~Magid.
\newblock Cohomology, {P}oincar{\'e} series, and group algebras of unipotent
  groups.
\newblock {\em Americal Journal of Mathematics}, 107(3):531--553, 1985.

\bibitem[Mil17]{milne}
J.S. Milne.
\newblock {\em Algebraic Groups}.
\newblock Number 170 in Cambridge studies in advanced mathematics. Cambridge
  University Press, 2017.

\bibitem[MT83]{mazur-tate:circle_pairing}
B.~Mazur and J.~Tate.
\newblock Canonical height pairings via biextensions.
\newblock In {\em Arithmetic and Geometry}, volume~35 of {\em Progress in
  Mathematics}, pages 195--237. Birkh{\"a}user Verlag, 1983.

\bibitem[MT87]{mazur-tate:bsd}
B.~Mazur and J.~Tate.
\newblock Refined conjectures of the ``{Birch} and {Swinnerton-Dyer}'' type.
\newblock {\em Duke Mathematical Journal}, 54(2):711--750, 1987.

\bibitem[Nek93]{nekovar}
J.~Nekov{\'a}{\u r}.
\newblock On $p$-adic height pairings.
\newblock In {\em S{\'e}minaire de {T}h{\'e}orie des {N}ombres, {P}aris,
  1990--91}, number 108 in Progress in Mathematics, pages 127--202.
  Birkh{\"a}user, 1993.

\bibitem[Nek00]{nekovar_abel-jacobi}
J.~Nekov{\'a\u r}.
\newblock $p$-adic {A}bel--{J}acobi maps and $p$-adic heights.
\newblock In {\em The Arithmetic and Geometry of Algebraic Cycles}, volume~24
  of {\em CRM Proceedings and Lecture Notes}, pages 367--379. American
  Mathematical Society, 2000.

\bibitem[Ols11]{olsson}
M.C. Olsson.
\newblock Towards non-abelian $p$-adic {Hodge} theory in the good reduction
  case.
\newblock {\em Memoirs of the American Mathematical Society}, 210(990), 2011.

\bibitem[Pir14]{pirashvili}
M.~Pirashvili.
\newblock Second cohomotopy and nonabelian cohomology.
\newblock {\em Journal of K-Theory}, 13:397--445, 2014.

\bibitem[PS08]{peters-steenbrink}
C.A.M. Peters and J.H.M. Steenbrink.
\newblock {\em Mixed {H}odge structures}, volume~52 of {\em Ergebnisse der
  Mathematik und ihrer Grenzgebiete. 3. Folge / A Series of Modern Surveys in
  Mathematics}.
\newblock Springer-Verlag, 2008.

\bibitem[Sak17]{sakugawa}
K.~Sakugawa.
\newblock On a non-abelian generalization of the {Bloch-Kato} exponential map.
\newblock {\em Mathematical Journal of Okayama University}, 59:41--70, 2017.

\bibitem[Sch94]{scholl}
A.~Scholl.
\newblock Height pairings and special values of {$L$}-functions.
\newblock In {\em Motives}, number 55.2 in Proceedings in Symposia on Pure
  Mathematics, pages 571--598. American Mathematical Society, 1994.

\bibitem[Ser97]{serre}
J.P. Serre.
\newblock {\em {Galois} {Cohomology}}.
\newblock Springer Monographs in Mathematics. Springer-Verlag, 1997.
\newblock Translated from the original French by P. Ion.

\bibitem[Shi00a]{shiho_p-adic_Hodge}
A.~Shiho.
\newblock Crystalline fundamental groups and $p$-adic {H}odge theory.
\newblock In {\em The Arithmetic and Geometry of Algebraic Cycles}, volume~24
  of {\em CRM Proceedings and Lecture Notes}, pages 381--398. American
  Mathematical Society, 2000.

\bibitem[Shi00b]{shiho_1}
A.~Shiho.
\newblock Crystalline fundamental groups {I} -- isocrystals on log crystalline
  site and log convergent site.
\newblock {\em Journal of Mathematical Sciences, the University of Tokyo},
  7:509--656, 2000.

\bibitem[Sti12]{stix}
J.~Stix.
\newblock On cuspidal sections of algebraic fundamental groups.
\newblock In {\em Galois--Teichm{\"u}ller Theory and Arithmetic Geometry,
  Proceedings for a Conference in Kyoto (October 2010)}, number~63 in Studies
  in Pure Mathematics, pages 519--563. American Mathematical Society, 2012.

\end{thebibliography}
\bibliographystyle{alpha}

\end{document}